\documentclass[11pt]{article}
\usepackage{amssymb}
\usepackage{amssymb}
\usepackage{amsfonts}
\usepackage{bm}
\usepackage{xcolor}
\usepackage{hyperref}
\hypersetup{
  colorlinks=false,
  linkbordercolor=red,linkcolor=green,pdfborderstyle={/S/U/W 1}
} 

\usepackage{tikz}
\usepackage{tikz-3dplot}
\usetikzlibrary{shapes.geometric, arrows,3d,positioning}
\usetikzlibrary{positioning,arrows.meta}

\usepackage{amssymb}
\usepackage[T1]{fontenc}
\usepackage{latexsym,amssymb,amsmath,amsfonts,amsthm}
\usepackage{graphics}
\usepackage{graphicx}
\usepackage{mathrsfs}
\usepackage{subfigure}
\usepackage{float}
\newcommand{\R}{{\mat R}}

\newcommand{\Bz}{\mathbf{B}}

\newcommand{\Z}{{\mat Z}}

\newcommand{\N}{{\mat N}}

\newcommand{\bR}{{\bm{ R}}}
\newcommand{\C}{{\mat C}}

\newcommand{\Cz}{{\mathbf{ C}}}
\newcommand{\bL}{{\mathbf{L}}}

\newcommand{\ben}{\begin{eqnarray*}}
\newcommand{\en}{\end{eqnarray}}
\newcommand{\enn}{\end{eqnarray*}}

\newcommand{\bF}{\mathbf{F}}
\newcommand{\bS}{\mathbf{S}}

\newcommand{\J}{\mathcal{J}}
\renewcommand{\L}{\mathcal{L}}

\renewcommand{\i}{{\rm i}}

\newcommand{\bgamma}{{\boldsymbol{\gamma}}}

\newcommand{\bomega}{{\bm \omega}}

\renewcommand{\d}{{\rm d}}
\newcommand{\e}{{\bm e}}
\renewcommand{\v}{{\bm v}}

\newcommand{\mat}{\mathbb}

\renewcommand{\t}{{\bm t}}
\newcommand{\n}{{\bm n}}
\renewcommand{\v}{{\bm v}}
\newcommand{\bx}{{\bm x}}
\newcommand{\by}{{\bm y}}
\newcommand{\bz}{{\bm z}}
\newcommand{\bj}{{\bm j}}

\newcommand{\balpha}{{\bm \alpha}}
\newcommand{\bbeta}{{\bm \beta}}

\renewcommand{\epsilon}{\varepsilon}
\newtheorem{theorem}{Theorem}[section]
\newtheorem{lemma}[theorem]{Lemma}
\newtheorem{corollary}[theorem]{Corollary}
\newtheorem{definition}[theorem]{Definition}
\newtheorem{remark}[theorem]{Remark}

\newtheorem{proposition}[theorem]{Proposition}
\newtheorem{assumption}[theorem]{Assumption}

\usepackage{color}

\begin{document}
\renewcommand{\theequation}{\arabic{section}.\arabic{equation}}

\begin{titlepage}
\title{\bf Sliced Spectral Analysis and Geometric Mechanisms of Radiation for Periodic Elliptic Operators}
\author{Ruming Zhang\thanks{Institute for Mathematics, Technische Universit{\"a}t Berlin, Berlin, Germany; \texttt{ruming.zhang@tu-berlin.de}}
}
\end{titlepage}
\maketitle

\begin{abstract}
Radiation in higher-dimensional periodic media is fundamentally more difficult than in one dimension because the multidimensional spectral parameterization no longer admits the one-dimensional analytic structure underlying classical complex-analytic methods. We introduce a sliced spectral analysis (SSA), which incorporates the observation direction as an explicit parameter and decomposes the Brillouin-zone integral into an outer integration over spectral slices and an inner one-dimensional spectral problem on each directional slice. This restores, on every slice, the analytic structure of the one-dimensional theory and thereby provides a unified analytic framework for radiation in arbitrary dimensions.

Making the observation direction explicit has a second, independent consequence. It makes direction-dependent spectral geometry accessible and, in particular, enables the definition of a grazing set, a genuinely higher-dimensional geometric object separating the real and complex Fermi surfaces and characterizing the transition between propagating and evanescent modes. This geometric structure naturally leads to a decomposition of the limiting absorption solution into evanescent, non-grazing, and grazing components, each associated with a distinct radiation mechanism.

Finally, we establish the necessity of the three-component decomposition through explicit examples. The grazing contribution may become the leading asymptotic term. The Helmholtz Green's function provides a degenerate example in which the limitation of the classical propagating--regular decomposition becomes explicit: the classical decomposition assigns separate leading-order contributions to the propagating and regular parts, whereas these contributions are naturally associated with the grazing mechanism and cancel only after recombination.
\end{abstract}



\section{Introduction}

The principal obstacle in extending the one-dimensional theory to higher-dimensional periodic elliptic operators is the loss of the analytic framework that makes complex-analytic methods effective. Although Floquet--Bloch theory provides a complete spectral representation, the multidimensional spectral parameter space no longer possesses the one-dimensional analytic properties on which contour deformation, residue analysis, and asymptotic expansions rely.

The underlying difficulty is not merely the loss of one-dimensional analyticity, but also the way the classical spectral representation is organized. Classical Floquet--Bloch theory organizes the spectral representation entirely in terms of quasi-momentum. While this is natural for spectral analysis, outgoing radiation is intrinsically directional, since it simultaneously determines the relevant spectral geometry and the associated one-dimensional analytic setting. This mismatch obscures the analytic mechanisms of radiation.

\begin{figure}[!t]
\centering

\resizebox{0.70\textwidth}{!}{%
\begin{tikzpicture}[
every node/.style={font=\footnotesize,align=center}
]


\node[font=\bfseries] (ssa) at (0,0)
{Sliced Spectral Analysis};

\node (restore) at (0,-1.0)
{One-dimensional analytic structure};


\node[font=\bfseries] (ana) at (-3.4,-2.1)
{Analytic Structure};

\node (slice) at (-3.4,-3.0)
{Spectral slicing};

\node (contour) at (-3.4,-3.9)
{Contour deformation};

\node (residue) at (-3.4,-4.8)
{Residue calculus};


\node[font=\bfseries] (geo) at (3.4,-2.1)
{Geometric Structure};

\node (fermi) at (3.4,-3.0)
{Complexified Fermi surface};

\node (grazing) at (3.4,-3.9)
{Grazing set};

\node (geometry) at (3.4,-4.8)
{Radiation geometry};


\node[font=\bfseries] (decomp) at (0,-6.0)
{Three-component decomposition};

\node (parts) at (0,-6.5)
{ Evanescent \;|\; Non-grazing \;|\; Grazing};

\node[font=\bfseries] (final) at (0,-7.5)
{Geometric Mechanisms of Radiation};


\draw[->] (ssa)--(restore);

\draw[->] (restore)--(ana);
\draw[->] (restore)--(geo);

\draw[->] (ana)--(slice);
\draw[->] (slice)--(contour);
\draw[->] (contour)--(residue);

\draw[->] (geo)--(fermi);
\draw[->] (fermi)--(grazing);
\draw[->] (grazing)--(geometry);

\draw[->] (residue) -- (decomp);
\draw[->] (geometry) -- (decomp);

\draw[->] (parts)--(final);

\end{tikzpicture}
}

\caption{Conceptual overview of sliced spectral analysis.}
\label{fig:ssa}

\end{figure}

To overcome this obstacle, we introduce a sliced spectral analysis (SSA), which promotes the observation direction to a structural variable of the spectral representation. The resulting reformulation transforms the multidimensional Floquet--Bloch spectral problem into a family of one-dimensional analytic problems by decomposing the quasi-momentum into components parallel and transverse to the observation direction. This restores the one-dimensional analytic structure required for contour deformation, residue calculus, and asymptotic expansions. On each spectral slice, the recovered one-dimensional analytic framework consequently permits a directional complexification of the Fermi surface. This local complex geometry provides the basis for the outgoing contour deformation and reveals geometric structures that are invisible on the real Fermi surface. Because the construction is dimension-independent, it remains applicable even for complicated spectral geometries, including band crossings.

Within the SSA framework, a previously unrecognized grazing set emerges. This geometric structure has no analogue in one dimension and gives rise to a distinct asymptotic mechanism of outgoing radiation. As we shall show, this mechanism is not resolved by the conventional propagating--regular decomposition, and therefore requires a new geometric decomposition of radiation.

Within the SSA framework, every outgoing limiting absorption solution admits the decomposition
\[
u
=
u_{\mathrm{evan}}
+
u_{\mathrm{ng}}
+
u_{\mathrm{gra}},
\]
where the three components correspond to evanescent, non-grazing, and grazing spectral modes. The non-grazing and grazing components arise from different spectral structures and possess distinct far-field asymptotic regimes. This decomposition separates the distinct mechanisms of outgoing radiation, replacing the conventional propagating--regular decomposition with one dictated by the underlying geometry. The logical structure of the proposed framework is illustrated in Figure ~\ref{fig:ssa}.

\subsection{The periodic elliptic operator and the limiting absorption principle}

We consider the equation
\begin{equation}
\label{eq:elliptic}
(\mathcal L-\lambda I)u=f
\qquad\text{in }\mathbb R^d,
\end{equation}
where \(\mathcal L\) is a self-adjoint periodic elliptic operator and
\(\lambda\in\mathbb R\) is a prescribed frequency. Precise assumptions
on \(\mathcal L\) will be stated in Section~\ref{sec:def_not}.

The analysis of outgoing radiation becomes nontrivial precisely when $\lambda$ lies in the spectrum of $\mathcal{L}$. In this case the resolvent
\(
(\mathcal L-\lambda I)^{-1}
\)
fails to exist and the construction of physically meaningful outgoing solutions becomes a nontrivial problem.

The standard approach is to interpret outgoing solutions through the limiting absorption principle (LAP), namely,
\[
u
=
\lim_{\varepsilon\to0^+}
u_\varepsilon.
\]
where $u_\epsilon$ solves
\begin{equation}
\label{eq:damp}
(\mathcal L-(\lambda+i\varepsilon)I)u_\varepsilon=f.
\end{equation}
The limiting absorption principle provides the appropriate notion of an outgoing solution, but it does not by itself distinguish the different mechanisms contributing to radiation. The objective of the present paper is therefore not merely to establish the limiting absorption solution, but to develop an analytic framework that reveals the underlying mechanisms of outgoing radiation.

\subsection{Previous work and remaining challenges}

The limiting absorption principle has a long history in the spectral
theory of elliptic operators. When the coefficients are decaying, the existence of LAP solutions
has been established in suitable weighted \(L^p\)-spaces; see, for
instance,
\cite{Odeh1961,Agmon1975,Rodni2015,Cacci2016,Gold2004}.
These approaches rely heavily on the decay of the perturbation and
cannot be directly extended to periodic media.

For periodic elliptic operators, Floquet--Bloch theory provides a natural framework for spectral analysis, in which the continuous spectrum is encoded by the Bloch variety and the associated Fermi surfaces; see
\cite{Bloch1929,Floquet1883,Kuchm2016,Kuchm2023}. Important existence results
for LAP solutions were obtained in
\cite{Gerard1998,Leven1998}, while more explicit Floquet--Bloch constructions were subsequently developed in \cite{Radosz2015,Mandel2019}.

Recent work has considerably advanced the understanding of Green's function asymptotics in periodic media. Near spectral edges, precise asymptotics have been obtained under non-degenerate band-edge assumptions \cite{Kha2017}, while anisotropic periodic media have been analyzed through decompositions into propagating and evanescent components together with caustic phenomena \cite{Ruks2022}. In contrast, the present work develops a direction-dependent spectral representation via SSA, from which the grazing set and the corresponding decomposition emerge naturally.

In the one-dimensional setting,  the LAP and the associated radiation field for periodic waveguides have been extensively studied by a variety of methods, including resolvent analysis, Floquet decompositions, singular perturbations, and energy methods \cite{Hoang2011,Fliss2015,Kirsc2017a,Kirsc2017,Schweizer2023}. These works provide a comprehensive understanding of outgoing radiation in one-dimensional periodic settings.

Despite substantial progress on the existence theory, higher-dimensional periodic elliptic operators remain analytically much less understood than their one-dimensional counterparts. The lack of an effective analytic framework prevents a systematic asymptotic analysis of outgoing radiation. The sliced spectral analysis introduced in this paper fills this gap by recovering the analytic structure required for contour deformation, residue calculus, and asymptotic expansions.

\subsection{Main contributions}

The contributions of the present paper are organized into four stages.

\bigskip
\noindent
{\bf Stage I. A sliced spectral framework (SSA). }The guiding principle is to exploit the intrinsic directional information of the spectral geometry as an organizing variable for analysis. Following this idea, we introduce a sliced spectral analysis (SSA), which restores the one-dimensional analytic structure unavailable in the original multidimensional Floquet--Bloch formulation. By treating the observation direction $\n$ as a parameter, SSA transforms the spectral problem into a family of one-dimensional analytic problems, thereby recovering contour deformation, residue calculus, and asymptotic analysis in higher dimensions. Sections \ref{sec:def_not}--\ref{sec:deform} develop the analytic and geometric foundations of SSA,   including the directional complexification of the Fermi surface, admissible contour deformations, and the characterization of the grazing set.

\bigskip
\noindent
{\bf Stage II. A geometric decomposition of radiation. }Building on SSA, Section \ref{sec:lap} establishes that every outgoing limiting absorption solution admits the decomposition
\[
u
=
u^{\rm evan}
+
u^{\rm ng}
+
u^{\rm gra},
\]
where the three components correspond to evanescent, non-grazing, and grazing spectral modes, respectively. The decomposition is dictated by the underlying spectral geometry and provides the analytic framework for the asymptotic characterization developed in the subsequent sections.

\bigskip
\noindent
{\bf Stage III. Geometric characterization of radiation mechanisms. }
The significance of the three-component decomposition is that its components are governed by different asymptotic laws. Sections \ref{sec:asym_non_gra} and \ref{sec:decay_gra} show that the non-grazing and grazing components possess distinct far-field asymptotic regimes, with leading-order behavior determined by the geometry of the Fermi surface and the grazing set, respectively. This establishes that the three components represent genuinely distinct mechanisms of outgoing radiation.

\bigskip
\noindent
{\bf Stage IV. Necessity of the geometric decomposition. }The final stage explains why the classical propagating--regular decomposition is insufficient for describing multidimensional radiation. Section \ref{sec:geo_gra} presents two representative examples. The first shows that the grazing contribution may dominate the far-field asymptotics, demonstrating that it constitutes an independent radiation mechanism. The second shows that, even when the grazing contribution vanishes in the final asymptotics of the Helmholtz Green's function, the conventional decomposition assigns leading-order contributions to the propagating and regular parts, whose cancellation obscures the underlying mechanism. The SSA decomposition resolves this by separating the grazing contribution explicitly according to its geometric origin.

\bigskip
\noindent
Taken together, these results establish a new analytic framework for outgoing radiation in higher-dimensional periodic media, in which the evanescent, non-grazing, and grazing contributions are identified as three fundamentally distinct radiation mechanisms.

\bigskip
\noindent
The remainder of the paper develops the SSA framework in a progressive manner.
Sections~\ref{sec:def_not} and~\ref{sec:ft} establish its spectral and geometric foundations by introducing
directional spectral slices, the associated Floquet structure, and the Fermi surfaces.
Section~\ref{sec:complex_surface} shows how this directional formulation naturally leads to the complexification
of the Fermi surface and reveals the grazing geometry.
Section~\ref{sec:deform} then combines these geometric ingredients into admissible contour
deformations for the Floquet--Bloch fiber operators.
With this analytic framework in place, Section~\ref{sec:lap} derives the geometric decomposition
of outgoing radiation, while Sections~\ref{sec:asym_non_gra} and~\ref{sec:decay_gra} characterize the asymptotic behavior of
its non-grazing and grazing components. Finally, Section~\ref{sec:geo_gra} illustrates, through explicit
examples, why the grazing contribution represents an essential radiation mechanism and
why the resulting three-component decomposition is necessary.

\section{Geometric and spectral framework}
\label{sec:def_not}

In this section, we introduce the periodic setting and the
spectral objects that will be used throughout the paper.
We first define the periodic elliptic operator and the
associated lattice geometry. We then introduce the
directional coordinates underlying the sliced spectral
analysis and the associated family of quasi-periodic fiber operators.

\subsection{Periodic elliptic operators}
\label{sec:def_L}

Throughout the paper, we consider the second-order elliptic
operator
\[
\mathcal L
=
-\nabla\cdot(A(\bx)\nabla)+V(\bx)
\]
acting on \(L^2(\mathbb R^d)\), where
\[
A\in (W^{1,\infty}(\mathbb R^d))^{d\times d},
\qquad
V\in L^\infty(\mathbb R^d)
\]
are real-valued and periodic with respect to a lattice
\(\Lambda\).

We assume that \(A(\bx)\) is symmetric and uniformly elliptic, i.e.,
there exists a constant \(c_0>0\) such that
\[
{\bm\xi}^T A(\bx){\bm\xi}
\ge
c_0|{\bm\xi}|^2,
\qquad
\forall \bx,{\bm\xi}\in\mathbb R^d.
\]
Define
\[
H^1(A;\mathbb R^d)
=
\left\{
u\in H^1(\mathbb R^d):
\nabla\cdot(A(\cdot)\nabla u)\in L^2(\mathbb R^d)
\right\}.
\]
Under these assumptions, $
\mathcal L:
H^1(A;\mathbb R^d)
\subset
L^2(\mathbb R^d)
\rightarrow
L^2(\mathbb R^d)
$
is a self-adjoint operator.

\subsection{Periodicity, unit cell and its dual cell}
\label{sec:bz}

Let $\Lambda\in \R^{d\times d}$ be an invertible matrix. The corresponding unit cell and first Brillouin zone are defined by:
\[
\Omega:=\left\{\Lambda\widetilde{\bx}:\,\widetilde{\bx}\in(-1/2,1/2]^d\right\}\subset\R^d\quad\text{ and }\quad
\Bz:=\left\{\Lambda^*\widetilde{\by}:\,\widetilde{\by}\in(-1/2,1/2]^d\right\}\subset\R^d
\]
where $\Lambda^*=2\pi\Lambda^{-T}$. Throughout the paper we specialize to $\Lambda=2\pi {\rm I}_d$ so that
\[
\Omega=(-\pi,\pi]^d,\quad \Bz:=(-1/2,1/2]^d.
\]
An illustration  of a  bi-periodic structure in $\R^2$  and  its unit cell $\Omega$ is shown in Figure \ref{fig:bi_per}.

\begin{figure}[htbp]
    \begin{minipage}{0.60\textwidth}
        \centering
        \includegraphics[width=0.9\textwidth]{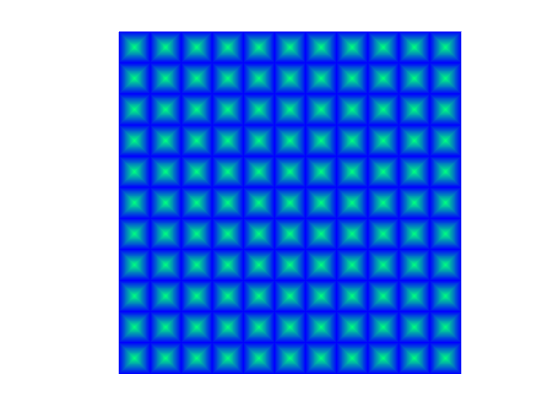}
    \end{minipage}%
    \hfill%
    \begin{minipage}{0.28\textwidth}
        \centering
        \includegraphics[width=\textwidth]{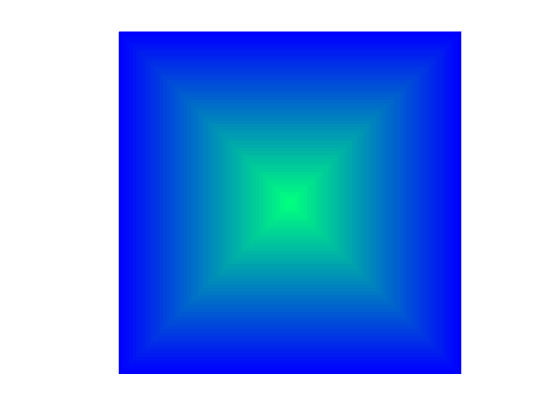}
    \end{minipage}
\caption{A bi-periodic structure in $\R^2$ (left) and one unit cell (right).}
\label{fig:bi_per}
\end{figure}

The faces of the Brillouin zone are denoted by
\[
\Cz_j^\pm
=
\{\alpha\in\overline{\mathcal B}:\alpha_j=\pm\tfrac12\},
\qquad
j=1,\ldots,d.
\]
Hence,
\[
\partial\Bz
=
\bigcup_{j=1}^d
(\Cz_j^+\cup\Cz_j^-).
\]
A two-dimensional illustration is shown in Figure~\ref{fig:B}. The outward unit normal vectors on
\(\Cz_j^\pm\)
are
\(\pm \e_j\),
where
\(\e_j\)
denotes the \(j\)-th canonical basis vector of
\(\mathbb R^d\).

\begin{figure}
\centering
\includegraphics[width=0.5\textwidth]{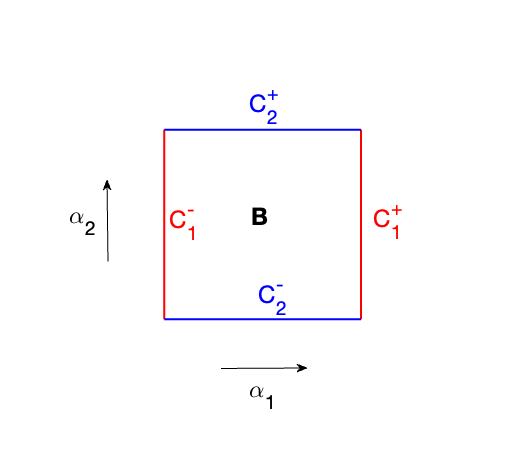}
\caption{Definitions in the dual lattice cell.}
\label{fig:B}
\end{figure}

\subsection{Directional coordinates and spectral slices}
\label{sec:slice_coordinate}

The first step of SSA is a directional decomposition of the Brillouin zone. Let
\[
\n\in\mathbb S^{d-1}
\]
be a fixed unit vector, which will later represent the observation
direction in the physical space.

Choose orthonormal vectors
\[
\{\t_1,\dots,\t_{d-1}\}
\]
spanning the orthogonal complement of \(\n\). Then every Floquet
parameter \(\balpha\in\mathbb R^d\) admits the representation
\[
\balpha
=
\gamma_1\t_1+\cdots+\gamma_{d-1}\t_{d-1}
+s\n
=
\bgamma\t+s\n,
\]
where
\[
\bgamma=(\gamma_1,\dots,\gamma_{d-1})
\in\mathbb R^{d-1},
\qquad
s\in\mathbb R,
\]
and
\[
\bgamma\t
:=
\gamma_1\t_1+\cdots+\gamma_{d-1}\t_{d-1}.
\]

In these coordinates, the Brillouin zone \(\Bz\) can be written as
\[
\Bz
=
\Big\{
\bgamma\t+s\n:
\bgamma\in D_\n,\;
s\in(\ell_l(\bgamma),\ell_r(\bgamma))
\Big\},
\]
where \(D_\n\subset\mathbb R^{d-1}\) is the projection of \(\Bz\) onto
the transverse coordinates.

For each fixed \(\bgamma\in D_\n\), we define the associated
\emph{spectral slice}
\[
\ell_{\bgamma}
:=
\{
\bgamma\t+s\n:
s\in\mathbb R
\}.
\]
The intersection of \(\ell_\bgamma\) with the closure of the
Brillouin zone is a line segment,
\[
\ell_\bgamma\cap\overline{\Bz}
=
\{
\bgamma\t+s\n:
s\in[\ell_l(\bgamma),\ell_r(\bgamma)]
\},
\]
where
\[
\ell_l(\bgamma)
<
\ell_r(\bgamma).
\]
Thus,  \(\bgamma\) labels the spectral slices, whereas
\(s\) parameterizes each slice. The family $\{\ell_\bgamma\}_{\bgamma\in D_\n}$  will be referred to as the spectral slices, which form the basis of the sliced spectral analysis.
Figure \ref{fig:new_coordinate} illustrates this construction for \(d=2\).

\begin{figure}[h]
\centering
\includegraphics[width=0.6\textwidth]{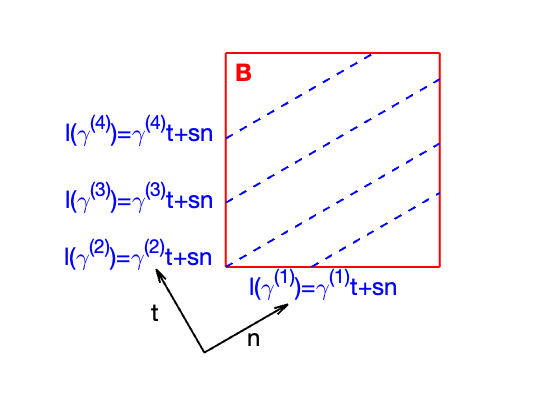}
\caption{The definitions in the new coordinate.}
\label{fig:new_coordinate}
\end{figure}

This directional decomposition forms the geometric foundation of
the SSA framework. By fixing $\bgamma$ and regarding $s$ as the only active variable, the multidimensional spectral problem is reduced to a family of one-dimensional problems, thereby recovering the one-dimensional analytic structure underlying the subsequent analysis.

\subsection{Analytic family of quasi-periodic operators}
\label{sec:ana_opera}

The analytic dependence of the quasi-periodic fiber operators on the Floquet parameter is the key analytic ingredient of SSA. It enables the complexification of the spectral parameter and provides the foundation for the contour deformation and residue analysis developed later. We therefore briefly recall Kato's notion of analytic families of operators, following \cite{Fliss2015}; see also \cite{Kato1995,Reed1978}.


\begin{definition}[Definition 3, \cite{Fliss2015}]
\label{def:ana_operator_ubdd}
Suppose $T(\bz)$ is unbounded. Then we distinguish two notions of
analyticity.

\begin{enumerate}
\item
The family $T(\bz)$ is analytic of type (A) if $D(T(\bz))=D$ for all $\bz\in\mathcal O\subset\C$, and for every $f\in D$ the mapping
\[
\bz\mapsto T(\bz)f
\]
is analytic from $\mathcal O$ into $D$.

\item
The family $T(\bz)$ is analytic if there exist an analytic family
of type (A), denoted by $\widetilde T(\bz)$, and a bounded
analytic family of isomorphisms $S_\bz$ such that
\[
T(\bz)
=
S_\bz \widetilde T(\bz)S_\bz^{-1}.
\]
In this case,
\[
D(T(\bz))
=
S_\bz D
\quad\text{ with }
D=D(\widetilde T(\bz)).
\]
\end{enumerate}
\end{definition}

We now introduce the quasi-periodic fiber operators associated
with  $\mathcal L$. For any
$\balpha\in\C^d$, let $H^s_\balpha(\Omega)$ denote the Sobolev space of $\balpha$-quasi-periodic functions, namely
\[
H^s_\balpha(\Omega)
:=
\Big\{\phi|_{\Omega}:\,
\phi\in H^s_{loc}(\R^d),\,
\phi(\bx+2\pi\bj)
=
e^{\,\i2\pi\balpha\cdot\bj}
\phi(\bx),\,\forall\,\bx\in\R^d,\ \bj\in\Z^d
\Big\}.
\]
We further define
\[
H^1_\balpha(A;\Omega)
=
\Big\{
\phi\in H^1_\balpha(\Omega):
\nabla\!\cdot(A(\cdot)\nabla\phi)
\in L^2(\Omega)
\Big\}.
\]
The quasi-periodic fiber operator is then given by
\[
\mathcal L_\balpha:
H^1_\balpha(A;\Omega)
\subset
L^2(\Omega)
\rightarrow
L^2(\Omega),
\]
with
\[
\mathcal L_\balpha\phi
=
-\nabla\!\cdot(A(\cdot)\nabla\phi)
+
V(\cdot)\phi.
\]

To transfer the quasi-periodicity from the function space to the operator, we introduce the gauge transformation
\[
S_\balpha\phi(\bx)
=
e^{-\i\balpha\cdot\bx}\phi(\bx).
\]
Then $S_\balpha$ is an isomorphism between
$H^1_\balpha(A;\Omega)$ and $H^1_0(A;\Omega)$, where the
subscript \(0\) denotes periodic functions.
Introducing
\[
\widetilde{\mathcal L}_\balpha
=
S_\balpha^{-1}
\mathcal L_\balpha
S_\balpha,
\]
we obtain a family of operators acting on the fixed space
$H^1_0(A;\Omega)$. By direct computation,
$\widetilde{\mathcal L}_\balpha$ depends polynomially on
$\balpha$, and therefore forms an analytic family of type (A) in
the sense of Definition~\ref{def:ana_operator_ubdd}.

Finally, fix a direction
$
\n\in\mathbb S^{d-1}
$
and let
$
\balpha
=
\bgamma\t+s\n.
$
For every fixed
$
\bgamma\in D_{\n},
$
the family
$
s
\longmapsto
\widetilde{\mathcal L}_{\bgamma\t+s\n}
$
is an analytic family of type (A) with respect to the single
complex variable \(s\).

This observation is the key spectral ingredient underlying the
SSA framework. It reduces the original multidimensional spectral problem to a family of one-dimensional analytic problems parameterized by the transverse variable $\bgamma$, thereby restoring the one-dimensional analytic structure underlying the subsequent contour deformation. This one-dimensional analytic structure will serve as the starting point for the Floquet analysis in the next section.

\section{The Floquet theory and spectral structure}
\label{sec:ft}

In this section, we establish the spectral and geometric framework underlying the SSA. We first recall the Floquet--Bloch description of periodic elliptic operators and then introduce the geometric structures that will be used throughout the paper.

The spectrum of a periodic elliptic operator is described by its band functions, whose level sets define the Fermi surfaces. These surfaces constitute the central geometric object of the present work. We therefore begin with the Floquet spectral structure before introducing the geometric assumptions required for the subsequent analysis.

\subsection{Floquet theory and band functions}
\label{sec:ft1}

The Floquet--Bloch theory reduces the spectral analysis of the periodic operator
\(\mathcal L\)
to a family of quasi-periodic fiber operators
\(\mathcal L_\alpha\)
defined on the bounded periodicity cell
\(\Omega\). More precisely,
\begin{equation}
\label{eq:FBT}
\sigma(\mathcal L)
=
\bigcup_{\balpha\in\Bz}
\sigma(\mathcal L_{\balpha}),
\end{equation}
Since each fiber operator acts on a bounded domain, it has compact resolvent and hence a discrete spectrum. We recall the following classical result.

\begin{theorem}[Lemma 5.2, \cite{Kuchm2016}]
    \label{th:kuch1}
For any $\balpha\in\R^d$, the operator $\L_\balpha$ is self-adjoint and therefore admits an increasing sequence of eigenvalues:
\[
\sigma(\L_\balpha)=\big\{\mu_j(\balpha):\, \mu_1(\balpha)\leq \mu_2(\balpha)\leq\cdots\leq\mu_m(\balpha)\leq\cdots\mapsto\infty\big\}.
\]
The function $\mu_j(\balpha)$ is called the {\bf $j$-th band function}.  For each eigenvalue $\mu_j(\balpha)$, there is also a related family of eigenfunctions in $L^2(\Omega)$ such that
\[
\L_\balpha\phi_j(\balpha,\cdot)=\mu_j(\balpha)\phi_j(\balpha,\cdot).
\]
\end{theorem}

The graphs of the band functions form the dispersion relation.
The corresponding dispersion relation is illustrated in Figure \ref{fig:dispersion}.

\begin{figure}[t]
\centering
\includegraphics[width=0.6\textwidth]{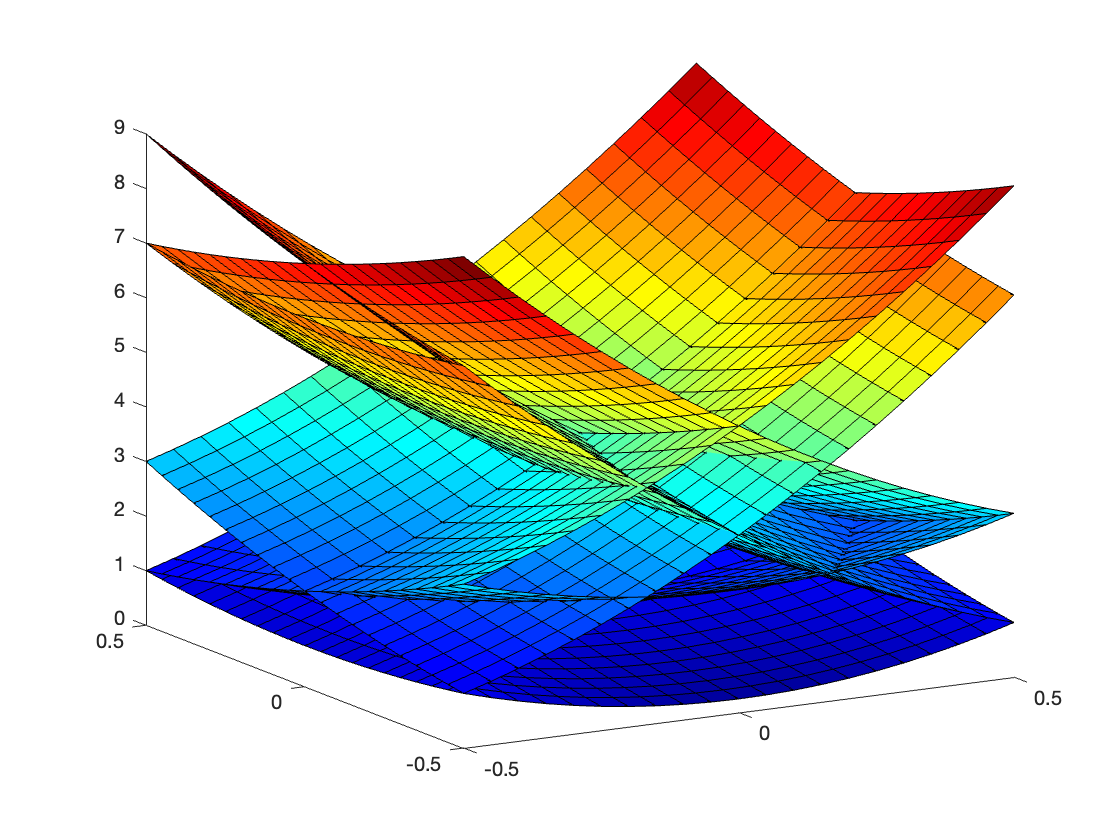}
\caption{The band functions in $\R^2$.}
\label{fig:dispersion}
\end{figure}

The following properties of the band functions will be used repeatedly throughout the paper.

\begin{theorem}[Theorem 5.5, \cite{Kuchm2016}]
    \label{th:kuch2}
The band functions have the following properties:    
\begin{enumerate}
    \item The band functions $\mu_j(\balpha)$ are globally Lipschitz continuous and piecewise real analytic. With a proper scaling, the eigenfunction $\phi_j(\balpha,\cdot) $ also depends Lipschitz continuously and piecewise analytically on $\balpha$.
    \item The graph of the multiple valued mapping
    \[
    \balpha\in\R^d\quad\mapsto\quad \sigma(\L_\balpha)
    \]
    coincides with the dispersion relation of $\L$.
    \item The dispersion relation is $d$-periodic with respect to $\balpha$ with the unit cell $\Bz$, thus it is sufficient to consider it only over the Brillouin zone $\Bz$.
    \item The dispersion relation is symmetric with respect to the mapping $\balpha\mapsto-\balpha$ when the functions $A$ and $V$ are real.
    \item The spectrum $\sigma(\L)$ is the range of the dispersion relation, i.e.,
    \begin{align*}
        \sigma(\L)=\cup_{\balpha\in\Bz}\,\sigma(\L_\balpha)
        &=\big\{\lambda\in\R:\,\exists\,\balpha\in\R^d,\text{ such that }\mu\in\sigma(\L_\balpha)\big\}\\
        &=\big\{\lambda\in\R:\,\exists\,\balpha\in\R^d\text{ and }j\in\N, \text{ such that }\lambda=\mu_j(\balpha)\big\}.
    \end{align*}
\end{enumerate}    
\end{theorem}

periodic operators. Existing analyses typically assume analytic dependence along the spectral parameter, whereas in higher dimensions the band functions are in general only piecewise analytic (see \cite{Kato1982}) because of band crossings.  However, analytic perturbation theory shows that analyticity is recovered whenever the dependence is reduced to a single complex parameter. This is precisely the mechanism underlying SSA.

\begin{theorem}[Theorem 3.9 in Chapter 7, \cite{Kato1995}]
\label{th:spec_1d}
Let $T(\bz):\,H\rightarrow H$ be a self-adjoint analytic family of type (A) where $\bz$ lies in a neighbourhood of an interval $I_0\subset\R$. Assume that $T(\bz)$ has compact resolvent. Then there is a sequence of scalar valued functions $\mu_n(\bz)$ and vector valued functions $\phi_n(\bz,\bx)$ such that they all depend analytically on $\bz$ in the neighbourhood of $I_0$, where $\mu_n(\bz)$ are all the repeated eigenvalues of $T(\bz)$ and $\phi_n(\bz,\cdot)$ the associated eigenvectors of $T(\bz)$. Moreover, $\phi_n(\bz,\cdot)$ form a complete orthonormal family in the space $H$.
\end{theorem} 

Although the band functions are generally only piecewise
analytic in the full Floquet variable \(\balpha\), the
one-dimensional analytic structure described in
Theorem~\ref{th:spec_1d} survives along the spectral slices introduced in Section \ref{sec:def_not}. This observation is the key spectral mechanism behind SSA.

\subsection{Fermi surfaces}
\label{sec:bs}

The central geometric object of the present work is the Fermi surface associated with a fixed frequency $\lambda\in\sigma(\L)$. It provides the geometric foundation for the SSA framework and the subsequent analysis of radiation.

Denote by
\begin{align*}
I_j:={\rm range}(\mu_j)&:=\{\mu\in\R:\, \exists\,\balpha\in\R^d\text{ such that }\mu=\mu_j(\balpha)\}.
\end{align*}
the $j$-th ($j=1,2,\dots$) band of $\sigma(\L)$. Since each band is bounded, only finitely many bands intersect a fixed frequency $\lambda$. The basic properties of the bands are summarized in the following theorem.

\begin{theorem}[Corollary 5.8, \cite{Kuchm2016}]
    \label{th:kuch3}
Let $I_j$ be the $j$-th band of the spectrum $\sigma(\L)$, then
\begin{enumerate}
    \item Each band $I_j$ is a finite closed interval, both  endpoints tend to infinity when $j\rightarrow\infty$.
    \item The band covers the whole spectrum:
    \[
    \sigma(\L)=\cup_{j\in\N} I_j.
    \]
    \item The bands can overlap, i.e. $I_j\cap I_\ell$ can be non-empty for $j\neq \ell$.    
\end{enumerate}
\end{theorem}

From the first argument in Theorem \ref{th:kuch3}, for any fixed $\lambda\in\R$, there are at most finitely many indices $j$ such that $\lambda\in I_j$, i.e.
\[
\exists\,\balpha\in\Bz\text{ such that }\lambda=\mu_j(\balpha).
\]
This motivates the following definition of the {\bf Fermi surface}. 

\begin{definition}[Definition 5.30, \cite{Kuchm2016}]
\label{def:FS}
The level set  for a real-valued $\lambda$ is given by
\[
\bF_\lambda:=\{\balpha\in\Bz:\, \L_\balpha u=\lambda u\text{ has a non-trivial solution}\}.
\]
\end{definition}

Define the set $J=J(\lambda)$ by
\begin{equation}
\label{eq:index_J}
J=\{j\in\N:\,\lambda\in I_j\}.
\end{equation}
Since for any $\lambda\in\R$, only finite number of $j$'s such that $\lambda\in I_j$, $J$ is a finite index set. Then $\bF_\lambda$ is composed of the level sets of finite number of band functions:
\begin{equation}
\label{eq:def_surface}
\bF_\lambda=\cup_{j\in J(\lambda)}\{\balpha\in\R^d:\,\mu_j(\balpha)=\lambda\}:=\cup_{j\in J(\lambda)}\bF_\lambda^j.
\end{equation}
The following theorem, which was summarized by Kuchment from \cite{Thomas1973} and Chapter VIII.16 in \cite{Reed1980}, states an important property of the level set $\bF_\lambda$.

\begin{theorem}[Theorem 5.20, \cite{Kuchm2016}]
    \label{th:kuch4}
   For any $\lambda\in\R$, the level set $\bF_\lambda$ has measure $0$ in $\R^d$. 
\end{theorem}

According to the Remark of Theorem 3.11 in Section 3.4.3 in the book \cite{Evans2015}, we get the following result directly.

\begin{corollary}
\label{th:cor:kuch1}
For any $\lambda\in\R$, the level set $\bF_\lambda$ is measurable as a $(d-1)$-dimensional hypersurface.
\end{corollary}
This regularity is required for the surface integration appearing in Section \ref{sec:lap}.

\subsection{Regular Fermi geometry}
\label{sec:regular}

We now introduce the geometric assumptions underlying the SSA framework. These assumptions provide the regularity needed for the local parameterization of the Fermi surface, the construction of the complexified spectral geometry, and the subsequent asymptotic analysis.

Let $\lambda \in \sigma(\L)$ be fixed. For each
$j \in J(\lambda)$, define the Fermi sheet
\[
\bF_\lambda^j
=
\{
\balpha \in \Bz :
\mu_j(\balpha)=\lambda
\}.
\]

\begin{assumption}[Regular energy]
\label{asp1}
For every $j \in J(\lambda)$ and  $\balpha \in \bF_\lambda^j$, one has $\nabla\mu_j(\balpha)\neq 0$.
\end{assumption}

\begin{assumption}[Analytic branches]
\label{asp2}
For every $j \in J(\lambda)$, there exists a finite covering
of a neighborhood of $\bF_\lambda^j$ by relatively open sets
$\bS_{j,\ell}$ such that:

\begin{enumerate}
    \item $\mu_j$ is real analytic on each
    $\bS_{j,\ell}$;

    \item $\mu_j$ admits a holomorphic extension to a
    complex neighborhood of
    $\overline{\bS_{j,\ell}}$;

    \item the first and second partial derivatives of the
    holomorphic extension, restricted to
    $\bS_{j,\ell}$, extend continuously to
    $\overline{\bS_{j,\ell}}$.
\end{enumerate}
\end{assumption}

\begin{assumption}[Strictly curved Fermi geometry]
\label{asp3}
For every $j \in J(\lambda)$ and every $\balpha \in \bF_\lambda^j$, the second fundamental form $\mathrm{II}_{\balpha}$ is definite on $T_{\balpha}\bF_\lambda^j$.
\end{assumption}

Under the above assumptions, each Fermi sheet is a finite
union of real-analytic hypersurface patches. Gradients and
Hessians are interpreted patchwise and extend continuously
to the closures of the analytic branches. In particular,
all boundary values are understood as limits taken from
within the corresponding branch. Moreover, these
derivatives coincide with the restrictions of the
corresponding derivatives of the holomorphic extension, so
all geometric quantities  admit canonical complexifications through the holomorphic extensions introduced above.

\begin{definition}[Geometric quantities]
\label{def:regular}
For $\balpha \in \bF_\lambda^j$, define
\[
v(\balpha)
:=
|\nabla\mu_j(\balpha)|
\quad\text{
and the unit normal vector}\quad 
\nu(\balpha)
=
\frac{\nabla\mu_j(\balpha)}
{v(\balpha)}.
\]
For a fixed direction $\n$, the quantity
\[
v_\n=\nabla\mu_j\cdot \n
\]
is called the normal velocity.
For tangent vectors ${\bm\xi} \in T_{\balpha}\bF_\lambda^j$ (i.e. $\nabla\mu_j(\balpha)\cdot{\bm\xi}=0$), the second fundamental form is defined by
\[
\mathrm{II}_{\balpha}({\bm\xi},{\bm\xi})
=
\frac{
D^2\mu_j(\balpha)[{\bm\xi},{\bm\xi}]
}{
v(\balpha)
}.
\]
The principal curvatures $\kappa_1(\balpha),\dots,\kappa_{d-1}(\balpha)$
are the eigenvalues of $\mathrm{II}_{\balpha}$. The Gaussian curvature is defined by
\[
K(\balpha)
=
\prod_{m=1}^{d-1}
\kappa_m(\balpha).
\]
\end{definition}

Unlike the full group velocity, only its component \(v_{\n}\) along the observation direction \(\n\) governs the directional analytic structure introduced by SSA. Its zero set on the Fermi surface will play a distinguished
role in the subsequent analysis.

\begin{lemma}[Uniform geometric bounds]
\label{lem:uniform-geometry}
Under the preceding assumptions, there exist positive constants $v_{\min},v_{\max},\kappa_{\min},\kappa_{\max}$ such that
\[
0<v_{\min}
\le v(\balpha)
\le v_{\max},\quad
0<\kappa_{\min}
\le
|\kappa_m(\balpha)|
\le
\kappa_{\max},\quad\forall\,
\balpha\in\bF_\lambda,\,m=1,\dots,d-1.
\]
In particular, the second fundamental form is uniformly
definite on every analytic branch.
\end{lemma}

\begin{proof}
The result follows from the compactness of the Fermi
surface, the finiteness of the analytic covering, and the
continuity of the corresponding geometric quantities.
\end{proof}

\begin{remark}
Assumption~\ref{asp3} is not merely a technical condition for the
application of stationary-phase arguments. It plays two fundamental
roles throughout the paper.

First, as will be shown in Section~\ref{sec:dift},
\[
D^2\mu(\balpha)[\n,\n]
=
v(\balpha)\,
\mathrm{II}_{\balpha}(\n,\n).
\]
Hence the definiteness of the second fundamental form guarantees a
non-degenerate quadratic behavior in the normal direction, leading to
the stable square-root branching of the complexified Fermi surface
near grazing points.

 Second, the leading coefficients in the
non-grazing and grazing asymptotics contain the factor
\[
|\det\mathrm{II}(\balpha^*)|^{-1/2},
\]
so the same geometric condition also guarantees the nondegeneracy of
the asymptotic expansions. The analysis of degenerate Fermi
geometries is left for future work.
\end{remark}

\begin{lemma}
\label{lm:finite_inter}
Under the above assumptions, every affine line intersects
each Fermi sheet $\bF_\lambda^j$ in at most finitely many
points.
\end{lemma}

\begin{definition}[Regular frequency]
\label{def:regular_frequency}
A frequency
$
\lambda\in\sigma(\mathcal L)
$
is called regular if Assumptions
\ref{asp1}--\ref{asp3}
are satisfied.
\end{definition}

\section{Properties of complexified Fermi surface}
\label{sec:complex_surface}

In this section, we study the local geometry of the
complexified Fermi surface
\[
\mu_j(\balpha)=\lambda+\i\epsilon
\]
along a fixed observation direction $\n$.

Let $\bS$ be an analytic branch of a band function
$\mu_j$, and write $\mu$ for the restriction of $\mu_j$ to
$\bS$. By Assumption~\ref{asp2}, $\mu$
admits a holomorphic extension to a complex neighborhood of
$\overline{\bS}$, and its first and second derivatives
extend continuously to $\overline{\bS}$.

For a fixed regular frequency $\lambda$, define the
corresponding local Fermi sheet by
\[
\bL
=
\{\balpha \in \overline{\bS} : \mu(\balpha)=\lambda\}.
\]
Choose an orthonormal basis
$\{\t_1,\dots,\t_{d-1},\n\}$ and write
\[
\balpha
=
\bgamma\t + s\n,
\qquad
\bgamma \in \R^{d-1},
\quad
s \in \R.
\]
Accordingly, we regard $\mu$ as a function of
$(\bgamma,s)$ through
\[
\mu(\bgamma,s)
=
\mu(\bgamma\t+s\n).
\]
Classical complex Bloch varieties are complexified with respect to the full quasi-momentum, whereas SSA complexifies only the slice variable associated with a fixed observation direction. It provides the geometric foundation for the outgoing contour deformation developed in the next section.

Let
\[
\balpha_0
=
\bgamma^{(0)}\t+s_0\n
\in \bL.
\]
The local structure of the complexified level set
\[
\mu(\bgamma,s)=\lambda+\i\epsilon
\]
depends on whether the directional derivative
$
\partial_{\n}\mu(\balpha_0)
=
\nabla\mu(\balpha_0)\cdot\n
$ equals  $0$. Accordingly, two geometrically distinct cases arise:
\begin{itemize}
    \item \textbf{Non-grazing points:}
    $\nabla\mu(\balpha_0)\cdot\n \neq 0$.
    In this case the complexified level set is locally
    represented by a single holomorphic branch
    $s=s(\bgamma,\epsilon)$.

    \item \textbf{Grazing points:}
    $\nabla\mu(\balpha_0)\cdot\n = 0$.
    In this case the complexified level set splits into two
    square-root branches
    $s_\pm(\bgamma,\epsilon)$.
\end{itemize}

These two local models constitute the basic local geometries underlying the SSA framework. They will later be assembled into a global parameterization of the outgoing complex Fermi surface.


\subsection{Non-grazing points}
\label{sec:cift}

At non-grazing points, the nonvanishing normal velocity
\[
v_n(\balpha_0)
=
\nabla\mu(\balpha_0)\cdot\n
\neq 0
\]
allows one to resolve the complexified Fermi surface
locally by the implicit function theorem.


\begin{theorem}[Local parameterization near non-grazing points]
\label{th:non-grazing}
Let
\[
\balpha_0
=
\bgamma^{(0)} \t + s_0 \n
\in \bL
\quad\text{
satisfy
}\quad
\nabla \mu(\balpha_0)\cdot \n \neq 0.
\]
Then there exist neighborhoods of $\bgamma^{(0)}$ and $0$
and a unique local branch
$
s=s(\bgamma,\epsilon)
$
such that
\[
s(\bgamma^{(0)},0)=s_0
\quad\text{
and
}\quad
\mu(\bgamma \t+s(\bgamma,\epsilon)\n)
=
\lambda+\i\epsilon.
\]
Furthermore,
\begin{equation}
\label{eq:ift_diff}
\nabla_{\bgamma}s(\bgamma^{(0)},0)
=
-
\frac{\nabla_{\bgamma}\mu(\balpha_0)}
{\nabla\mu(\balpha_0)\cdot \n},
\qquad
\partial_\epsilon s(\bgamma^{(0)},0)
=
\frac{\i}
{\nabla\mu(\balpha_0)\cdot \n}.
\end{equation}
In particular, for sufficiently small $\epsilon>0$,
\begin{equation}
\label{eq:ift_ext}
s(\bgamma^{(0)},\epsilon)
=
s_0
+
\frac{\i\epsilon}
{\nabla \mu(\balpha_0)\cdot \n}
+
O(\epsilon^2).
\end{equation}
Therefore, $\operatorname{Im}\bigl(s(\bgamma^{(0)},\epsilon)-s_0
\bigr)$ has the same sign as $\nabla \mu(\balpha_0)\cdot \n$.

\end{theorem}

\begin{proof}
Define
\[
G(\bgamma,\epsilon,s)
=
\mu(\bgamma \t + s \n)
-
\lambda
-
\i\epsilon.
\]
Then
\[
G(\bgamma^{(0)},0,s_0)=0
\quad\text{and}\quad
\partial_s G(\bgamma^{(0)},0,s_0)
=
\nabla \mu(\balpha_0)\cdot \n
\neq 0.
\]

By Assumption~\ref{asp2}, the function $G$ extends to a
$C^2$ function in a neighborhood of
$(\bgamma^{(0)},0,s_0)$. The implicit function theorem
therefore yields a unique local branch
$s=s(\bgamma,\epsilon)$ satisfying the stated identity. 
Differentiating with respect to $\bgamma$ and $\epsilon$ at
$(\bgamma^{(0)},0)$ similarly gives \eqref{eq:ift_diff},
which implies the expansion \eqref{eq:ift_ext}.
\end{proof}

Thus, away from the grazing set, the complexified Fermi surface therefore admits a single analytic branch whose imaginary displacement is determined by the normal velocity. The sign of the normal velocity uniquely determines the outgoing complex branch.

\subsection{Grazing points}
\label{sec:dift}

\subsubsection{The grazing set}

We now consider points $\balpha_0=\bgamma^{(0)} \t + s_0 \n\in \bL$
such that $v_\n(\balpha_0)=\nabla \mu(\balpha_0)\cdot \n = 0$. 
At such points, the first-order normal derivative vanishes,
and the complexified Fermi surface develops a square-root
branch structure.

Define the grazing set
\[
G_\n=
\{
\balpha \in \bL :v_\n(\balpha)=\nabla\mu(\balpha)\cdot n = 0
\}.
\]
For every $\balpha_0\in G_\n$, according to Assumption \ref{asp1}, we choose
\[
\t_1
=
\frac{\nabla\mu(\balpha_0)}
{|\nabla\mu(\balpha_0)|},
\]
and complete it to an orthonormal basis
\[
\{\t_1,\dots,\t_{d-1},\n\}.
\]
We write
\[
\bgamma=(\bgamma_1,\bgamma'),
\qquad
\bgamma'
\in \R^{d-2}.
\]

\begin{theorem}[Local parameterization of the grazing set]
\label{th:grazing-set}
Let $\balpha_0\in G_\n$. Then there exist a neighborhood
$W_0$ of ${\bgamma^{(0)}}'$ and functions
$
\phi(\bgamma')$
 and $
\psi(\bgamma')
$
such that
\[
(\psi(\bgamma'),\bgamma')\t
+
\phi(\bgamma')\n
\]
parameterizes a neighborhood of $\balpha_0$ in $G_\n$.
\end{theorem}

\begin{proof}
The grazing set is characterized by
\[
\mu(\bgamma\t+s\n)=\lambda,
\qquad
\partial_\n\mu(\bgamma\t+s\n)=0.
\]
Define
\[
F(\gamma_1,s,\bgamma')
=
\begin{pmatrix}
\mu(\bgamma\t+s\n)-\lambda\\
\partial_n\mu(\bgamma\t+s\n)
\end{pmatrix}.
\]
At $\balpha_0=\left(\bgamma_1^{(0)},{\bgamma^{(0)}}',s_0
\right)$, one has
\[
F(\gamma_1^{(0)},s_0,{\bgamma^{(0)}}')=0.
\]
From the definition of $\t_1$,
\[
\partial_{\bgamma_1}\mu(\balpha_0)=\t_1\cdot\nabla\mu(\balpha_0)
=
|\nabla\mu(\balpha_0)|
=
v(\balpha_0).
\]
Since $\balpha_0\in G_\n$, the direction $\n$ is tangent to
the Fermi surface at $\balpha_0$, i.e.
\[
\partial_{s}\mu(\balpha_0)
=
\nabla\mu(\balpha_0)\cdot\n
=0.
\]
Moreover,
\[
\partial_s^2\mu(\balpha_0)
=
D^2\mu(\balpha_0)[\n,\n]
=
v(\balpha_0)\,
\mathrm{II}_{\balpha_0}(\n,\n).
\]
Hence
\[
\det
D_{(\gamma_1,s)}F
(\gamma_1^{(0)},s_0,{\bgamma^{(0)}}')
=
v(\balpha_0)^2
\,
\mathrm{II}_{\balpha_0}(\n,\n).
\]

By Assumptions~\ref{asp1} and~\ref{asp3},
\[
v(\balpha_0)\neq0,
\qquad
\mathrm{II}_{\balpha_0}(\n,\n)\neq0.
\]
Therefore the Jacobian is non-degenerate, and the implicit
function theorem yields the stated local parameterization.
\end{proof}

\begin{remark}
The grazing set is precisely the locus where the complexified
Fermi surface changes from a single-sheeted to a two-sheeted
structure. It is characterized by the two equations
\[
\mu(\balpha)=\lambda,
\qquad
\nabla\mu(\balpha)\cdot \n=0.
\]
Under Assumptions~\ref{asp1} and~\ref{asp3}, Theorem~\ref{th:grazing-set}
shows that these constraints are locally independent. Consequently,
the grazing set forms a smooth codimension-one submanifold of the
Fermi sheet. Since the Fermi sheet itself is a hypersurface of the
Brillouin zone, the grazing set is a codimension-two submanifold of
the Brillouin zone.

In particular, this geometric structure is specific to dimensions
$d\ge2$. In one-dimensional periodic problems, the Fermi surface
consists of isolated points and therefore cannot support a codimension-one
grazing stratum. This explains why the grazing geometry studied in the
present paper has no genuine one-dimensional analogue.
\end{remark}

\subsubsection{Square-root branching}

For $\bgamma' \in W_0$, define
\[
\balpha(\bgamma')
=
(\psi(\bgamma'),\bgamma')\t
+
\phi(\bgamma')\n
\in G_\n.
\]
We introduce the coefficients
\[
a(\bgamma')
:=
D^2\mu(\balpha(\bgamma'))[\n,\n],
\quad
b(\bgamma')
:=
\t_1\cdot
\nabla\mu(\balpha(\bgamma')).
\]

At the base point, from Assumption \ref{asp3},
\[
a({\bgamma^{(0)}}')
=
D^2\mu(\balpha_0)[\n,\n]
=
v(\balpha_0)\,
\mathrm{II}_{\balpha_0}(\n,\n)
\neq0.
\]
Also from Assumption \ref{asp1},
\[
b({\bgamma^{(0)}}')
=
|\nabla\mu(\balpha_0)|
=
v(\balpha_0)
\neq0.
\] Therefore, after possibly
shrinking $W_0$, they remain continuous and
nonvanishing on $W_0$. 

The coefficients
$
a(\bgamma')
$ and $b(\bgamma')$
represent, respectively, the normal quadratic curvature and
the transverse velocity of the Fermi surface along the
grazing set.

\begin{theorem}[Square-root branching near grazing points]
\label{th:grazing-branch}
Let $\balpha_0\in G_\n$. Then, after possibly shrinking the
neighborhoods, the equation
\[
\mu(\bgamma\t+s\n)
=
\lambda+\i\epsilon
\]
splits into two local branches
$
s_\pm(\bgamma,\epsilon)
$
such that
\begin{equation}
\label{eq:square_branch}
s_\pm(\bgamma,\epsilon)
=
\phi(\bgamma')
\pm
\sqrt{
-\frac{2b(\bgamma')}{a(\bgamma')}
\bigl(
\bgamma_1-\psi(\bgamma')
\bigr)
+
\frac{2\i\epsilon}{a(\bgamma')}
}
+
O\!\left(
\bigl(
|\bgamma_1-\psi(\bgamma')|
+
|\epsilon|
\bigr)^{3/2}
\right),
\end{equation}
where the square root is taken with positive imaginary
part.
\end{theorem}

\begin{proof}
Expanding $\mu$ near the grazing set gives
\[
\mu(\bgamma\t+s\n)-\lambda
=
b(\bgamma')
\bigl(
\gamma_1-\psi(\bgamma')
\bigr)
+
\frac12
a(\bgamma')
\bigl(
s-\phi(\bgamma')
\bigr)^2
+
R,
\]
where
\[
R
=
O\!\left(
|\gamma_1-\psi(\bgamma')|^2
+
|\gamma_1-\psi(\bgamma')|
|s-\phi(\bgamma')|
+
|s-\phi(\bgamma')|^3
\right).
\]

Neglecting the higher-order remainder, the leading-order balance
\[
b(\bgamma')
(\gamma_1-\psi(\bgamma'))
+
\frac12
a(\bgamma')
(s-\phi(\bgamma'))^2
=
\i\epsilon
\]
shows that
\[
s-\phi(\bgamma')
=
O\!\left(
\bigl(
|\gamma_1-\psi(\bgamma')|
+
|\epsilon|
\bigr)^{1/2}
\right).
\]

Consequently,
\[
|\gamma_1-\psi(\bgamma')|
\,|s-\phi(\bgamma')|
=
O\!\left(
\bigl(
|\gamma_1-\psi(\bgamma')|
+
|\epsilon|
\bigr)^{3/2}
\right),
\]
and therefore the mixed term is of strictly higher order than the
leading terms. The same estimate also shows that the cubic remainder
is
\[
O\!\left(
\bigl(
|\gamma_1-\psi(\bgamma')|
+
|\epsilon|
\bigr)^{3/2}
\right).
\]

Since $a(\bgamma')\neq0$, the quadratic equation determined by the
leading terms admits two branches. Solving it perturbatively gives
\eqref{eq:square_branch}.
\end{proof}

\begin{remark}
The two branches $s_+(\bgamma,0)$ and $s_-(\bgamma,0)$ coincide
if and only if
\[
\bgamma_1
=
\psi(\bgamma').
\]
In this case,
\[
s_\pm(\bgamma,0)
=
\phi(\bgamma'),
\]
and the corresponding point belongs to $G_\n$.
\end{remark}


\subsection{Global outgoing geometry}
\label{sec:par_F}

The local constructions developed above can now be assembled into a finite global description of the outgoing geometry. For a fixed direction $\n\in\mathbb{S}^{d-1}$, the corresponding outgoing real sheet is defined by
\[
\bL^+
:=
\left\{
\balpha\in\bS:
\nabla\mu(\balpha)\cdot\n>0
\right\},
\]
while the outgoing complex sheet is defined by
\[
\bL_c^+
:=
\left\{
\balpha\in\bS^\C:
\operatorname{Im}s(\balpha)>0
\right\},
\]
where $\bS^\C$ denotes the local complexification of $\bS$.

 Around each grazing point, Theorem~\ref{th:grazing-branch} provides a square-root parametrization whose real and complex sides describe respectively the outgoing Fermi sheet and its complex extension.

Since \(G_\n\) is compact, finitely many such grazing charts
 cover \(G_\n\). We denote the corresponding real
and complex pieces by
\[
\Gamma_j^+,
\qquad
\Gamma_{j,c}^+,
\qquad
j=1,\dots,M.
\]
Removing these neighborhoods leaves a
compact subset of \(\bL^+\)  uniformly separated
from \(G_n\). By Theorem~\ref{th:non-grazing}, this remaining part is covered by finitely many regular charts
\[
\Gamma_j^+,
\qquad
j=M+1,\dots,N.
\]

Combining the grazing and regular charts gives a finite
global description of the outgoing real and complex sheets.
This is summarized in the following theorem.

\begin{theorem}[Finite parametrization of the outgoing geometry]
\label{th:finite-param-outgoing}

Let \(\lambda\) be regular and fix an observation direction
\(\n\in S^{d-1}\). Then there exist finitely many local coordinate systems
\[
\left\{\t_1^{(j)},\dots,\t_{d-1}^{(j)},\n^{(j)}\right\},
\qquad
j=1,\dots,N,
\]
parameter domains \(V_j^+\subset\R^{d-1}\), and local
branches \(s_j(\bgamma,\epsilon)\) such that the outgoing
sheet admits the representation
\[
\bL^+
=
\bigcup_{j=1}^{N}\Gamma_j^+,
\quad\text{ 
where }
\Gamma_j^+
=
\Bigl\{
\bgamma\t^{(j)}
+
s_j(\bgamma,0)\n^{(j)}
:
\bgamma\in V_j^+
\Bigr\}.
\]

Moreover, there exists an integer \(M\le N\) such that the
charts \(j=1,\dots,M\) are grazing charts. For each grazing
chart there is an additional parameter domain
\(V_j^-\subset\R^{d-1}\) corresponding to the complex side
of the square-root branch. The outgoing complex extension
is therefore given by
\[
\bL_c^+
=
\bigcup_{j=1}^{M}
\Gamma_{j,c}^+,
\quad\text{ 
where }
\Gamma_{j,c}^+
=
\Bigl\{
\bgamma\t^{(j)}
+
s_j(\bgamma,0)\n^{(j)}
:
\bgamma\in V_j^-
\Bigr\}.
\]

\end{theorem}

For every sufficiently small \(\epsilon>0\), the associated
complexified charts
\[
\Sigma_j^+(\epsilon)
=
\Bigl\{
\bgamma\t^{(j)}
+
s_j(\bgamma,\epsilon)\n^{(j)}
:
\bgamma\in V_j^+
\Bigr\},\quad j=1,2,\dots,N,
\]
and
\[
\Sigma_{j,c}^+(\epsilon)
=
\Bigl\{
\bgamma\t^{(j)}
+
s_j(\bgamma,\epsilon)\n^{(j)}
:
\bgamma\in V_j^-
\Bigr\},\quad j=1,2,\dots,M,
\]
yield the finite representation
\[
\Sigma^+(\epsilon)
=
\bigcup_{j=1}^{N}
\Sigma_j^+(\epsilon)
\cup
\bigcup_{j=1}^{M}
\Sigma_{j,c}^+(\epsilon).
\]

Theorem~\ref{th:finite-param-outgoing} is stated for a single analytic branch of a band function. Applying it to every branch associated with the regular frequency $\lambda$ yields finite coverings of the corresponding outgoing real and complex Fermi surfaces, i.e., $\bF^+_j$ and $\bF^+_{j,c}$. These local chart families will be used throughout the subsequent contour deformation and residue analysis.

\section{Admissible deformation of the fiber resolvent}
\label{sec:deform}

The purpose of this section is to construct an admissible complex
deformation for the Floquet--Bloch fiber resolvent at a regular
frequency $\lambda$. The local geometric construction developed in Section~\ref{sec:complex_surface} applies to individual analytic branches. This section constructs an admissible complex deformation of the Floquet--Bloch fiber resolvent. The construction naturally separates into two
independent parts according to the spectral structure of the fiber
operator.

The first part concerns the resonant component, namely the finitely
many analytic band functions intersecting the energy level
$\lambda$. The second part concerns the remaining infinitely many
non-resonant bands.Finally, the deformation obtained for the resonant component and the
uniform invertibility of the non-resonant component are combined to
yield an admissible deformation for the full Floquet--Bloch fiber
operator.

\subsection{Resonant component}
\label{sec:complex_Fermi}

We begin with the resonant component of the spectrum. The deformation problem can
be reduced to a branchwise analysis. The objective of this subsection
is first to construct a local admissible deformation for each
resonant branch, and then to combine these local constructions into a
uniform deformation valid for all resonant branches. The key step is to understand how the complexified level
set intersects directional complex rays
\[
\left\{\balpha+\i t\boldsymbol{\omega}:\,
,t\geq0\right\},\quad\bomega\in\mathbb S^{d-1}.
\]
The resulting crossing analysis provides the geometric basis for the
contour deformation constructed below.

Throughout this subsection, we fix a single analytic branch of a band
function and work with the corresponding Fermi sheet.  For simplicity, the band index \(j\)
is omitted whenever no confusion arises. 

\subsubsection{Quadratic approximation of the crossing equation}

The construction of the admissible deformation begins with a local
analysis of the intersection equation
\begin{equation}
\label{eq:int_cross}
\mu(\balpha+s\nu(\balpha)+\i t\boldsymbol{\omega})
=\lambda,
\end{equation}
where $\nu(\balpha)$ is the unit normal vector of the Fermi sheet and  $\boldsymbol{\omega}\in\mathbb S^{d-1}$ is fixed. 
To analyze the local intersection geometry, we consider the Taylor expansion of \eqref{eq:int_cross} near the Fermi surface.
The derivation of the expansion is straightforward but lengthy and is therefore deferred to Corollary \ref{cor:directional_expansion}. 
The expansion takes the form of the local equation
\begin{equation}
\label{eq:quad_cross}
A(\balpha,\bomega)t^2+B(\balpha)(s+\i \bomega_1 t)+C(\balpha) (s+\i \bomega_1 t)^2+(s+\i \bomega_1 t)t D(\balpha,\bomega)+o(s^2+t^2)=0
\end{equation}
where
\[
A(\balpha,\bomega)
=
-\frac12
\sum_{k=2}^d
\kappa_k(\balpha)\bomega_k^2,
\qquad
B(\balpha)
=v(\balpha)\]\[ C(\balpha)=\frac{1}{2}\frac{\partial^2\mu(\balpha)}{\partial x_1^2},\qquad D(\balpha,\bomega)=\sum_{k=2}^d \frac{\partial^2 \mu(\balpha)}{\partial x_1 x_k}\bomega_k.
\]
The coefficients appearing in the quadratic expansion represent the
normal velocity, the principal curvatures, and the mixed second-order
derivatives of the band function. According to Lemma~\ref{lem:crossing-non-degenerate}, there exist
$c_0>0$ and sufficiently small $\eta,\tau,\delta>0$ such that
\begin{equation}
\label{eq:coef1}
|A(\balpha,\bomega)-C(\balpha)\bomega_1^2|
\ge \frac{\kappa_{\min}}{2},
\qquad
|\bomega_1|\le\eta,
\end{equation}
and
\begin{equation}
\label{eq:coef2}
|B(\balpha)+2C(\balpha)s+D(\balpha,\bomega)t|
\ge \frac{v_{\min}}2,
\qquad
|s|\le\tau,\quad 0\le t\le\delta,
\end{equation}
uniformly for $\balpha\in\overline{U_{\balpha_0}}$ and
$\bomega\in\mathbb S^{d-1}$.

The preceding quadratic expansion allows us to control how directional
complex rays intersect the complexified Fermi surface near a real Fermi
sheet. The behavior depends on the normal component
\(\bomega_1=\bomega\cdot\nu(\balpha)\) of the perturbation direction.
When \(|\bomega_1|\) is small, the intersection is governed by the
nondegenerate quadratic term in \(t\); when \(|\bomega_1|\) is bounded
away from zero, the imaginary part of the intersection equation
precludes any nontrivial intersection for sufficiently small \(t\).
These two regimes together yield the local crossing property below.

To formulate the local crossing property uniformly near the Fermi
surface, let $\balpha_0\in\bF_\lambda$ and let
$U_{\balpha_0}\subset\bF_\lambda$ be a sufficiently small neighborhood
of $\balpha_0$. For $\tau>0$, define the tubular neighborhood
\[
\mathcal N_\tau(U_{\balpha_0})
=
\{
\balpha+s\nu(\balpha):
\balpha\in U_{\balpha_0},\ |s|<\tau
\}.
\]

\begin{proposition}[Local one-crossing property]
\label{prop:local_crossing}
Let $\balpha_0\in\bF_\lambda$. Then there exist a neighborhood
$U_{\balpha_0}\subset\bF_\lambda$ of $\balpha_0$ and constants
$\tau,\delta_0>0$ such that, for every
$
\bbeta\in\mathcal N_\tau(U_{\balpha_0})$
 and 
$\bomega\in\mathbb S^{d-1}$, 
the complex ray
\[
\{\bbeta+\i t\bomega:0<t<\delta_0\}
\]
intersects the local complexified Fermi surface
\[
\{{\bm \zeta}:\mu({\bm \zeta})=\lambda\}
\]
in at most one point. Moreover,
\[
\mu(\bbeta+\i\delta_0\bomega)\neq\lambda,\quad
\forall\,
\bbeta\in\mathcal N_\tau(U_{\balpha_0}),
\,
\bomega\in\mathbb S^{d-1}.
\]
\end{proposition}

\begin{proof}

For
\[
\bbeta=\balpha+s\nu(\balpha)
\in\mathcal N_\tau(U_{\balpha_0}),
\]
the intersection condition
\[
\mu(\bbeta+\i t\bomega)=\lambda
\]
is equivalent to the local equation \eqref{eq:quad_cross}.
The coefficients satisfy the uniform estimates
\eqref{eq:coef1}--\eqref{eq:coef2}.
Taking the real and imaginary parts of \eqref{eq:quad_cross}, we obtain
\begin{eqnarray}
\label{eq:loc_int_r}
A(\balpha,\bomega)t^2+B(\balpha)s
+C(\balpha)(s^2-\bomega_1^2t^2)
+D(\balpha,\bomega)st
+o(s^2+t^2)
&=&0,
\\
\label{eq:loc_int_i}
B(\balpha)\bomega_1t
+2C(\balpha)\bomega_1st
+D(\balpha,\bomega)\bomega_1t^2
+o(s^2+t^2)
&=&0.
\end{eqnarray}
We distinguish two cases according to the size of the normal
component $\bomega_1$.

\medskip
\noindent
{\it Case (i): $|\bomega_1|\le\eta$.}
By \eqref{eq:coef1}, the real equation
\eqref{eq:loc_int_r} can be written as
\[
t^2
=
-\frac{
B(\balpha)+C(\balpha)s+D(\balpha,\bomega)t+o(s)
}{
A(\balpha,\bomega)-C(\balpha)\bomega_1^2+o(1)
}\,s.
\]
The uniform nondegeneracy of the denominator, together with
$B(\balpha)\ge v_{\min}>0$, implies that every solution satisfies
\[
c|s|\le t^2\le C|s|
\]
for some constants $0<c<C$ independent of
$\balpha,\bomega,s$, and $t$. Hence
\[
t\sim\sqrt{|s|}.
\]
Thus, in the nearly tangential regime, all possible intersections
are confined to the parabolic scaling region $t^2\asymp |s|$.
In particular, once a sufficiently small deformation height
$\delta_0>0$ is fixed, the tubular radius $\tau>0$ can be chosen
sufficiently small so that every such intersection satisfies
\[
0<t<\delta_0.
\]

\medskip
\noindent
{\it Case (ii): $|\bomega_1|\ge\eta$.}
In this regime, \eqref{eq:loc_int_i} yields
\[
t
=
-\frac{o(t^2)}
{B(\balpha)+2C(\balpha)s+D(\balpha,\bomega)t}
\frac{1}{\bomega_1}.
\]
By \eqref{eq:coef2}, the denominator is uniformly bounded away
from zero. Since $|\bomega_1|\ge\eta$, we obtain
\[
|t|\le C\,o(t^2).
\]
For sufficiently small $t$, this is impossible unless $t=0$.
Hence no nontrivial intersection occurs in this regime.

Combining the two cases, we conclude that, after choosing
$\delta_0>0$ sufficiently small and then shrinking $\tau>0$,
every possible nontrivial intersection is confined to
\[
0<t<\delta_0.
\]
Consequently,
\[
\mu(\bbeta+\i\delta_0\bomega)\neq\lambda,\quad\forall\,
\bbeta\in\mathcal N_\tau(U_{\balpha_0}),
\,
\bomega\in\mathbb S^{d-1}.
\]

Therefore, each complex ray
\[
\{\bbeta+\i t\bomega:0<t<\delta_0\}
\]
intersects the local complexified Fermi surface in at most one point.
\end{proof}

\subsubsection{Global admissible deformation}
\label{sec:deformation}

Proposition~\ref{prop:local_crossing} establishes the local
one-crossing property in a neighborhood of every point of the
Fermi surface. We now globalize this construction and use it to
obtain a uniform admissible deformation.

Define the tubular neighborhood
\[
\mathcal N_\tau(\bF_\lambda)
=
\Bigl\{
\balpha+s\nu(\balpha):
\balpha\in\bF_\lambda,\ |s|<\tau
\Bigr\}.
\]
By the compactness of \(\bF_\lambda\), finitely many local
neighborhoods \(U_{\balpha_k}\) suffice to cover
\(\bF_\lambda\). Applying
Proposition~\ref{prop:local_crossing} to these neighborhoods
and taking the minimum of the corresponding local parameters,
we obtain uniform constants $\tau>0$ and $\delta_{\rm int}>0$
for which the local one-crossing property holds throughout
\(\mathcal N_\tau(\bF_\lambda)\).

It remains to control the region away from the Fermi surface.
Since
\[
\overline{\Bz\setminus\mathcal N_\tau(\bF_\lambda)}
\]
is compact and disjoint from \(\bF_\lambda\), a sufficiently
small complex perturbation remains uniformly separated from the
complexified Fermi surface. Combining these two regions yields
the following global deformation.

\begin{theorem}[Existence of a global admissible deformation]
\label{thm:global_sigma}

Assume Assumption~\ref{asp3}. Then there exist
$\tau>0 $ and $0<\delta_{\rm ext}\le\delta_{\rm int}$,
and a piecewise constant function
\[
\sigma(\balpha)
=
\begin{cases}
\delta_{\rm int},
&
\balpha\in\mathcal N_\tau(\bF_\lambda),
\\[2mm]
\delta_{\rm ext},
&
\balpha\in
\Bz\setminus
\mathcal N_\tau(\bF_\lambda),
\end{cases}
\]
such that for every
\(
\boldsymbol{\omega}\in S^{d-1}
\),
the deformed Brillouin zone
\[
\Bz_{\boldsymbol{\omega}}
=
\Bigl\{
\balpha+\i\sigma(\balpha)\boldsymbol{\omega}
:
\balpha\in\Bz
\Bigr\}
\]
avoids the complexified Fermi surface, i.e.,
\[
\Bz_{\boldsymbol{\omega}}
\cap
\{
\zeta:
\mu(\zeta)=\lambda
\}
=
\varnothing.
\]

More precisely:
\begin{enumerate}

\item
If
\(
\balpha\in\mathcal N_\tau(\bF_\lambda),
\)
then the ray
$\{
\balpha+\i t\boldsymbol{\omega}:\,
0<t<\sigma(\alpha)\}$ satisfies the uniform one-crossing property inherited from
Proposition~\ref{prop:local_crossing}. In particular,
$
\balpha+\i\sigma(\balpha)\boldsymbol{\omega}
$
does not lie on the complexified Fermi surface.

\item
If
\(
\balpha\in
\Bz\setminus
\mathcal N_\tau(\bF_\lambda),
\)
then
$\{
\balpha+\i t\boldsymbol{\omega}:\,
0\leq t\leq\sigma(\balpha)\}$
remains uniformly separated from the complexified Fermi
surface.
\end{enumerate}
Consequently, \(
\Bz_{\boldsymbol{\omega}}
\)
defines a globally admissible complex deformation of the
band function $\mu$. 

\end{theorem}

\begin{proof}

By Proposition~\ref{prop:local_crossing} and the compactness
argument above, there exist uniform constants
$\tau>0$ and $\delta_{\rm int}>0$
such that, for every $\balpha\in\mathcal N_\tau(\bF_\lambda)$ and
$\bomega\in\mathbb S^{d-1}$,
the corresponding complex ray satisfies the uniform
one-crossing property and
\[
\mu(\balpha+\i\delta_{\rm int}\bomega)\neq\lambda.
\]

It remains to consider the region away from the Fermi surface.
Since $
\overline{\Bz\setminus\mathcal N_\tau(\bF_\lambda)}$ is compact and disjoint from
$\bF_\lambda$, we have
\[
|\mu(\balpha)-\lambda|\geq c_0>0,\quad\forall\balpha\in \overline{\Bz\setminus\mathcal N_\tau(\bF_\lambda)}.
\]
By continuity of the holomorphic extension of $\mu$, there exists
$\delta_{\rm ext}>0$ such that
\[
\mu(\balpha+\i t\bomega)\neq\lambda,\quad\forall\,\balpha\in \overline{\Bz\setminus\mathcal N_\tau(\bF_\lambda)},
\,
\bomega\in\mathbb S^{d-1},
\,
0\le t\le\delta_{\rm ext}.
\]
After replacing $\delta_{\rm ext}$ by
\[
\min\{\delta_{\rm ext},\delta_{\rm int}\},
\]
we may assume
\[
0<\delta_{\rm ext}\le\delta_{\rm int}.
\]

Combining the two regions gives
\[
\mu\bigl(\balpha+\i\sigma(\balpha)\bomega\bigr)
\neq\lambda,\quad\forall\,\balpha\in\Bz,\,\bomega\in\mathbb{S}^{d-1}.
\]
 Hence
\[
\Bz_{\bomega}
\cap
\{{\bm\zeta}:\mu({\bm\zeta})=\lambda\}
=
\varnothing.
\]
\end{proof}

\begin{corollary}
\label{cr:deform_allbranch}
Under Assumption~\ref{asp1}-\ref{asp3}, the deformation constructed in
Theorem~\ref{thm:global_sigma} may be chosen uniformly for all resonant
analytic branches associated with the energy $\lambda$.
\end{corollary}

\begin{proof}
Since only finitely many resonant branches occur, the result follows by
applying Theorem~\ref{thm:global_sigma} to each branch.
\end{proof}

\subsection{Non-resonant component}

The remaining spectrum consists of infinitely many non-resonant band
functions. Unlike the resonant component, these bands are not treated
individually. Instead, we return to the operator level and introduce
the spectral projection onto the non-resonant subspace. This allows
all non-resonant bands to be handled simultaneously through the
corresponding reduced operator, thereby establishing the uniform
invertibility required for the final deformation theorem.

\subsubsection{Damped quasi-periodic cell problems}

We begin with the damped resolvent equation
\[
(\L-\lambda-\i\varepsilon)u_\varepsilon=f,
\qquad
\varepsilon>0.
\]
Applying the Floquet--Bloch transform yields, for each
$\balpha\in\Bz$, the fiber problem
\[
(\L(\balpha)-\lambda-\i\varepsilon)
w_\varepsilon(\balpha,\cdot)
=
f.
\]

Let
\[
v_\balpha^\varepsilon(\bx)
=
e^{-\i\balpha\cdot\bx}
w_\varepsilon(\balpha,\bx).
\]
Then
$v_\balpha^\varepsilon\in H_0^1(\Omega)$ satisfies
\begin{equation}
\label{eq:qp_weak_damp}
a_\balpha(v_\balpha^\varepsilon,\phi)
-
(\lambda+\i\varepsilon)
(v_\balpha^\varepsilon,\phi)_{L^2(\Omega)}
=
\left\langle
e^{-\i\balpha\cdot\bx}f,\phi
\right\rangle,
\qquad
\forall\phi\in H_0^1(\Omega),
\end{equation}
where
\[
a_\balpha(u,v)
=
\int_\Omega
\Bigl[
A(\nabla+\i\balpha)u
\cdot
(\nabla-\i\balpha)\overline v
+
Vu\overline v
\Bigr]
\,\d\bx.
\]

For every $\varepsilon>0$ and $\balpha\in\Bz$, the operator
$\L(\balpha)-\lambda-\i\varepsilon$ is invertible. Indeed,
injectivity follows immediately by taking the imaginary part of
the $L^2$-inner product, while surjectivity follows from the
Fredholm alternative. Moreover, the analytic dependence of the
fiber operators on $\balpha$ implies that
$v_\balpha^\varepsilon$ depends analytically on $\balpha$.

Let
\[
\{\mu_m(\balpha),\phi_m(\balpha,\cdot)\}_{m\ge1}
\]
denote the Bloch eigenpairs introduced in
Section~\ref{sec:ft}. The fiber solution admits the spectral
representation
\[
w_\varepsilon(\balpha,\cdot)
=
\sum_{m=1}^{\infty}
\frac{\hat f_m(\balpha)}
{\mu_m(\balpha)-\lambda-\i\varepsilon}
\phi_m(\balpha,\cdot),
\]
where
\[
\hat f_m(\balpha)
=
(f,\phi_m(\balpha,\cdot))_{L^2(\Omega)}.
\]
Consequently,
\begin{equation}
\label{eq:u_damped}
u_\varepsilon(\bx)
=
\sum_{m=1}^{\infty}
\int_\Bz
\frac{\hat f_m(\balpha)}
{\mu_m(\balpha)-\lambda-\i\varepsilon}
\phi_m(\balpha,\bx)
\,\d\balpha.
\end{equation}

This representation naturally separates the spectral modes into
those that meet the energy level $\lambda$ and those that remain
away from it, leading to the resonant/non-resonant decomposition
introduced below.

\subsubsection{Spectral decomposition}
\label{sec:res_decomp}

Recall that
\[
J(\lambda)
=
\left\{
m\ge1:
\exists\,\balpha\in\Bz
\text{ such that }
\mu_m(\balpha)=\lambda
\right\}
\]
is finite by Theorem~\ref{th:kuch3}. Accordingly, for each
$\balpha\in\Bz$, we define the resonant spectral subspace
\[
\mathcal H_{\rm res}(\balpha)
=
\operatorname{span}
\left\{
\phi_m(\balpha,\cdot):
m\in J(\lambda)
\right\},
\]
and denote by
\[
P_\lambda(\balpha):
L^2(\Omega)
\longrightarrow
\mathcal H_{\rm res}(\balpha)
\]
the corresponding orthogonal projection. We further set
\[
Q_\lambda(\balpha)
=
I-P_\lambda(\balpha),
\qquad
\mathcal H_{\rm nr}(\balpha)
=
Q_\lambda(\balpha)L^2(\Omega).
\]
Thus
\[
L^2(\Omega)
=
\mathcal H_{\rm res}(\balpha)
\oplus
\mathcal H_{\rm nr}(\balpha).
\]

The spectral representation of the damped fiber solution therefore
decomposes as
\[
w_\varepsilon
=
w_\varepsilon^{\rm res}
+
w_\varepsilon^{\rm nr},
\]
where
\[
w_\varepsilon^{\rm res}(\balpha)
=
P_\lambda(\balpha)w_\varepsilon(\balpha)
=
\sum_{m\in J(\lambda)}
\frac{\hat f_m(\balpha)}
{\mu_m(\balpha)-\lambda-\i\varepsilon}
\phi_m(\balpha),
\]
and
\[
w_\varepsilon^{\rm nr}(\balpha)
=
Q_\lambda(\balpha)w_\varepsilon(\balpha)
=
\sum_{m\notin J(\lambda)}
\frac{\hat f_m(\balpha)}
{\mu_m(\balpha)-\lambda-\i\varepsilon}
\phi_m(\balpha).
\]

The resonant component is handled by the deformation theory of
Section~\ref{sec:complex_Fermi}. The following theorem shows that
the non-resonant component remains uniformly separated from
$\lambda$ under sufficiently small complex perturbations of the
Floquet parameter.

\begin{theorem}[Uniform invertibility of the non-resonant part]
\label{thm:non-resonant-invertibility}

There exists $\delta_{\rm nr}>0$ such that
\[
T_{\balpha,\bomega,t}
:=
Q_\lambda(\balpha)
\bigl(
\L(\balpha+\i t\bomega)-\lambda
\bigr)
Q_\lambda(\balpha)
:
\mathcal H_{\rm nr}(\balpha)
\longrightarrow
\mathcal H_{\rm nr}(\balpha)
\]
is invertible for all $\balpha\in\Bz$, $\bomega\in\mathbb S^{d-1}$ and $|t|\le\delta_{\rm nr}$. 
Moreover, its inverse is uniformly bounded:
\[
\left\|
T_{\balpha,\bomega,t}^{-1}
\right\|
\leq C<\infty,\quad\forall\,
\balpha\in\Bz,\,
\bomega\in\mathbb S^{d-1},\,
|t|\le\delta_{\rm nr}.
\]
\end{theorem}

\begin{proof}

For \(t=0\), the operator
\[
T_{\balpha,\bomega,0}
=
Q_\lambda(\balpha)
\bigl(\L(\balpha)-\lambda\bigr)
Q_\lambda(\balpha)
\]
is invertible on \(\mathcal H_{\rm nr}(\balpha)\) by the
definition of the non-resonant subspace.

Since the family of operators $
T_{\balpha,\bomega,t} $ depends continuously on
\((\balpha,\bomega,t)\), invertibility persists under sufficiently
small perturbations. Therefore, for every
$(\balpha,\bomega)\in \Bz\times\mathbb S^{d-1}$, 
there exist a neighborhood
\(U_{\balpha,\bomega}\) of \((\balpha,\bomega)\) and a constant
\(\delta_{\balpha,\bomega}>0\) such that
\(T_{\bbeta,\v,t}\) is invertible whenever
$(\bbeta,\v)\in U_{\balpha,\bomega}$ and 
$|t|<\delta_{\balpha,\bomega}$.

Since $\Bz\times\mathbb S^{d-1}$ is compact, finitely many such neighborhoods suffice to cover it.
Taking the minimum of the corresponding values of
\(\delta_{\balpha,\bomega}\), we obtain a uniform constant
\(\delta>0\) such that  $T_{\balpha,\bomega,t}$ is invertible for all
$\balpha\in\Bz$, $\bomega\in\mathbb S^{d-1}$ and $|t|\le\delta$.

Finally, the inverse depends continuously on the parameters
wherever the operator is invertible. Since
$
\Bz\times\mathbb S^{d-1}\times[-\delta,\delta]
$
is compact, it follows that $T^{-1}_{\balpha,\bomega,t}$
is uniformly bounded with respect to $\balpha,\bomega,t$.
\end{proof}

Thus, for sufficiently small complex perturbations of the
Floquet parameter, no additional singularities arise from the
non-resonant spectral component. All possible singularities are
therefore associated with the resonant branches meeting the
energy level $\lambda$, which are precisely the branches controlled
by the deformation theory of Section~\ref{sec:complex_Fermi}.

\subsection{Global complex deformation of the fiber representation}
\label{sec:global-fiber-deformation}

We now combine the branchwise deformation theory developed in
Section~\ref{sec:deformation} with the resonant--non-resonant
decomposition established in Section~\ref{sec:res_decomp} to
construct an admissible complex deformation for the full fiber
operator.

Before presenting the construction, we summarize in
Figure~\ref{fig:geometry} the geometric framework developed in the
previous sections. The figure illustrates the Fermi surface together
with the observation direction, the decomposition into the
non-grazing and grazing parts, a tubular neighbourhood, and the local
complexifications near the grazing set. These geometric ingredients
form the basis of the global contour deformation constructed below.

\begin{figure}[t]
    \centering
    \includegraphics[width=0.72\textwidth]{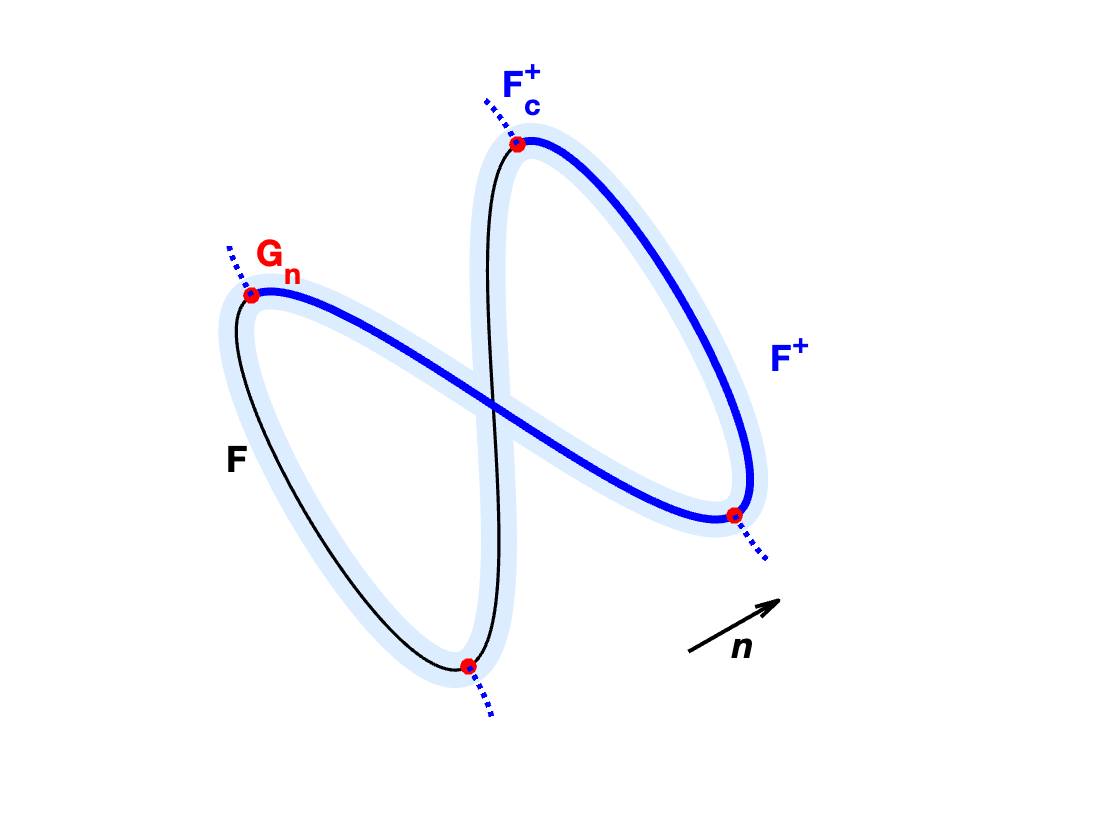}
    \caption{Schematic illustration of the geometric framework underlying the contour deformation. The Fermi surface $\bF_\lambda$ is shown as a solid curve. The non-grazing component $\bF^+$ associated with the direction $\n$ is highlighted by a thick solid curve. Circular markers indicate the grazing set $G_\n$. The shaded region represents a tubular neighbourhood of the Fermi surface $\mathcal{N}_\tau(\bF_\lambda)$, and the dashed curves illustrate local complexifications near the grazing set. The arrow indicates the observation direction $\n$.}
    \label{fig:geometry}
\end{figure}

Compared with the branchwise analysis, the only additional issue is
the presence of band crossings, where several analytic branches may
contribute simultaneously to the same fiber operator. The following
theorem shows that the branchwise deformation can nevertheless be
assembled into a uniform admissible deformation for the full fiber
problem.

\begin{theorem}[Admissible deformation for the fiber problem]
\label{thm:admissible-fiber}

Let $\lambda\in\sigma(\L)$ be a regular frequency.
Then there exist uniform constants
$
\tau>0$ and $
0<\delta_{\rm ext}\le\delta_{\rm int}$,
such that the piecewise constant function
\[
\sigma(\balpha)=
\begin{cases}
\delta_{\rm int},
&
\balpha\in\mathcal N_\tau(\bF_\lambda),\\[2mm]
\delta_{\rm ext},
&
\balpha\in
\Bz\setminus\mathcal N_\tau(\bF_\lambda),
\end{cases}
\]
satisfies the following properties.

\begin{enumerate}
\item[(i)]
The branchwise deformation results established in
Section~\ref{sec:deformation} hold simultaneously for every local
resonant analytic branch.

\item[(ii)]
The non-resonant component
\[
\L_{\rm nr}(\balpha+\i t\bomega)-\lambda I
\]
remains uniformly invertible for all
$
0\le t\le\sigma(\balpha)$,
$
\balpha\in\Bz$  and $
\bomega\in\mathbb S^{d-1}$.
\end{enumerate}

Consequently, each local resonant analytic branch contributes at
most one singular point of the fiber resolvent along the deformation
ray, while the non-resonant component contributes none.
Moreover,
\[
\L(\balpha+\i\sigma(\balpha)\bomega)-\lambda I
\]
is invertible for every
$\balpha\in\Bz$
and
$\bomega\in\mathbb S^{d-1}$.
Hence
\[
\Bz_\bomega
=
\{
\balpha+\i\sigma(\balpha)\bomega:
\balpha\in\Bz
\}
\]
is an admissible deformation for the fiber resolvent.
\end{theorem}

\begin{proof}

We first treat the non-resonant component.
By Theorem~\ref{thm:non-resonant-invertibility},
there exists a constant
$
\delta_{\rm nr}>0
$
such that
\[
\L_{\rm nr}(\balpha+\i t\bomega)-\lambda I
\]
is invertible for every
$
\balpha\in\Bz$, $
0\le t\le\delta_{\rm nr}$ and $
\bomega\in\mathbb S^{d-1}.
$

Next we consider the resonant component.
By Corollary~\ref{cr:deform_allbranch},
there exist constants
$\tau>0$ and $\delta_{\rm int}\geq\delta_{\rm ext}>0$,
such that every local resonant analytic branch admits the deformation
constructed in Section~\ref{sec:deformation}.
Since only finitely many resonant branches occur, these constants may be
chosen uniformly.
Replacing $\delta_{\rm int}$ by
\[
\min\{\delta_{\rm int},\delta_{\rm nr}\}
\]
and shrinking $\tau $ if necessary. Since $\overline{\Bz\setminus\mathcal{N}_\tau(\bF_\lambda)}$ is compact, by perturbation theory, we can again find a $0<\delta_{\rm ext}\leq\delta_{\rm int}$ such that $\mathcal{L}(\balpha+\i t \bomega)-\lambda I$ is uniformly invertible for all $\balpha\in \overline{\Bz\setminus\mathcal{N}_\tau(\bF_\lambda)}$, $\bomega\in\mathbb{S}^{d-1}$ and $|t|\leq \delta_{\rm ext}$.

Define
\[
\sigma(\balpha)=
\begin{cases}
\delta_{\rm int},
&
\balpha\in\mathcal N_\tau(\bF_\lambda),
\\[2mm]
\delta_{\rm ext},
&
\balpha\notin\mathcal N_\tau(\bF_\lambda).
\end{cases}
\]

Then, by construction, the branchwise deformation theorem applies to every
resonant analytic branch, while the non-resonant component remains uniformly
invertible along the entire deformation. Hence
\[
\L(\balpha+\i\sigma(\balpha)\bomega)-\lambda I
\]
is invertible for every
\(
\balpha\in\Bz
\)
and
\(
\bomega\in\mathbb S^{d-1}.
\)
Therefore
\[
\Bz_\bomega
=
\{
\balpha+\i\sigma(\balpha)\bomega:
\balpha\in\Bz
\}
\]
is an admissible deformation for the fiber resolvent.

\end{proof}

\begin{remark}
The admissible deformation obtained in
Theorem~\ref{thm:admissible-fiber}
is, in general, necessarily non-constant.
Indeed, a constant deformation height
\[
\sigma(\balpha)\equiv\delta>0
\]
cannot avoid all complexified singularities.
 This is one of the essential differences between the higher-dimensional problems and the one-dimensional
setting studied in our previous work \cite{Zhang2019a}.

In one dimension, the Fermi surface consists of isolated
points and there is no grazing set. Consequently, a
constant deformation height is sufficient to avoid all
complexified singularities. In contrast, in dimensions
\(d\ge2\), the grazing set \(G_\n\) forms a codimension-one
stratum of the Fermi surface and generates the branching
phenomena described in Section~\ref{sec:dift}. The
associated singular geometry forces the admissible
deformation to depend on the location of \(\alpha\).

This phenomenon is already visible for the free Laplacian
in two dimensions. Indeed,
\[
\mu_{\bm j}(\balpha)
=
|\bm j+\balpha|^2,
\qquad
\bm j\in\mathbb Z^2.
\]
Assume that a constant deformation
\(
\sigma(\balpha)\equiv\delta>0
\)
were admissible. Writing
\[
\bm j+\balpha
=
(R\cos\theta,R\sin\theta),
\qquad
\n
=
(\cos\phi,\sin\phi),
\]
one obtains
\[
\mu_{\bm j}(\balpha+\i\delta\n)
=
R^2-\delta^2
+
2\i R\delta\cos(\theta-\phi).
\]
Note that
\[
\cos(\theta-\phi)=0
\quad\Leftrightarrow\quad
(\bm j+\balpha)\cdot\n=0\quad\Leftrightarrow\quad
\nabla\mu_{\bm j}(\balpha)\cdot\n=0.
\]
Thus this condition precisely characterizes the grazing set
\(G_\n\).

Choosing
\[
R^2=\lambda+\delta^2,
\qquad
\cos(\theta-\phi)=0,
\]
gives
\[
\mu_{\bm j}(\balpha+\i\delta\n)
=
\lambda.
\]
Hence the deformed contour intersects the complexified
Fermi surface exactly along grazing directions.

This shows that the obstruction to a constant deformation
height is generated by the grazing geometry. In one
dimension, where no grazing set exists, a constant
deformation is possible. By contrast, in dimensions
\(d\ge2\), the presence of \(G_{\n}\)
forces the deformation function \(\sigma(\balpha)\) to
depend on the location of \(\balpha\).
\end{remark}

\begin{proposition}[Non-resonant limiting absorption principle]
\label{prop:non-resonant-lap}
Fix $\n\in\mathbb{S}^{d-1}$ and let
\[
\Bz_{\bm n}
=
\Bigl\{
\balpha+\i\sigma(\balpha)\bm n:
\balpha\in\Bz
\Bigr\}
\]
be the admissible deformation constructed in
Theorem~\ref{thm:admissible-fiber}. Then
\[
Q_\lambda(\zeta)
\bigl(
\L(\zeta)-\lambda-\i\varepsilon
\bigr)^{-1}
Q_\lambda(\zeta)\rightarrow Q_\lambda(\zeta)
\bigl(
\L(\zeta)-\lambda
\bigr)^{-1}
Q_\lambda(\zeta),\quad\epsilon\rightarrow 0^+
\]
converges uniformly on \(\Bz_{\bm n}\).

\end{proposition}

\begin{proof}

By Theorem~\ref{thm:non-resonant-invertibility},
$
Q_\lambda(\zeta)
\bigl(
\L(\zeta)-\lambda
\bigr)
Q_\lambda(\zeta)
$
is uniformly invertible on \(\Bz_{\bm n}\).
Hence both
\[
Q_\lambda(\zeta)
(\L(\zeta)-\lambda-\i\varepsilon)^{-1}
Q_\lambda(\zeta)\quad
\text{ 
and
 }\quad
Q_\lambda(\zeta)
(\L(\zeta)-\lambda)^{-1}
Q_\lambda(\zeta)
\]
are uniformly bounded for sufficiently small
\(\varepsilon\). 
The claim follows from the resolvent identity
\[
(\L-\lambda-\i\varepsilon)^{-1}
-
(\L-\lambda)^{-1}
=
\i\varepsilon
(\L-\lambda-\i\varepsilon)^{-1}
(\L-\lambda)^{-1}.
\]

\end{proof}

Proposition~\ref{prop:non-resonant-lap}
shows that the non-resonant component is completely
regular under the limiting absorption process.
Hence the entire singular structure of the fiber resolvent is carried by the resonant component.
The subsequent sections are devoted exclusively to the
analysis of these resonant contributions.

\section{The Limiting Absorption Principle}
\label{sec:lap}

The admissible deformation constructed in Section~\ref{sec:deform} removes all singularities from the deformed contour while recording the resonant contributions through the enclosed poles. Consequently, the limiting absorption solution can be recovered by contour deformation and the residue theorem.
With the directional organization of the spectral representation established in the previous sections, the limiting absorption solution can now be analyzed slice by slice.

Fix an observation direction
$\n\in\mathbb{S}^{d-1}$. The SSA decomposes the Brillouin zone
into one-dimensional slices parallel to $\n$. This reduces
the Floquet--Bloch representation
\[
u_\epsilon(\bx)
=
\int_{\Bz}
w_\epsilon(\balpha,\bx)\,\d\balpha
\]
to a family of contour integrals. The admissible
deformation constructed in the previous sections separates
the analytic and singular contributions: the non-resonant
component is deformed analytically, while the resonant
component is recovered through residue contributions.

The section is organized as follows. We first derive the
contour deformation formulas for both the non-resonant and
resonant components and rewrite the residue contribution as
an integral over the complexified Fermi surfaces. We then
pass to the limit $\epsilon\to0^+$ to establish the
limiting absorption principle. Finally, we analyze the
resulting solution by decomposing it into its evanescent,
propagating, and tangential components. This section not only establishes the limiting absorption principle but also derives the geometric decomposition that forms the basis of the subsequent asymptotic analysis.

\subsection{Contour deformation}

For each slice of the SSA decomposition, the contour
deformation naturally separates into the non-resonant and
resonant components. The former is treated by analytic
continuation, whereas the latter gives rise to residue
contributions associated with the resonant bands. In this
subsection, we derive the corresponding contour
representation and rewrite the residue contribution as an
integral over the complexified Fermi surfaces.

\subsubsection{The non-resonant contribution}

Fix an observation direction
\(
\n\in\mathbb S^{d-1}
\)
and write
\[
\balpha=\bgamma\t+s\n,
\]
where
\(
\bgamma\in D
\)
denotes the transverse variables. The corresponding slice
decomposition is
\[
\int_{\Bz}F(\balpha)\,\d\balpha
=
\int_D
\int_{\ell(\bgamma)}
F(\bgamma\t+s\n)
\,\d s\,\d\bgamma .
\]
where for each slice,
\[
\ell(\bgamma)
=
[\ell_l(\bgamma),\ell_r(\bgamma)]
\]
denotes the original contour. We define the deformed contour
\[
\ell_\n(\bgamma)
=
\Bigl\{
s+\i\sigma(\bgamma\t+s\n):
s\in(\ell_l(\bgamma),\ell_r(\bgamma))
\Bigr\},
\]
together with the region
\[
R(\bgamma)
=
\Bigl\{
s+\i\delta:
s\in(\ell_l(\bgamma),\ell_r(\bgamma)),
\;
0<\delta<\sigma(\bgamma\t+s\n)
\Bigr\},
\]
bounded by $\ell(\bgamma)$, $\ell_\n(\bgamma)$, $C_\bgamma^1$ and $C_\bgamma^2$ (for the definitions see \eqref{eq:left_C}-\eqref{eq:right_C}).

The non-resonant component is obtained by projecting onto
the complement of the resonant eigenspaces. Writing
\[
Q_\lambda(\balpha)
=
I-P_\lambda(\balpha),
\]
we define
\[
w_\epsilon^{\rm nr}(\balpha,\bx)
=
(\L_\alpha-\lambda-\i\epsilon)^{-1}
Q_\lambda(\balpha)f(\balpha,\cdot).
\]

\begin{theorem}[Non-resonant reduction]
\label{thm:nonres}

For every sufficiently small $\epsilon>0$, the non-resonant
component
$
w_\epsilon^{\rm nr}(\balpha,\cdot)
$
extends analytically to a neighborhood of the deformed
Brillouin zone $\Bz_\n$. Consequently,
\begin{equation}
\label{eq:con_def_nr_eps}
u_\epsilon^{\rm nr}(\bx)
=
\int_{\Bz_\n}
w_\epsilon^{\rm nr}(\balpha,\bx)\,\d\balpha,
\end{equation}
and the contour deformation produces no residue
contribution.

\end{theorem}

\begin{remark}
Since the deformation function \(\sigma(\balpha)\) is only
piecewise constant, the deformed Brillouin zone
\[
\Bz_\n
=
\{
\balpha+\i\sigma(\balpha)\n:
\balpha\in\Bz
\}
\]
is understood through its piecewise-translational
parametrization. Accordingly, integrals over \(\Bz_\n\) are
defined chartwise and summed over the corresponding
pieces.
\end{remark}

\begin{proof}

Applying the slice decomposition gives
\[
u_\epsilon^{\rm nr}(\bx)
=
\int_D
\int_{\ell(\bgamma)}
w_\epsilon^{\rm nr}
(\bgamma\t+s\n,\bx)
\,\d s\,\d\bgamma .
\]
From Proposition~\ref{prop:non-resonant-lap}, the map
\[
s\longmapsto
w_\epsilon^{\rm nr}
(\bgamma\t+s\n,\bx)
\]
extends analytically to a neighborhood of
\(R(\bgamma)\).
Therefore,  Cauchy's theorem yields
\[
\oint_{\partial R(\bgamma)}
w_\epsilon^{\rm nr}
(\bgamma\t+s\n,\bx)\,ds
=
0.
\]
Hence
\[
\int_{\ell(\bgamma)}
w_\epsilon^{\rm nr}
(\bgamma\t+s\n,\bx)\,\d s
=
\int_{\ell_\n(\bgamma)}
w_\epsilon^{\rm nr}
(\bgamma\t+s\n,\bx)\,\d s
+
\Bigl(
\int_{C_\bgamma^1}
-
\int_{C_\bgamma^2}
\Bigr)
w_\epsilon^{\rm nr}
(\bgamma\t+s\n,\bx)\,\d s.
\]
The two boundary contributions cancel pairwise by the
periodicity of the Floquet--Bloch fibers. The cancellation of the boundary terms is proved in Appendix B.2. Integrating with
respect to \(\bgamma\) yields \eqref{eq:con_def_nr_eps}.

\end{proof}

\subsubsection{The resonant contribution}

We now turn to the resonant component. Unlike the
non-resonant part, the contour deformation crosses the
poles associated with the resonant bands. Consequently,
the contour deformation is governed by the residue theorem.

Fix \(j\in J(\lambda)\). The enclosed poles are localized
near the outgoing portion of the \(j\)-th Fermi sheet. We
first derive a residue representation for the pole
contribution, and then rewrite it as an integral over the
complexified outgoing Fermi surfaces.

\begin{theorem}[Residue representation]
\label{thm:residue_representation}

For every sufficiently small $\epsilon>0$,
\[
u_{j,\epsilon}^{\rm res}(\bx)
=
u_{j,\epsilon}^{\rm reg}(\bx)
+
u_{j,\epsilon}^{\rm pole}(\bx),
\]
where
\[
u_{j,\epsilon}^{\rm reg}(\bx)
=
\int_D
\int_{\ell_\n(\bgamma)}
w_{j,\epsilon}^{\rm res}
(\bgamma\t+s\n,\bx)
\,\d s\,\d\bgamma,
\]
and
\[
u_{j,\epsilon}^{\rm pole}(\bx)
=
2\pi \i
\int_D
\sum_{m=1}^{N(\bgamma)}
{\rm{Res}}
\Bigl(
w_{j,\epsilon}^{\rm res},
z_m(\bgamma,\epsilon)
\Bigr)
\,\d\bgamma .
\]

\end{theorem}

\begin{proof}

Fix \(\bgamma\in D\). By
Lemma~\ref{lm:finite_inter}, the slice
\(\ell(\bgamma)\) intersects the  Fermi surface at only
finitely many points. For each such intersection,
Proposition~\ref{prop:local_crossing} yields at most one
corresponding pole enclosed by the contour
\(\partial R(\bgamma)\). Hence
\(\partial R(\bgamma)\) encloses only finitely many
isolated poles. Applying the residue theorem gives
\[
\int_{\ell(\bgamma)}
w_{j,\epsilon}^{\rm res}
(\bgamma\t+s\n,\bx)\,\d s
=
\int_{\ell_\n(\bgamma)}
w_{j,\epsilon}^{\rm res}
(\bgamma\t+s\n,\bx)\,\d s
+
2\pi\i
\sum_{m=1}^{N(\bgamma)}
\operatorname{Res}
\Bigl(
w_{j,\epsilon}^{\rm res},
z_m(\bgamma,\epsilon)
\Bigr)
+
B(\bgamma),
\]
where \(B(\bgamma)\) denotes the contribution from the two
vertical boundary segments.

The boundary integral \(B(\bgamma)\) again cancel pairwise by the periodicity
of the Floquet--Bloch fibers as shown in Appendix~\ref{sec:bdr}. 
Integrating the  identity with respect to
\(\bgamma\in D\) yields
\[
u_{j,\epsilon}^{\rm res}
=
u_{j,\epsilon}^{\rm reg}
+
u_{j,\epsilon}^{\rm pole},
\]
which completes the proof.

\end{proof}

The residue representation obtained above is still expressed
in terms of the slice parameter $\bgamma$. To identify the
outgoing contribution with the geometry of the Fermi
surface, we next rewrite the residue sum as an integral over
the complexified outgoing Fermi sheets. This requires three
ingredients: a global parametrization of the sheets, the
associated surface measure, and a uniform estimate for the
resulting surface density.

The local parametrizations obtained in
Theorem~\ref{th:finite-param-outgoing} are now assembled
into a finite atlas covering the outgoing real and
complexified Fermi surfaces.

Fix \(j\in J(\lambda)\). Let
$
\bS_{j,1},\dots,\bS_{j,K_j}
$
denote the relevant analytic branches of the band function
\(\mu_j\) whose outgoing Fermi sheets intersect the energy
surface \(\{\mu_j=\lambda\}\). Applying
Theorem~\ref{th:finite-param-outgoing}
to each branch
\(\bS_{j,m}\),
we obtain finitely many outgoing charts
\[
\Gamma_{j,m,\ell}^+,
\qquad
\ell=1,\dots,N_{j,m},
\]
covering the corresponding real outgoing sheet.
Among these,
\[
\ell=1,\dots,M_{j,m},
\]
are precisely the charts intersecting the grazing set, and
each admits an associated complex chart
\[
\Gamma_{j,m,\ell}^-,
\qquad
\ell=1,\dots,M_{j,m},
\]
covering the corresponding component of the complexified
outgoing sheet.

More precisely, for each chart there exist orthonormal
frames 
\[
\left\{\t_{j,m,\ell}^{(1)},\ldots,\t_{j,m,\ell}^{(d-1)},\n_{j,m,\ell}\right\},
\]
with parameter domains
\[
V_{j,m,\ell}^{\pm}\subset\R^{d-1},
\]
(the domains \(V_{j,m,\ell}^{-}\) being present only for the
grazing charts), and local branches
\[
s_{j,m,\ell}(\bgamma,\epsilon),
\]
such that
\[
\Sigma_{j,m,\ell}^{\pm}(\epsilon)
=
\Bigl\{
\bgamma\cdot\t_{j,m,\ell}
+
s_{j,m,\ell}(\bgamma,\epsilon)\n_{j,m,\ell}
:
\bgamma\in V_{j,m,\ell}^{\pm}
\Bigr\},
\]
and the complexified outgoing sheet admits the finite
decomposition
\begin{equation}
\label{eq:par_Sigma}
\Sigma_j^+(\epsilon)
=
\bigcup_{m=1}^{K_j}
\left(
\bigcup_{\ell=1}^{N_{j,m}}
\Sigma_{j,m,\ell}^{+}(\epsilon)
\cup
\bigcup_{\ell=1}^{M_{j,m}}
\Sigma_{j,m,\ell}^{-}(\epsilon)
\right).
\end{equation}

Passing to the limit as \(\epsilon\to0^+\), the
corresponding real and complex outgoing Fermi surfaces are
covered by
\begin{equation}
\label{eq:par_F}
\bF_j^+
=
\bigcup_{m=1}^{K_j}
\bigcup_{\ell=1}^{N_{j,m}}
\Gamma_{j,m,\ell}^{+},
\qquad
\bF_{j,c}^{+}
=
\bigcup_{m=1}^{K_j}
\bigcup_{\ell=1}^{M_{j,m}}
\Gamma_{j,m,\ell}^{-},
\end{equation}
where
\[
\Gamma_{j,m,\ell}^{\pm}
=
\Bigl\{
\bgamma\cdot\t_{j,m,\ell}
+
s_{j,m,\ell}(\bgamma,0)\n_{j,m,\ell}
:
\bgamma\in V_{j,m,\ell}^{\pm}
\Bigr\}.
\]

With the finite parametrization of the complexified outgoing
sheets in hand, we next introduce the corresponding surface
measure.

\begin{definition}[Parametrized surface measure on complexified sheets]
\label{def:param_surface_measure}

The finite parametrization constructed above naturally
induces a surface measure on each component of the
complexified outgoing Fermi surface. This allows the
residue representation to be rewritten as an integral over
the complexified sheets.

Let
\[
\Sigma
=
\{\Phi(\bgamma):\bgamma\in V\}\quad\text{where }\quad  \Phi(\bgamma)
=
\bgamma\t+s(\bgamma)\n
\]
be one of the parametrized components of
\(\Sigma_j^+(\epsilon)\). 
Write
$
s=s_1+i s_2,
$
where \(s_1,s_2:V\to\mathbb R\). We regard
\(\Sigma\) as a real \((d-1)\)-dimensional submanifold of
\(\mathbb R^{d+1}\) via the embedding
\[
\Phi(\bgamma)
\equiv
(\bgamma,s_1(\bgamma),s_2(\bgamma)).
\]
The associated surface Jacobian is given by the Gram
determinant
\[
J_\Phi(\bgamma)
=
\sqrt{
1
+
|\nabla_\bgamma s_1|^2
+
|\nabla_\bgamma s_2|^2
+
|\nabla_\bgamma s_1|^2
|\nabla_\bgamma s_2|^2
-
\bigl(
\nabla_\bgamma s_1
\cdot
\nabla_\bgamma s_2
\bigr)^2
}.
\]
The induced surface measure is defined by
\[
\d S(\balpha)
=
J_\Phi(\bgamma)\,\\d\bgamma,
\qquad
\balpha=\Phi(\bgamma),
\]
so that
\begin{equation}
\label{eq:part_complex_sheet}
\int_\Sigma
G(\balpha)\,\d S(\balpha)
:=
\int_V
G(\Phi(\bgamma))
J_\Phi(\bgamma)\,\d\bgamma.
\end{equation}

For the finite parametrization
\eqref{eq:par_Sigma}, the surface integral is understood
chartwise and summed over all parametrized components.

\end{definition}

Before deriving the surface representation, we establish a uniform estimate for the coefficient arising in the change of variables from the slice parametrization to the surface parametrization.

\begin{lemma}[Uniform control of the residue density]
\label{lem:est_denominator}
Let Assumption \ref{asp1} hold. Let
\[
\Phi_\epsilon(\bgamma)
=
\bgamma \t+s(\bgamma,\epsilon)\n
\]
be a local parametrization of a branch of
\(
\Sigma^+(\epsilon)
\).
Then, for some sufficiently small \(\epsilon_0>0\),
\[
|\partial_s\mu(\Phi_\epsilon(\bgamma))|
\,J_{\Phi_\epsilon}(\bgamma)
\ge
|\nabla\mu(\Phi_\epsilon(\bgamma))|
\ge C,
\]
uniformly for \(0<\epsilon\le\epsilon_0\).

\end{lemma}

\begin{proof}
For each \(0<\epsilon\le\epsilon_0\), by definition,
\[
\mu(\Phi_\epsilon(\bgamma))=\mu(\bgamma\t+s(\bgamma,\epsilon)\n)
=
\lambda+i\epsilon.
\]
Differentiating this identity with respect to \(\bgamma\) gives
\[
\nabla_\bgamma\mu(\Phi_\epsilon(\bgamma))
+
\partial_s\mu(\Phi_\epsilon(\bgamma))
\,\nabla_\bgamma s(\bgamma,\epsilon)
=
0.
\]
Since the poles are simple for every fixed
\(\epsilon>0\),
\(
\partial_s\mu(\Phi_\epsilon(\bgamma))
\neq0,
\)
and hence
\[
\nabla_\bgamma s(\bgamma,\epsilon)
=
-
\frac{
\nabla_\bgamma\mu(\Phi_\epsilon(\bgamma))
}{
\partial_s\mu(\Phi_\epsilon(\bgamma))
}.
\]

By Definition~\ref{def:param_surface_measure},
\[
J_{\Phi_\epsilon}(\bgamma)
\ge
\sqrt{
1+
|\nabla_\bgamma s(\bgamma,\epsilon)|^2
}.
\]
According to Assumption ~\ref{asp1}, the following estimation holds uniformly for
\(0\le\epsilon\le\epsilon_0\):
\[
\begin{aligned}
\bigl|
\partial_s\mu(\Phi_\epsilon(\bgamma))
\bigr|
J_{\Phi_\epsilon}(\bgamma)
&\ge
\sqrt{
\bigl|
\partial_s\mu(\Phi_\epsilon(\bgamma))
\bigr|^2
+
\bigl|
\nabla_\bgamma\mu(\Phi_\epsilon(\bgamma))
\bigr|^2
}
\\
&=
|\nabla\mu(\Phi_\epsilon(\bgamma))|\ge
C>0.
\end{aligned}
\]
Thus, even though
\(
\partial_s\mu(\Phi_\epsilon(\bgamma))
\)
may approach zero as
\(\epsilon\to0^+\),
the growth of the surface Jacobian exactly compensates for this loss of transversality. Therefore the coefficient appearing in the surface representation remains uniformly bounded.
\end{proof}

The local surface representations obtained on the
individual charts combine into the following global
representation over the complexified outgoing Fermi
surface.

\begin{theorem}[Fermi-surface representation of the pole contribution]
\label{thm:pole_surface}

For every sufficiently small \(\epsilon>0\), the pole
contribution admits the following surface representation:
\[
u_{j,\epsilon}^{\rm pole}(\bx)
=
2\pi \i
\int_{\Sigma^+(\epsilon)}
\frac{
\hat f_j(\balpha)\,
\phi_j(\balpha,\bx)
}
{
\bigl(
\nabla\mu_j(\balpha)\cdot\n^{(\ell)}
\bigr)
J_\Phi(\bgamma)
}
\,\d S(\balpha),
\]
where \(dS\) denotes the parametrized surface measure
introduced in Definition~\ref{def:param_surface_measure}.

\end{theorem}

\begin{proof}
Recall the finite parametrization of
\(\Sigma^+(\epsilon)\).
On each chart, the residue representation is transformed
into a surface integral by the associated
parametrization. For each chart, let
\[
\Phi_{\ell,\epsilon}(\bgamma)
=
\bgamma\t^{(\ell)}
+
s_{j,\ell}(\bgamma,\epsilon)\n^{(\ell)}.
\]

On each chart, the residue at the simple pole
\(s_{j,\ell}(\bgamma,\epsilon)\) is given by
\[
\operatorname{Res}
\left(
\frac{
\hat f_j(\bgamma\t+s\n)\,
\phi_j(\bgamma\t+s\n,\bx)
}
{
\mu_j(\bgamma\t+s\n)-\lambda-\i\epsilon
},
\,s_{j,\ell}(\bgamma,\epsilon)
\right)
=
\frac{
\hat f_j(\balpha)\,
\phi_j(\balpha,\bx)
}
{
\partial_s\mu_j(\balpha)
},
\qquad
\balpha
=
\Phi_{\ell,\epsilon}(\bgamma).
\]

Substituting this identity into the residue representation
and applying the change of variables induced by
\(\Phi_{\ell,\epsilon}\), together with
Definition~\ref{def:param_surface_measure}, we obtain
\begin{align*}
u_{j,\epsilon}^{\rm pole}(\bx)
&=
2\pi \i
\sum_{m=1}^{K_j}
\sum_{\ell=1}^{N_{j,m}}
\int_{V_\ell}
\frac{
\hat f_j(\balpha)\,
\phi_j(\balpha,\bx)
}
{
\nabla\mu_j(\balpha)\cdot\n^{(\ell)}
}
\,\d\bgamma
\\
&=
2\pi \i
\sum_{m=1}^{K_j}
\sum_{\ell=1}^{N_{j,m}}
\int_{\Sigma_{j,m,\ell}^+(\epsilon)}
\frac{
\hat f_j(\balpha)\,
\phi_j(\balpha,\bx)
}
{
\bigl(
\nabla\mu_j(\balpha)\cdot\n^{(\ell)}
\bigr)
J_\Phi(\bgamma)
}
\,\d S(\balpha).
\end{align*}

Corollary~\ref{th:cor:kuch1} and
Lemma~\ref{lem:est_denominator}
guarantee that each chartwise surface integral is well
defined.
Since the outgoing complex Fermi surface is covered by
finitely many parametrized charts, these chartwise
integrals combine into a single integral over
\(\Sigma^+(\epsilon)\).
Hence the local surface representations glue together to
define the desired global integral over
\(\Sigma^+(\epsilon)\).
\end{proof}

\subsection{The final LAP solution}

From the previous subsections, the LAP solution has been
decomposed into a non-resonant part, a regular resonant
part, and a pole contribution,
\[
u_\epsilon(\bx)
=
u_\epsilon^{\rm nr}(\bx)
+
\sum_{j\in J(\lambda)}
\Bigl(
u_{j,\epsilon}^{\rm reg}(\bx)
+
u_{j,\epsilon}^{\rm pole}(\bx)
\Bigr).
\]
We now pass to the limit
\(\epsilon\to0^+\)
for each component separately and combine the resulting
limits to obtain the limiting absorption solution.

\begin{theorem}[Limit of the regular contributions]
\label{thm:regular_limit}
Assume Assumptions~\ref{asp1}--\ref{asp3}.
Then the non-resonant contribution and the regular part of
the resonant contribution converge as
\(\epsilon\to0^+\):
\[
\lim_{\epsilon\to0^+}
u_\epsilon^{\rm nr}(\bx)
=
\int_{\Bz_\n}
w^{\rm nr}(\balpha,\bx)\,\d\balpha,
\]
and
\[
\lim_{\epsilon\to0^+}
u_{j,\epsilon}^{\rm reg}(\bx)
=
\int_{\Bz_\n}
w_j^{\rm res}(\balpha,\bx)\,\d\balpha.
\]

\end{theorem}

\begin{proof}
Both limits follow from the same argument.
By Proposition~\ref{prop:non-resonant-lap}, the integrands
are uniformly bounded on \(\Bz_\n\).
Since \(\Bz_\n\) is bounded, the dominated convergence theorem
permits the exchange of limit and integration.
This proves both assertions.

\end{proof}

The limiting non-resonant and regular resonant
contributions naturally combine into the evanescent part
of the solution.

\begin{definition}[Evanescent contribution]
\label{def:evan}
The evanescent contribution is given by
\[
u^{\rm evan}(\bx)
:=
\int_{\Bz_\n}
w^{\rm nr}(\balpha,\bx)\,\d\balpha
+
\sum_{j\in J(\lambda)}
\int_{\Bz_\n}
w_j^{\rm res}(\balpha,\bx)\,\d\balpha.
\]
Since
\[
w(\balpha,\bx)
=
w^{\rm nr}(\balpha,\bx)
+
\sum_{j\in J(\lambda)}
w_j^{\rm res}(\balpha,\bx),
\]
we equivalently have
\[
u^{\rm evan}(\bx)
=
\int_{\Bz_\n}
w(\balpha,\bx)\,\d\balpha.
\]

\end{definition}

Before passing to the limit in the pole contribution, we
first verify that the limiting surface integral is well
defined.

\begin{lemma}[Well-definedness of the limiting pole integral]
\label{lem:pole_limit_welldefined}

The  integral is well defined:
\[
\int_{\bF_j^+\cup\bF_{j,c}^+}
\frac{
\hat f_j(\balpha)\,
\phi_j(\balpha,\bx)
}
{
\partial_s\mu_j(\balpha)\,
J(\balpha)
}
\,\d S(\balpha).
\]

\end{lemma}
 
\begin{proof}
By Lemma~\ref{lem:est_denominator},
\[
|\partial_s\mu_j(\balpha)|J(\balpha)
\ge
|\nabla\mu_j(\balpha)|
\ge c>0 .
\]
Hence
\[
\left|
\frac{
\hat f_j(\balpha)\phi_j(\balpha,\bx)
}
{
\partial_s\mu_j(\balpha)J(\balpha)
}
\right|
\le
\frac{
|\hat f_j(\balpha)\phi_j(\balpha,\bx)|
}
{c}.
\]
Since
\(
\hat f_j(\balpha)\phi_j(\balpha,\bx)
\)
is continuous on
\(
\bF_j^+\cup\bF_{j,c}^+
\),
the integrand is bounded.

Moreover,
\(
\bF_j^+\cup\bF_{j,c}^+
\)
is a finite union of parametrized compact sheets and therefore has finite surface measure. Consequently the integral is finite.

\end{proof}

With the well-definedness established, we may now pass to
the limit in the pole contribution.

\begin{theorem}[Limit of the pole contribution]
\label{thm:pole_lap}

One has
\[
\lim_{\epsilon\to0^+}
u_{j,\epsilon}^{\rm pole}(\bx)
=
u_j^{\rm pole}(\bx),
\]
where
\[
u_j^{\rm pole}(\bx)
:=
2\pi \i
\int_{\bF_j^+\cup\bF_{j,c}^+}
\frac{
\hat f_j(\balpha)\,
\phi_j(\balpha,\bx)
}
{
\partial_s\mu_j(\balpha)\,
J(\balpha)
}
\,\d S(\balpha).
\]

\end{theorem}

\begin{proof}

Using the finite parametrization of
\(
\Sigma^+_j(\epsilon)
\),
we may write
\[
u_{j,\epsilon}^{\rm pole}(\bx)
=
2\pi\i
\sum_{m=1}^{K_j}
\sum_{\ell=1}^{N_{j,m}}
\int_{V_{j,m,\ell}}
\frac{
\hat f_j(\Phi_{\ell,\epsilon}(\bgamma))
\phi_j(\Phi_{\ell,\epsilon}(\bgamma),\bx)
}
{
\partial_s\mu_j(\Phi_{\ell,\epsilon}(\bgamma))
}
\,\d\bgamma .
\]

Since
\(
\Phi_{\ell,\epsilon}(\bgamma)
\to
\Phi_{\ell,0}(\bgamma)
\)
for almost every
\(
\bgamma\in V_\ell
\),
the integrands converge pointwise almost everywhere to
\[
\frac{
\hat f_j(\Phi_{\ell,0}(\bgamma))
\phi_j(\Phi_{\ell,0}(\bgamma),\bx)
}
{
\partial_s\mu_j(\Phi_{\ell,0}(\bgamma))
}.
\]
By Lemma~\ref{lem:pole_limit_welldefined}, it is integrable on
\(
\bF_j^+\cup\bF_{j,c}^+
\).
Therefore the dominated convergence theorem completes the proof.
\end{proof}

Collecting the previous results, the limiting absorption
solution admits the decomposition
\begin{eqnarray}
\label{eq:evan}
u^{\rm evan}(\bx)
&=&
\int_{\Bz_\n}
w(\balpha,\bx)\,\d\balpha,
\\
\label{eq:pole}
u^{\rm pole}(\bx)
&=&
2\pi i
\sum_{j\in J(\lambda)}
\int_{\bF_j^+\cup\bF_{j,c}^+}
\frac{
\hat f_j(\balpha)\,
\phi_j(\balpha,\bx)
}
{
\partial_s\mu_j(\balpha)\,
J(\balpha)
}
\,\d S(\balpha).
\end{eqnarray}

\begin{theorem}[The limiting absorption solution]
\label{thm:LAP_final}

Assume Assumptions~\ref{asp1}--\ref{asp3}.
Then
\[
u_\epsilon \to u
\qquad
\text{in }
H^1_{\rm loc}(\R^d),\quad
\epsilon\to0^+.
\]
Moreover,
\[
u
=
u^{\rm evan}
+
u^{\rm pole},
\]
where \(u^{\rm evan}\) and \(u^{\rm pole}\) are given by
\eqref{eq:evan} and \eqref{eq:pole}, respectively.

For every compact set \(K\subset\R^d\),
there exists a constant \(C(K)>0\) such that
\[
\|u^{\rm evan}\|_{H^1(K)}
+
\|u^{\rm pole}\|_{H^1(K)}
\le
C(K)\,
\|f\|_{L^2(\R^d)}.
\]

\end{theorem}

\begin{proof}

By Theorem~\ref{thm:regular_limit},
Definition~\ref{def:evan}, and
Theorem~\ref{thm:pole_lap},
the limiting solution admits the decomposition
\[
u
=
u^{\rm evan}
+
u^{\rm pole}.
\]

The convergence
\[
u_\epsilon\to u
\qquad
\text{in }
H^1_{\rm loc}(\R^d)
\]
follows by combining the convergence of the evanescent
and pole contributions. Since both limits are obtained by
dominated convergence together with the corresponding
uniform local \(H^1\)-bounds, the same argument applies to
the first-order derivatives with respect to \(\bx\).

Finally, the asserted local \(H^1\)-estimate follows by
combining the corresponding estimates for
\(u^{\rm evan}\) and \(u^{\rm pole}\).
This completes the proof.

\end{proof}

\subsection{Asymptotic decomposition of the LAP solution}
\label{sec:asym_decomp}

To analyze the far-field behavior of the limiting
absorption solution, we further decompose the pole
contribution into a non-grazing part and a grazing part. Recall
from Section~\ref{sec:par_F} that, for each \(j\in J(\lambda)\), the
grazing set \(G_{j,\n}\subset\bF_j^+\) is covered by finitely many
grazing charts, while the remaining charts stay uniformly separated from
\(G_{j,\n}\).

Using this partition of unity, we decompose the pole
contribution into its non-grazing and grazing parts.
Choose a cutoff function
\[
\chi_{j,\n}^{\rm gra}\in C^\infty\left(\overline{\bF_j^+}\right),
\qquad
0\le \chi_{j,\n}^{\rm gra}\le 1,
\]
such that
\[
\chi_{j,\n}^{\rm gra}\equiv1
\quad\text{on a neighborhood of }G_{j,\n},
\]
and the support of $\chi_{j,\n}^{\rm gra}$ is contained in the union of the grazing charts on \(\bF_j^+\).
Set
\[
\chi_{j,\n}^{\rm ng}:=1-\chi_{j,\n}^{\rm gra}
\qquad\text{on }\bF_j^+.
\]

\begin{definition}[Non-grazing and grazing contributions]
\label{def:ng_gra}
For each \(j\in J(\lambda)\), define
\begin{equation}
\label{eq:def_ngj}
u_j^{\rm ng}(\bx)
:=
2\pi \i
\int_{\bF_j^+}
\chi_{j,\n}^{\rm ng}(\balpha)
\frac{
\hat f_j(\balpha)\,
\phi_j(\balpha,\bx)
}{
\partial_s\mu_j(\balpha)\,
J(\balpha)
}
\,\d S(\balpha),
\end{equation}
and
\begin{equation}
\label{eq:def_graj}
u_j^{\rm gra}(\bx)
:=
2\pi \i
\int_{\bF_j^+\cup \bF_{j,c}^+}
\chi_{j,\n}^{\rm gra}(\balpha)
\frac{
\hat f_j(\balpha)\,
\phi_j(\balpha,\bx)
}{
 \partial_s\mu_j(\balpha)\,
J(\balpha)
}
\,\d S(\balpha).
\end{equation}
Summing over \(j\in J(\lambda)\), we set
\begin{equation}
\label{eq:def_ng_gra}
u^{\rm ng}(\bx):=\sum_{j\in J(\lambda)}u_j^{\rm ng}(\bx),
\qquad
u^{\rm gra}(\bx):=\sum_{j\in J(\lambda)}u_j^{\rm gra}(\bx).
\end{equation}
\end{definition}

By construction,
\[
u^{\rm pole}
=
u^{\rm ng}
+
u^{\rm gra},
\]
so that the limiting absorption solution admits the
decomposition
\[
u
=
u^{\rm evan}
+
u^{\rm ng}
+
u^{\rm gra}.
\]

This decomposition separates three geometrically distinct radiation mechanisms and provides the foundation for the remainder of the paper.

We first estimate the evanescent contribution.
Since it is represented by the fiber integral
\eqref{eq:evan}, the exponential decay established in
Theorem~\ref{thm:admissible-fiber}
immediately yields the following result.

\begin{theorem}[Exponential decay of the evanescent contribution]
\label{thm:decay_evan}
There are constants $\sigma_0,C>0$ such that for all $\bx=R\n$ with $R>>1$, 
\[
\left|u^{\rm evan}(\bx)\right|\leq C e^{-\sigma_0 R}.
\]
\end{theorem}

\begin{proof}

By Theorem~\ref{thm:admissible-fiber},
\[
\sigma(\balpha)\ge \sigma_0>0,
\qquad
\forall \balpha\in\Bz_\n.
\]
Hence
\[
|w(\balpha,R\n)|
\le
C e^{-R\sigma(\balpha)}
\le
C e^{-\sigma_0R}.
\]

Using the representation
\[
u^{\rm evan}(R\n)
=
\int_{\Bz_\n}
w(\balpha,R\n)\,\d\balpha,
\]
we obtain
\[
|u^{\rm evan}(R\n)|
\le
C e^{-\sigma_0R}
|\Bz_\n|.
\]
Since the Brillouin zone has finite measure,
the desired estimate follows.

\end{proof}

The exponentially decaying evanescent contribution is
therefore negligible in the far field. 
Consequently, the leading asymptotic behavior of the
limiting absorption solution is determined entirely by the
non-grazing and grazing contributions.
These two terms will be analyzed separately in
Sections~\ref{sec:asym_non_gra} and \ref{sec:decay_gra}, respectively.

\section{Classical asymptotics of the non-grazing component}
\label{sec:asym_non_gra}

In the previous section, the limiting absorption solution
was decomposed into an evanescent contribution, a
non-grazing contribution, and a grazing contribution. The
evanescent part decays exponentially and therefore does
not contribute to the leading far-field behavior. It
remains to analyze the non-grazing and grazing
contributions.

The purpose of the present section is not to develop a
complete asymptotic theory for arbitrary Fermi surfaces.
Rather, we use the SSA decomposition to identify the
generic asymptotic behavior associated with each
component. In this section, we show how the SSA decomposition reduces
the non-grazing contribution to finitely many localized
oscillatory integrals, thereby recovering the generic
classical stationary phase asymptotics.

\subsection{Local decomposition of the non-grazing contribution}
\label{sec:ng_local}

Combining the local parametrization of the outgoing Fermi sheets
established in Section~\ref{sec:cift} with a finite partition of
unity reduces the non-grazing contribution to finitely many localized
oscillatory integrals.
It therefore suffices to analyze a model oscillatory integral on a
bounded domain in $\mathbb{R}^{d-1}$.

For each observation direction $\n\in\mathbb S^{d-1}$, define
\[
\mathcal C^{\rm ng}(\n)
:=
\left\{
\balpha\in\bF^+:\,
\nu(\balpha)=\n
\right\},
\]
namely, the set of points on the outgoing Fermi surface whose outward
unit normal coincides with the observation direction.
Under Assumption~\ref{asp3},
\(\mathcal C^{\rm ng}(\n)\)
is discrete.
Since
\(\operatorname{supp}(\chi_\n^{\rm ng})\)
is compact,
only finitely many critical points contribute to the
localized stationary phase analysis.

From Definition~\ref{def:ng_gra} and $\bF_j^+=\bigcup_{m=1}^{K_j}\bL_{j,m}^+$, therefore
\[
u_j^{\rm ng}(\bx)
=
2\pi \i
\sum_{m=1}^{K_j}
\int_{\bL_{j,m}^+}
\chi_{j,\n}^{\rm ng}(\balpha)
\frac{
\hat f_j(\balpha)\,
\phi_j(\balpha,\bx)
}{
\partial_s\mu_j(\balpha)\,
J(\balpha)
}
\,\d S(\balpha).
\]
Since  ${\rm supp}\left(\chi_{j,\n}^{\rm ng}\right)\cap G_{j,\n}=\emptyset\), there exists a  \(c_0>0\) such that  for all $\balpha\in {\rm supp}(\chi_{j,\n}^{\rm ng})$, $ \partial_s\mu_j(\balpha)\ge c_0$. Therefore the integral is always well defined on the support of $\chi_{j,\n}^{\rm ng}$. We next localize the integral near the support of \(\chi_\n^{\rm ng}\). Define the compact set
 \[ K_\n := {\rm supp}(\chi_\n^{\rm ng})\cap\bL^+_{j,m}. \] 
To localize the stationary phase analysis, we construct a finite coordinate covering of \(K_\n\).

From Section~\ref{sec:cift}, for every \(\balpha\in K_\n\), there exist an open neighborhood \(U_\balpha\subset\bL^+_{j,m}\), a parameter domain \(V_\balpha\subset\R^{d-1}\), and a local parametrization
\[ \kappa:V_\balpha\longrightarrow U_\balpha, \qquad \bgamma\longmapsto \bgamma\t+s(\bgamma,0)\n, \]
 such that \(U_\balpha=\kappa(V_\balpha)\). If \(\balpha\in\mathcal C^{\rm ng}(\n)\), we further shrink \(U_\balpha\) so that 
 $ U_\balpha\cap\mathcal C^{\rm ng}(\n) = \{\balpha\}$. 
 By compactness, there exists a finite subcover 
 \[ K_\n \subset \bigcup_{q=1}^Q U_q, \]
  with the following properties: 
\begin{itemize}
\item for \(q=1,\dots,P\), the neighborhood \(U_q\) contains exactly one
point \(\balpha_q^*\in\mathcal C^{\rm ng}(\n)\). Moreover,
\[
\{\balpha_q^*:q=1,\dots,P\}
=
\mathcal C^{\rm ng}(\n)\cap\bL^+;
\]

\item for \(q=P+1,\dots,Q\), the neighborhood \(U_q\) contains no point
of \(\mathcal C^{\rm ng}(\n)\).
\end{itemize}
Denote the related parameter domain and local parameterization as $V_q$ and $\kappa_q$, respectively. Thus only finitely many local coordinate charts are
required to describe the non-grazing contribution.

Using the finite covering constructed above, we introduce a partition of unity:
\[
1=\sum_{q=1}^Q \eta_q
\qquad\text{on }K_\n,
\]
subordinate to the cover \(\{U_q\}_{q=1}^Q\). For each \(q\), let
\[
\widetilde\eta_q:=\eta_q\circ\kappa_q\in C_0^\infty(V_q).
\]
Consequently,
\[
u^{\rm ng}(R\n)=2\pi \i\sum_{q=1}^Q I_q(R\n),
\]
where
\[
I_q(R\n)
=
\int_{V_q}
\widetilde\eta_q(\bgamma)\,
\chi_\n^{\rm ng}(\kappa_q(\bgamma))
\frac{
\hat f(\kappa_q(\bgamma))\,
\psi(\kappa_q(\bgamma),R\n)
}{
\partial_s\mu(\kappa_q(\bgamma))
}
e^{iR\n\cdot\kappa_q(\bgamma)}
\,\d\bgamma,
\]
so that the global asymptotic analysis is reduced to the
study of finitely many localized oscillatory integrals.

Since only finitely many localized integrals occur, it
suffices to analyze one representative integral.
Suppressing the index \(q\), we write
\begin{equation}
\label{eq:ngr_I}
I(R\n)=\int_V B(\bgamma,R\n)e^{iR\Phi_\n(\bgamma)}\,\d\bgamma,
\end{equation}
where the phase function is
\begin{equation}
\label{eq:ngr_phi}
\Phi_\n(\bgamma):=\n\cdot\kappa(\bgamma),
\end{equation}
and the amplitude is given by
\begin{equation}
\label{eq:ngr_B}
B(\bgamma,R\n)
:=
\widetilde{\eta}(\bgamma)\,
\chi_\n^{\rm ng}(\kappa(\bgamma))
\frac{
\hat f(\kappa(\bgamma))\,
\psi(\kappa(\bgamma),R\n)
}{
\partial_s\mu(\kappa(\bgamma))
}.
\end{equation}
By abuse of notation, we also regard \(B\) as a function on the
corresponding local patch of the Fermi surface via the parametrization
\(\kappa\), namely,
\[
B(\balpha,R\n)
:=
B(\bgamma,R\n)\quad\text{ where }\balpha=\kappa(\bgamma).
\]

Hence the asymptotic analysis of the non-grazing
contribution is reduced to the model oscillatory integral
\eqref{eq:ngr_I}. This reduction is precisely the benefit of the SSA decomposition. By organizing the Floquet--Bloch representation slice by slice and localizing the outgoing Fermi surface into finitely many charts, the multidimensional far-field analysis is transformed into finitely many oscillatory integrals on bounded parameter domains, to which the classical stationary phase theorem applies directly. The remainder of this section is devoted
to the analysis of this integral.

\subsection{Generic non-grazing asymptotics}
\label{sec:ng_generic}

The localized representation reduces the asymptotic
analysis to the oscillatory integral \eqref{eq:ngr_I}. In
this section, we show that, under a generic non-grazing
assumption, its critical points are non-degenerate and
therefore the classical stationary phase theorem applies
directly.

\begin{lemma}[Geometric characterization of critical points]
\label{lem:critical}
Let \(\bgamma_*\in V\), and set $
\balpha_*:=\kappa(\bgamma_*).$
Then
\[
\nabla_\bgamma \Phi_\n(\bgamma_*)=0
\quad\Leftrightarrow\quad
\n=\nu(\balpha_*)
=
\frac{\nabla\mu(\balpha_*)}{|\nabla\mu(\balpha_*)|}.
\]
Equivalently, the critical points of \(\Phi_\n\) are precisely the
points on the local Fermi sheet whose unit normal coincides with the
observation direction \(\n\).
\end{lemma}

\begin{proof}
By definition of $\Phi_n$ in \eqref{eq:ngr_phi},
\[
\partial_{\bgamma_k}\Phi_\n(\bgamma)
=
\n\cdot\partial_{\bgamma_k}\kappa(\bgamma)=0,
\qquad
k=1,\ldots,d-1.
\]
Since \(\kappa\) is a local parametrization of the Fermi sheet, the tangent space is represented by
\[T_{\balpha_*}\bL^+_{j,m}:={\rm span}
\bigl\{
\partial_{\bgamma_1}\kappa(\bgamma_*),
\ldots,
\partial_{\bgamma_{d-1}}\kappa(\bgamma_*)
\bigr\},
\]
which is equivalent to
$
\n\perp T_{\balpha_*}\bL^+_{j,m}.
$
By the definition of the unit normal vector and the choice of $\bF^+$ which implies $\n\cdot\nu(\balpha_*)>0$, this is equivalent to
$
\n=\nu(\balpha_*).
$
\end{proof}

According to Lemma~\ref{lem:critical}, the asymptotic
behavior of \eqref{eq:ngr_I} is determined by the critical
points of the phase. Throughout the remainder of this section, we assume that
all critical points lie in the interior of the analytic
branch, namely,
\[
\mathcal C^{\rm ng}(\n)
\cap
\partial\bL^+
=
\emptyset.
\]
Under this generic assumption, the only remaining
requirement for applying the classical stationary phase
theorem is the non-degeneracy of the critical points.

\begin{lemma}[Hessian and the second fundamental form]
\label{lem:nondeg}
Let $\bgamma_*\in\mathcal{C}^{\rm ng}(\n)$ be a critical point and
$\balpha_*=\kappa(\bgamma_*)$. Then
\[
D_\bgamma^2\Phi_\n(\bgamma_*)
=
\mathrm{II}_{\balpha_*},
\]
where the second fundamental form is written in the local coordinates
$\bgamma$. Consequently, under Assumption~\ref{asp3}, every critical
point of $\Phi_\n$ is non-degenerate.
\end{lemma}

\begin{proof}
Let $\balpha_*=\kappa(\bgamma_*)$. By Lemma~\ref{lem:critical},
$
\n=\nu(\balpha_*).
$
From the definition of $\Phi_\n(\bgamma)$, 
we have
\[
\partial_{\bgamma_i\bgamma_j}^2\Phi_\n(\bgamma)
=
\n\cdot
\partial_{\bgamma_i\bgamma_j}^2\kappa(\bgamma).
\]
Let $
\tau_i(\bgamma):=\partial_{\bgamma_i}\kappa(\bgamma)
$
be a tangential vector to the Fermi surface for any fixed
$\bgamma\in V$. Hence
\[
\nu(\kappa(\bgamma))\cdot\tau_i(\bgamma)=0,
\qquad \bgamma\in V.
\]
Differentiating with respect to $\bgamma_j$ yields
\[
(\partial_{\bgamma_j}\nu)\cdot\tau_i
+
\nu\cdot
\partial_{\bgamma_i\bgamma_j}^2\kappa
=
0.
\]
Evaluating at $\bgamma=\bgamma_*$ and using
$\n=\nu(\balpha_*)$, we obtain
\[
\partial_{\bgamma_i\bgamma_j}^2\Phi_\n(\bgamma_*)
=
-(\partial_{\bgamma_j}\nu)\cdot\tau_i(\bgamma_*)=\mathrm{II}_{\balpha_*}(\tau_i,\tau_j).
\]
Therefore
\[
D_\bgamma^2\Phi_\n(\bgamma_*)
=
\mathrm{II}_{\balpha_*}.
\]
Assumption~\ref{asp3} implies that
\(
\det\mathrm{II}_{\balpha_*}\neq0
\),
so every critical point is non-degenerate.
\end{proof}

\begin{theorem}[Generic non-grazing asymptotics via SSA]
\label{thm:generic_ng}

Assume Assumption \ref{asp1}-\ref{asp3}, and
$
\mathcal C^{\rm ng}(\n)
\cap
\partial\bL^{+}
=
\emptyset.
$
Then the localized oscillatory integral \eqref{eq:ngr_I}
admits the asymptotic expansion
\[
I(R\n)
=
(2\pi)^{\frac{d-1}{2}}
R^{-\frac{d-1}{2}}
\sum_{\balpha_*\in\mathcal C^{\rm ng}(\n)}
\frac{
B(\balpha_*,R\n)
e^{iR\Phi(\balpha_*)}
e^{i\pi\sigma_*/4}
}
{\sqrt{|\det {\rm II}_{\balpha_*}|}}
+
O\!\left(R^{-\frac{d+1}{2}}\right),\quad R\rightarrow\infty.
\]
If no critical point is present, the stationary phase sum
is empty and the estimate reduces to the remainder term.
\end{theorem}

\begin{proof}For each
\(\balpha_*\in\mathcal C^{\rm ng}(\n)\),
let
\(\bgamma_*\)
denote its local coordinate. 
By Lemma~\ref{lem:critical}, the critical points of $\Phi_\n$ are
precisely the points of $\mathcal C^{\rm ng}(\n)$. The assumption
$
\mathcal C^{\rm ng}(\n)\cap\partial\bL^+=\emptyset
$
ensures that all critical points are interior, while
Lemma~\ref{lem:nondeg} shows that they are non-degenerate. The stated
expansion therefore follows directly from the classical stationary phase
theorem. 
\end{proof}

\begin{remark}
The generic assumption
\[
\mathcal C^{\rm ng}(\n)
\cap
\partial\bL^+
=
\emptyset
\]
excludes the situation where a stationary point reaches
the natural boundary of an analytic branch. In that case,
the local model becomes an oscillatory integral on a
domain with boundary, whose asymptotic behavior is
governed by the stationary phase theory for domains with
boundary.

The purpose of the present section is to identify the
generic asymptotic mechanism associated with the
non-grazing contribution. A complete treatment of
boundary stationary points depends on the geometry of the
natural boundaries separating adjacent analytic branches
of the Fermi surface and is therefore beyond the scope of
the present work.
\end{remark}

\section{Asymptotic analysis of the grazing component}
\label{sec:decay_gra}

In this section, we investigate the asymptotic behavior of
the grazing contribution defined in
\eqref{eq:def_graj}. Applying the localization procedure
developed in Section~\ref{sec:ng_local}, the grazing
contribution is decomposed into finitely many localized
oscillatory integrals associated with the local
parametrizations of the analytic branches. It therefore
suffices to consider a single localized contribution.
 Suppressing all branch indices, we
write
\[
I(R\n)
=
I^+(R\n)+I^-(R\n),
\]
where
\[
I^\pm(R\n)
=
\int_{V^\pm}
\chi_\n^{\rm gra}(\kappa^\pm(\bgamma))
\frac{
\hat f(\kappa^\pm(\bgamma))
\psi(\kappa^\pm(\bgamma),R\n)
}{
\partial_s\mu(\kappa^\pm(\bgamma))
}
e^{iR\n\cdot\kappa_q(\bgamma)}
\,\d\bgamma .
\]
The asymptotic behavior of \(u^{\rm gra}\) is therefore obtained by
summing the asymptotic expansions of these localized pairs.

Throughout this section, we write
\(
\balpha=\kappa^\pm(\bgamma)
\)
for the local parametrizations of the grazing branches.
Unlike the non-grazing case, these parametrizations
possess square-root singularities near the grazing set.
Consequently, the asymptotic analysis cannot be carried
out directly by the classical stationary phase method. The main task of this section is therefore to resolve the
square-root singularities introduced by the grazing
geometry and reduce the asymptotic analysis to a regular
oscillatory integral on the grazing manifold.

\subsection{Singularity analysis}
\label{sec:singular_grazing}

We fix one grazing chart associated with a connected
component of \(G_\n\). Since the analyses of the two local
branches are identical, we suppress the superscript
\(+\) or \(-\) whenever no confusion can arise.
The corresponding localized oscillatory integral takes the
form
\[
I(R\n)
=
\int_{V}
\frac{
A(\bgamma,R\n)
}
{
\partial_s\mu
\bigl(
\balpha(\bgamma)
\bigr)
}
e^{\i R\Phi_\n(\bgamma)}
\,\d\bgamma,
\]
where the phase function $\Phi_\n$ is defined in \eqref{eq:ngr_phi} and
\[
A(\bgamma,R\n)=\chi_\n^{\rm gra}(\kappa^\pm(\bgamma))
\hat f(\kappa^\pm(\bgamma))
\psi(\kappa^\pm(\bgamma),R\n).
\]
Unlike the non-grazing case, the local parametrizations
develop square-root singularities near the grazing set.
The first task is therefore to identify the singular
behavior of the phase, the residue factor, and the
amplitude.

By construction, if $V=V^+\subset\overline{\bF^+_j}$, $A$ contains the localization cutoff function $\mathcal{X}$ and therefore
vanishes near $\partial V\setminus G_\n$. If $V=V^-$ then $\kappa(\bgamma)$ lies on the complex extension of the Fermi surface and  $A$ no longer vanishes on $\partial V\setminus G_\n$. Instead, the decay of $e^{\i R\Phi_\n(\bgamma)}$ needs to be investigated. 

By Theorem~\ref{th:grazing-branch}, the local parameterizationnear $G_\n$ has the following form:
\begin{equation}
\label{eq:par_branch}
\kappa(\bgamma)=\kappa_G(\bgamma')+r \t_1+\sqrt{\eta(\bgamma')r+i0}\,\n
+
O(r)
\end{equation}
where
\begin{equation}
\label{eq:par_branch_G}
\kappa_G(\bgamma')=
\psi(\bgamma')\t_1
+\sum_{j=2}^{d-1}\bgamma_j\t_j
+\phi(\bgamma')
\n
\end{equation}
is the local paramterization of a neighbourhood in $G_\n$,
\[
r=\bgamma_1-\psi(\bgamma'),
\qquad
\eta(\bgamma')
=
-\frac{2b(\bgamma')}{a(\bgamma')}.
\]
Recall that the coefficient functions
\(a(\bgamma')\)
and
\(b(\bgamma')\) defined on $G_\n$
were introduced in Section~\ref{sec:dift}. Under Assumption~\ref{asp3},
\[
a(\bgamma')\neq0,
\qquad
b(\bgamma')\neq0.
\]

We first show that the complex branch contributes only in
an arbitrarily small neighborhood of the grazing set.

\begin{lemma}[Exponential localization]
\label{lem:exp-localization}

Let $
\kappa(\bgamma)
$ 
be the complex branch given by \eqref{eq:par_branch}. Then there exist
\(\delta>0\) such that, for every
\(0<\delta_0<\delta\),
\[
\left|
e^{iR\kappa(\bgamma)\cdot\n}
\right|
\le
e^{-c_{\delta_0}R},
\quad\forall\,
\delta_0
\le
\left|
\gamma_1-\psi(\bgamma')
\right|
\le
\delta
\]
on the complex branch, where \(c_{\delta_0}>0\) is independent of
\(\bgamma\).

\end{lemma}

\begin{proof}
Let $s(\bgamma):=\sqrt{\eta(\bgamma)+\i 0}+O(r)$. For every fixed \(0<\delta_0<\delta\), $
\Im s_+(\bgamma)
$
is  continuous and positive on the compact complex branch where $
\delta_0
\le
\left|
\gamma_1-\psi(\bgamma')
\right|
\le
\delta
$
Therefore,
\[
\Im s_+(\bgamma)\ge c_{\delta_0}>0,
\]
uniformly in \(\bgamma'\). Since
$
\kappa(\bgamma)\cdot\n=s(\bgamma),
$
we obtain
\[
\left|
e^{iR\kappa(\bgamma)\cdot\n}
\right|
\le
e^{-c_{\delta_0}R},
\]
which immediately yields the desired estimate.

\end{proof}

It therefore remains to identify the local singular
structure near the grazing set. The next two lemmas
describe the singular behavior of the residue factor and
the corresponding phase-amplitude expansion.

\begin{lemma}[Square-root singularity of the residue factor]
\label{lem:sqrt-singularity}
From the definitions above,
\[
\partial_s\mu
\bigl(\kappa(\bgamma)
\bigr)
=a(\bgamma')
\sqrt{
\eta(\bgamma')\,r+i0
}
+
O(r).
\]

\end{lemma}

\begin{proof}
Since \(\kappa_G(\bgamma)\in G_\n\),
\[
\partial_s\mu(\kappa_G(\bgamma))
=
\nabla\mu(\kappa_G(\bgamma))\cdot\n
=
0.
\]
Therefore, Taylor's formula gives
\[
\partial_s\mu(\kappa(\bgamma))
=
D^2\mu(\kappa_G(\bgamma))
\bigl[\n,\kappa(\bgamma)-\kappa_G(\bgamma)\bigr]
+
O\!\left(
|\kappa(\bgamma)-\kappa_G(\bgamma)|^2
\right).
\]
Using
\[
\kappa(\bgamma)-\kappa_G(\bgamma)
=
r\t_1
+
\sqrt{\eta(\bgamma')\,r+i0}\,\n
+
O(r),
\]
together with the definitions
\[
a(\bgamma')
=
D^2\mu(\kappa_G(\bgamma))[\n,\n].
\]
Since the tangential contribution is of order \(r\) we prove this lemma.
\end{proof}

\begin{lemma}[Local phase and amplitude expansions]
\label{lem:local-expansion}

Under the local parametrization
\[
\kappa(\bgamma)
=
\kappa_G(\bgamma')
+r\t_1
+\sqrt{\eta(\bgamma')\,r+i0}\,\n
+O(r),
\]
the amplitude and phase admit the expansions
\[
A(r,\bgamma',R\n)
=
A_0(\bgamma',R\n)
+
A_1(\bgamma',R\n)
\sqrt{\eta(\bgamma')\,r+i0}
+
O(r),
\]
and
\[
\Phi_\n(r,\bgamma')
=
\Phi_0(\bgamma')
+
\sqrt{\eta(\bgamma')\,r+i0}
+
O(r),
\]
where
\[
\Phi_0(\bgamma')
:=
\n\cdot\kappa_G(\bgamma'),
\]
and \(A_0\), \(A_1\) are smooth coefficient functions of
\((\bgamma',R\n)\), with
\[
A_0(\bgamma',R\n)
=
A\bigl(\kappa_G(\bgamma'),R\n\bigr).
\]

\end{lemma}

\begin{proof}
Since \(A(\balpha,R\n)\) is analytic on each local analytic branch,
Taylor's theorem together with the local parametrization of
\(\kappa(\bgamma)\) yields the stated expansion of \(A\).

Moreover, $
\Phi_\n(\bgamma)
=
\n\cdot\kappa(\bgamma)$
and the expansion of \(\Phi_\n\) follows immediately from the local
parametrization and the identity
$
\n\cdot\t_1=0.
$
\end{proof}

The preceding lemmas identify the local structure of the
grazing oscillatory integral. In particular, the
square-root singularity is completely described by the
transversal variable \(r\), while the remaining variables
parameterize the grazing manifold. This decomposition will
allow us, in the next subsection, to perform the
transversal integration and reduce the grazing asymptotic
analysis to a regular oscillatory integral on \(G_\n\).

\subsection{Generic grazing geometry}
\label{sec:gra_geo}

We now reduce the localized grazing oscillatory integral
to an oscillatory integral on the grazing manifold.
By the localization on the real branch together with
Lemma~\ref{lem:exp-localization} for the complex branch,
every localized contribution is supported in an
arbitrarily small one-sided neighborhood of the grazing
set. More precisely, after shrinking the
localization neighborhood if necessary, we may assume that the
transversal variable satisfies $
r\in I,$
where
$
I=[0,\delta_0]$ or $I=[-\delta_0,0]$,
depending on the branch under consideration, with
\(0<\delta_0<\delta\). For $r=0$, the point $\kappa(r,\bgamma')\in G_\n$.

The next lemma performs the transversal integration across
the grazing manifold and provides the key reduction for
the subsequent stationary phase analysis.

\begin{lemma}[Evaluation of the transversal integral]
\label{lem:transversal-reduction}

Assume that \(A((r,\bgamma'),R\n)\) satisfies the expansion in
Lemma~\ref{lem:local-expansion}. For sufficiently small $\delta_0>0$ then the following asymptotic analysis holds uniformly in $\bgamma'$:
\begin{equation}
\label{eq:gra_trans_1order}
K(R,\bgamma')
:=
\int_I
\frac{
A((r,\bgamma'),R\n)
}
{\partial_s\mu(\kappa(r,\bgamma'))}
e^{\i R\Phi_\n(r,\bgamma')}
\,\d r
=R^{-1} B(\bgamma',R\n)+O(R^{-2})
\end{equation}
where
\[B(\bgamma',R\n)=
\frac{A(\kappa_G(\bgamma'),R\n)}
{\i  b(\bgamma')}e^{\i R\Phi_0(\bgamma')}.
\]
\end{lemma}

\begin{proof}
It suffices to consider the case
$
I=[0,\delta_0].
$
Indeed, if
\(
I=[-\delta_0,0],
\)
the change of variables
\(
r\mapsto-r
\)
reduces the integral to the same form.

By Lemmas~\ref{lem:sqrt-singularity} and
\ref{lem:local-expansion}, let $r=\rho^2$,
\[
K(R,\gamma')
=
2
\int_0^{\sqrt{\delta_0}}
a(\rho,\gamma')
e^{\i R\phi(\rho,\gamma')}
\,\d\rho,
\]
where
\[
a(\rho,\bgamma')
=\frac{
A(\rho^2,\bgamma',R\n)}{\partial_s\mu(\kappa(\rho^2,\bgamma'))}\rho e^{\i R\Phi_0(\bgamma')}
=\frac{
A(\kappa_G(\bgamma'),R\n))}{a(\bgamma')\sqrt{\eta(\bgamma')}}e^{\i R\Phi_0(\bgamma')} +O(\rho)
\]
is a smooth function in $[0,\sqrt{\delta_0}]$, and
\[
\phi(\rho,\gamma')
=
\sqrt{\eta(\gamma')}\rho+O(\rho^2).
\]
Hence for sufficiently small \(\delta_0\),
\[
\partial_\rho\phi(\rho,\gamma')
\ge c_0>0.
\]

Integration by parts gives
\[
K(R,\bgamma')
=-
\frac{2 a(0,\bgamma')}
{\i R\sqrt{\eta(\gamma')}}
+
R_1(R,\gamma')
+
B(R,\gamma'),
\]
where \(R_1\) denotes the remaining oscillatory integral after one integration by parts, and \(B\) is the contribution from the upper endpoint $\sqrt{\delta_0}$. In the real case,
\(B(R,\gamma')=0\) by the cutoff function, while in the complex case,
Lemma~\ref{lem:exp-localization} implies
$
B(R,\gamma')
=
O(e^{-cR}).
$
A second integration by parts yields
$
R_1(R,\gamma')
=
O(R^{-2}),
$
uniformly in \(\gamma'\). This completes the proof.

\end{proof}

More generally, the integration-by-parts argument can be
iterated to obtain an asymptotic expansion of arbitrary
order. For every \(N\in\mathbb N\),
\begin{equation}
\label{eq:gra_trans_Norder}
K(R,\bgamma')
=
e^{iR\Phi_0(\bgamma')}R^{-1}
B_N(\bgamma',R\n)
+
O_N(R^{-N-1})
,
\end{equation}
where
\[
B_N(\bgamma',R\n)=
\sum_{m=0}^{N-1}
\frac{c_m(\bgamma',R\n)}{R^{m}}=B(\bgamma',R\n)+O(R^{-1})
\]
uniformly in \(\bgamma'\). 
This higher-order expansion will be used
implicitly in the proof of the main asymptotic theorem.

Similar to Section \ref{sec:ng_local}, by  abuse of notation, we still regard \(B\) and $B_N$ as functions on the grazing set $G_\n$ via the parametrization
\(\kappa_G\):
\[
 B(\balpha,R\n)
:=
B(\bgamma',R\n)\quad B_N(\balpha,R\n)
:=
B_N(\bgamma',R\n)\text{ where }\balpha=\kappa_G(\bgamma').
\]

Let \(V'\) denote the coordinate domain of the local parametrization
$
\kappa_G:V'\to G_\n
$
is defined as in \eqref{eq:par_branch_G}.
The localized grazing contribution can be written as
\[
I(R\n)
=
\int_{V'}
K(R,\bgamma')
\,\d\bgamma',
\]
where
\[
K(R,\bgamma')
:=
\int_I
\frac{
A((r,\bgamma'),R\n)
}
{
\partial_s\mu\bigl(\kappa(r,\bgamma')\bigr)
}
e^{iR\Phi_\n(r,\bgamma')}
\,\d r .
\]
Consequently, the grazing contribution is reduced to an
oscillatory integral over the grazing manifold
\(G_\n\),
\[
I(R\n)
=
R^{-1}
\int_{V'}
B_N(\bgamma',R\n)
e^{iR\Phi_0(\bgamma')}
\,\d\bgamma'
+
O(R^{-N-1}),
\]
whose phase is the restriction of the original phase to
the grazing manifold. From this point on, the grazing manifold becomes the natural geometric object governing the remaining asymptotic analysis. The remaining asymptotic analysis is therefore entirely
determined by the critical points of
\(\Phi_0\). Since
\[
\Phi_0(\bgamma')
=
\n\cdot\kappa_G(\bgamma'),
\]
we have
\[
\nabla_{\bgamma'}\Phi_0(\bgamma')=0
\quad\Longleftrightarrow\quad
\n\perp T_{\kappa_G(\bgamma')}G_\n.
\]

\begin{definition}[Grazing critical points]
\label{def:grazing-critical}

For each observation direction \(\n\in\mathbb S^{d-1}\), define
\[
\mathcal C^{\rm gra}(\n)
:=
\{\balpha_*\in G_\n:\n\perp T_{\balpha_*} G_\n\}.
\]
To identify the generic grazing asymptotics, we further
assume that
\(\Phi|_{G_\n}\)
is a Morse function.

\end{definition}

\begin{theorem}[Generic grazing asymptotics]
\label{thm:generic-grazing}

Assume that the restriction of \(\Phi\) to \(G_\n\) is a Morse
function. Then
\begin{equation}
\label{eq:gra_asym}
I(R\n)
=(2\pi)^{\frac{d-2}{2}}
R^{-d/2}
\sum_{\balpha_*\in\mathcal{C}^{\rm gra}(\n)}
\frac{
B(\balpha_*,R\n)
}{
\sqrt{
\left|
\det D^2(\Phi|_{G_\n})(\balpha_*)
\right|
}
}
e^{iR\Phi(\balpha_*)}
e^{i\pi\sigma_*/4}
+
O(R^{-d/2-1}),
\end{equation}
where
\[
\sigma_*
=
\operatorname{sgn}
D^2(\Phi|_{G_\n})(\balpha_*).
\]
\end{theorem}

\begin{proof}
By Lemma~\ref{lem:transversal-reduction}, the grazing
oscillatory integral is reduced to an oscillatory integral
on the grazing manifold, up to arbitrarily high order:
\[
I(R\n)
=
R^{-1}
\int_{V'}
B_N(\bgamma',R\n)
e^{iR\Phi_0(\bgamma')}
\,\d\bgamma'
+
O(R^{-N}),
\]
where \(B_N(\cdot,R\n)\) is smooth on \(V'\) and satisfies
\[
B_N(\bgamma',R\n)
=
B(\kappa_G(\bgamma'),R\n)
+
O(R^{-1}).
\]

Choose \(N>d/2\). For each \(\balpha_*\in\mathcal C^{\rm gra}(\n)\), let
\(\bgamma_*\in V'\) denote its local coordinate, i.e.,
\(\kappa_G(\bgamma_*)=\balpha_*\).
The stationary phase theorem then yields
\[
I(R\n)
=
R^{-d/2}
\sum_{\balpha_*\in \mathcal{C}^{\rm gra}(\n)}
(2\pi)^{\frac{d-2}{2}}
\frac{
B_N(\bgamma_*,R\n)
}{
\sqrt{
\left|
\det D^2\Phi_0(\bgamma_*)
\right|
}
}
e^{iR\Phi_0(\bgamma_*)}
e^{i\pi\sigma_*/4}
+
O(R^{-d/2-1}),
\]
where
\(
\kappa_G(\bgamma_*)=\balpha_*
\). 
Since
\[
B_N(\bgamma_*,R\n)
=
B(\balpha_*,R\n)
+
O(R^{-1}),
\quad
\Phi_0(\bgamma_*)
=
\Phi(\balpha_*),
\quad
D^2\Phi_0(\bgamma_*)
=
D^2(\Phi|_{G_\n})(\balpha_*).
\]
Substituting these identities into the preceding expansion proves
\eqref{eq:gra_asym}.
\end{proof}

\begin{remark}
The normal space of the grazing manifold satisfies
\[
N_\balpha G_\n
=
\operatorname{span}
\{\nabla\mu(\balpha),\,D^2\mu(\balpha)\n\}.
\]
Hence the grazing critical set is characterized entirely by the local
differential structure of the dispersion relation. In particular, the
Morse assumption is an intrinsic geometric condition on the grazing
manifold.
\end{remark}

\section{Examples and Comparisons}
\label{sec:geo_gra}

\subsection{The grazing contribution as an independent radiation mechanism}
\label{sec:grazing_dominates}

Sections~\ref{sec:asym_non_gra} and~\ref{sec:decay_gra} establish the
far-field asymptotic expansions of the propagating and grazing
contributions, respectively.
A natural question remains:

\begin{center}
\emph{Does the grazing component represent a genuinely new radiation mechanism, or is it merely a higher-order correction to the classical propagating radiation?}
\end{center}

We construct a local spectral model satisfying the assumptions of
Theorem~\ref{thm:generic-grazing} for which the propagating leading
terms vanish while the grazing contribution remains nonzero.
As a consequence, the grazing contribution becomes the leading term of
the far-field asymptotics.

Let $\balpha_0\in G_\n$ be a grazing critical point. After a translation and a rigid change of
coordinates, we may assume that
\[
\balpha_0=(0,0,0),
\qquad
\n=\e_1.
\]

Assume that, in a neighborhood of \(\balpha_0\), for $\kappa_1\kappa_2\neq 0$, the dispersion
relation is given by
\[
\mu(\balpha)
=
\alpha_2
-
\frac12
\left(
\kappa_1\alpha_1^2
+
\kappa_2\alpha_3^2
\right)
+
\alpha_1\alpha_3^2.
\]
Then the Fermi surface at energy \(\lambda=0\) is locally represented
by
\[
\bF
=
\left\{
\alpha_2
=
\frac12
\left(
\kappa_1\alpha_1^2
+
\kappa_2\alpha_3^2
\right)
-
\alpha_1\alpha_3^2
\right\}.
\]
The Fermi surface is nondegenerate since
\[
D^2f(0)
=
{\rm diag}(\kappa_1,\kappa_2).
\]

With the grazing condition
$
\nabla\mu(\balpha)\cdot\n=0
$
we obtain
\[
\alpha_1
=
\frac1{\kappa_1}\alpha_3^2,
\qquad
\alpha_2
=
\frac{\kappa_2}{2}\alpha_3^2
+
O(\alpha_3^4).
\]
Hence the grazing manifold admits the local parametrization
\[
G_\n(\alpha_3)
=
\left(
\frac1{\kappa_1}\alpha_3^2,\;
\frac{\kappa_2}{2}\alpha_3^2
+
O(\alpha_3^4),\;
\alpha_3
\right).
\]

Moreover,
\[
\Phi|_{G_\n}
=
\balpha\cdot\n
=
\frac1{\kappa_1}\alpha_3^2,
\]
so that
\[
\Phi(\balpha_0)=0,
\qquad
(\Phi|_{G_\n})'(0)=0,
\qquad
(\Phi|_{G_\n})''(0)
=
\frac2{\kappa_1}\neq0.
\]
Hence
\(\balpha_0\)
is a non-degenerate grazing critical point.

Applying
Theorem~\ref{thm:generic-grazing},
we obtain
\[
u^{\rm gra}(R\n,\bx)
=
C(\balpha_0,\bx)
R^{-3/2}
+
o(R^{-3/2}),
\]
where
\[
C(\balpha_0,\bx)
=
(2\pi)^{-1/2}
e^{i\pi\operatorname{sgn}(\kappa_1)/4}
\frac{
\hat f_j(\balpha_0)\psi_j(\balpha_0,\bx)
}
{\sqrt{|1/\kappa_1|}}.
\]
Let
\[
\mathcal C^{\rm ng}(\n)
=
\{\balpha_1,\ldots,\balpha_m\}
\]
denote the propagating stationary points associated with the observation
direction \(\n\).
Since the source is arbitrary, we may choose
\[\hat{f}_j(\balpha_0)\neq 0\qquad\text{and}\qquad
\hat f_j(\balpha_r)=0,
\qquad
\nabla\hat f_j(\balpha_r)\neq0,
\qquad
r=1,\ldots,m.
\]
Theorem~\ref{thm:generic_ng}
implies
\[
u^{\rm ng}(R\n,\bx)
=
O(R^{-2}).
\]
The result is summarized in the following theorem.

\begin{theorem}[Dominance of the grazing contribution]
\label{thm:grazing_dominance}

The local spectral model constructed above satisfies the assumptions of
Theorems~\ref{thm:generic_ng}
and~\ref{thm:generic-grazing}.
For this model,
\[
u^{\rm ng}(R\n,\bx)
=
O(R^{-2}),
\]
whereas
\[
u^{\rm gra}(R\n,\bx)
=
C(\balpha_0,\bx)
R^{-3/2}
+
o(R^{-3/2}),
\qquad
C(\balpha_0,\bx)\neq0.
\]
Consequently,
\[
u(R\n,\bx)
=
C(\balpha_0,\bx)
R^{-3/2}
+
O(R^{-2}),
\]
so that the grazing contribution determines the leading far-field
asymptotics.

\end{theorem}

The theorem is an immediate consequence of the above construction
together with
Theorems~\ref{thm:generic_ng}
and~\ref{thm:generic-grazing}.

\begin{remark}[Geometric necessity of the grazing contribution]

The theorem above shows that the propagating and grazing contributions
originate from two independent geometric mechanisms.
The propagating contribution is generated by stationary points of the
phase on the outgoing Fermi sheets, whereas the grazing contribution is
determined by critical points of the phase restricted to the grazing
manifold.

The two mechanisms are independent.
Suppressing the leading propagating coefficients does not eliminate the
grazing contribution, which may therefore become the leading term of the
radiation field.

The example therefore establishes that the three-term decomposition is not merely a finer decomposition of the classical radiation field, but is geometrically necessary even at the level of the leading-order asymptotics.

\end{remark}

\subsection{The Helmholtz equation as a degenerate grazing geometry}
\label{sec:green_grazing}

The Helmholtz equation represents the simplest example where the generic assumptions fail, allowing the cancellation mechanism to be analyzed completely explicitly.  Owing to the spherical symmetry of its Fermi surface, the restriction of the phase to the grazing manifold is constant. Although the generic stationary phase analysis no longer applies, this exceptional symmetry makes the interaction between the real and complex grazing branches completely explicit.

We consider the free-space Green's function for the Helmholtz equation in $\R^d$ for $d\geq 2$:
\begin{equation}
\label{eq:green}
-\Delta u-k^2 u=\delta(\bx),
\end{equation}
where $k>0$ is the wave number. 
It is already known that the decay rate of the Green's function
\[
|u(\bx)|=O\left(R^{-\frac{d-1}{2}}\right),\quad\forall d\geq 2.
\]
The Helmholtz operator corresponds to the simplest homogeneous periodic
medium, for which the Floquet--Bloch transform reduces to the Fourier
transform.

Taking the Fourier transform gives
\[
u(\bx)=\frac{1}{(2\pi)^d}\int_{\R^d}\frac{e^{\i\xi\cdot\bx}}{|\xi|^2-k^2-\i 0}\d\xi.
\]
The band function is
\[
\mu(\xi)=|\xi|^2,
\]
so that $\bF$ is the sphere of radius $k$:
\[
\bF
=
\{\xi\in\R^d:\ |\xi|=k\}.
\]

Fix the observation direction
$
\n=\e_d,
$. Denote
$
\xi=(\xi',\xi_d)$ and
$\xi'\in\R^{d-1},
$
and 
$
B_k^{d-1}
=
\{\xi'\in\R^{d-1}:|\xi'|<k\}.
$
Therefore,
\[
\bx=R\n,\qquad
\xi\cdot\bx=R\xi_d.
\]
The outgoing branch of the Fermi surface is
\[
\bF^+
=
\{(\xi',\sqrt{k^2-|\xi'|^2}):\xi'\in B_k^{d-1}\},
\]
whose boundary is the grazing set
\[
G_\n
=
\{(\xi',0):|\xi'|=k\}.
\]
Across this interface, the real branch is naturally continued by the complex branch
\[
\bF_c^+
=
\left\{(\xi',i\sqrt{|\xi'|^2-k^2}):
\xi'\in\R^{d-1}\setminus B_k^{d-1}\right\}.
\]

Applying Theorem~\ref{thm:pole_surface} together with the
parametrizations of \(\bF^+\) and \(\bF_c^+\), and letting
\(\varepsilon\to0^+\), we obtain
\[
u(\bx)
=
\frac{\pi \i}{(2\pi)^d}
\int_{B_k^{d-1}}
\frac{
e^{\i R\sqrt{k^2-|\xi'|^2}}
}
{\sqrt{k^2-|\xi'|^2}}
\,\d\xi'
+
\frac{\pi}{(2\pi)^d}
\int_{\R^{d-1}\setminus B_k^{d-1}}
\frac{
e^{-R\sqrt{|\xi'|^2-k^2}}
}
{\sqrt{|\xi'|^2-k^2}}
\,\d\xi'.
\]
Since the integrands depend only on \(|\xi'|\), passing to polar
coordinates yields
\[
u(\bx)
=
\i C_d
\int_0^k
\frac{\rho^{d-2}}
{\sqrt{k^2-\rho^2}}
e^{\i R\sqrt{k^2-\rho^2}}
\,\d\rho
+
C_d
\int_k^\infty
\frac{\rho^{d-2}}
{\sqrt{\rho^2-k^2}}
e^{-R\sqrt{\rho^2-k^2}}
\,\d\rho,\quad
C_d=\frac{\pi\omega_{d-2}}{(2\pi)^d}.
\]

\subsubsection{Degenerate grazing geometry}

Following the decomposition in Section~\ref{sec:asym_decomp}, we split
the Green's function into the evanescent, non-grazing, and grazing
contributions. Let \(\chi\in C^\infty([0,k])\) satisfy that for a small $0<\delta_0<k$,
\[
\chi(\rho)=\begin{cases}
1,\quad |\rho- k|\leq \delta_0/2;\\
0,\quad |\rho- k|\geq \delta_0;\\
\text{smooth},\quad\text{otherwise}.
\end{cases}
\]
Then \[
u=u^{\rm evan}+u^{\rm ng}+u^{\rm gra}
\]
where
\begin{align*}
u^{\rm evan}(\bx)&=C_n
\int_{k}^\infty(1-\chi(\rho))
\frac{\rho^{d-2}}
{\sqrt{\rho^2-k^2}}
e^{-R\sqrt{\rho^2-k^2}}
\,\d\rho;\\
u^{\rm ng}(\bx)&=
\i C_n
\int_0^k(1-\chi(\rho))
\frac{\rho^{d-2}}
{\sqrt{k^2-\rho^2}}
e^{\i R\sqrt{k^2-\rho^2}}
\,\d\rho;\\
u^{\rm gra}(\bx)&=\i C_n
\int_0^k \chi(\rho)
\frac{\rho^{d-2}}
{\sqrt{k^2-\rho^2}}
e^{\i R\sqrt{k^2-\rho^2}}\,\d\rho+C_n
\int_k^{k+\delta_0}\chi(\rho)
\frac{\rho^{d-2}}
{\sqrt{\rho^2-k^2}}
e^{-R\sqrt{\rho^2-k^2}}
\,\d\rho.
\end{align*}
The non-grazing contribution is treated by the analysis of
Section~\ref{sec:asym_non_gra}, while the evanescent contribution is
exponentially small. Therefore, it remains to analyze the grazing
contribution.

Accordingly, we write
\[
u^{\rm gra}
=
C_n I^{\rm r}
+
C_n I^{\rm c},
\]
where
\[
I^{\rm r}
=\i 
\int_{k-\delta_0}^k \chi(\rho)
\frac{\rho^{d-2}}
{\sqrt{k^2-\rho^2}}
e^{\i R\sqrt{k^2-\rho^2}}\,\d\rho
\]
and
\[
I^{\rm c}
=\int_k^{k+\delta_0}\chi(\rho)
\frac{\rho^{d-2}}
{\sqrt{\rho^2-k^2}}
e^{-R\sqrt{\rho^2-k^2}}
\,\d\rho.
\]

Introducing the transversal variable
\[
s=\sqrt{|k^2-\rho^2|},
\]
we obtain
\[
I^{\rm r}
=
i
\int_0^{\sqrt{2k\delta_0-\delta_0^2}}
(k^2-s^2)^{\frac{d-3}{2}}
e^{\i Rs}
\chi(s)
\,\d S,
\]
and
\[
I^{\rm c}
=
\int_0^{\sqrt{2k\delta_0+\delta_0^2}}
(k^2+s^2)^{\frac{d-3}{2}}
e^{-Rs}\chi(s)
\,\d S.
\]

For every \(N\ge1\),
\[
(k^2\pm s^2)^{\frac{d-3}{2}}
=
\sum_{m=0}^{N-1}
(\pm1)^m c_m s^{2m}
+
O(s^{2N}),
\]
where \(c_m\) are the Taylor coefficients of
\((k^2+t)^{\frac{d-3}{2}}\) at \(t=0\). 

\begin{lemma}
For every \(m\ge0\),
\[
\int_0^{a}
s^{2m}\chi(s)e^{\i Rs}\,\d s
=
-\frac{(2m)!}{(\i R)^{2m+1}}
+
O(R^{-2m-2}),
\]
and
\[
\int_0^{b}
s^{2m}\chi(s)e^{-Rs}\,\d s
=
\frac{(2m)!}{R^{2m+1}}
+
O(R^{-2m-2}).
\]
\end{lemma}

\begin{proof}
Since
$\chi(s)=1$ near $0$ and
 $s^{2m}$ vanishes to order $2m$,
the first nonzero boundary term appears after
$
2m+1
$
integrations by parts. This gives
\[
\int_0^a
s^{2m}\chi(s)e^{\i Rs}\,\d s
=
-\frac{(2m)!}{(\i R)^{2m+1}}
+
O(R^{-2m-2}),
\]
which is exactly the first identity. The second follows in the same way.
\end{proof}

Consequently,
\[
I^{\rm r}
=
\i
\sum_{m=0}^{N-1}
(-1)^m c_m
\int_0^{a}
s^{2m}
e^{\i Rs}
\chi(s)\,\d s
+
O(R^{-2N-1})=
-
\sum_{m=0}^{N-1}
\frac{(2m)!c_m}{R^{2m+1}}
+
O(R^{-2N-1}),
\]
and
\[
I^{\rm c}
=
\sum_{m=0}^{N-1}
c_m
\int_0^{b}
s^{2m}
e^{-Rs}
\chi(s)\,\d S
+
O(R^{-2N-1})=
\sum_{m=0}^{N-1}
\frac{(2m)!c_m}{R^{2m+1}}
+
O(R^{-2N-1}),
\]
where
\[
a=\sqrt{2k\delta_0-\delta_0^2},
\qquad
b=\sqrt{2k\delta_0+\delta_0^2}.
\]

Therefore, every algebraic term in the grazing asymptotic expansions
cancels identically between the real and complex branches. In
particular, the \(O(R^{-1})\) contribution predicted by the generic
grazing analysis vanishes completely in the Helmholtz case.

This complete cancellation is a direct consequence of the degenerate grazing geometry. Indeed, the phase function satisfies
\[
\Phi_\n(\bgamma)
=
\n\cdot\balpha(\bgamma)
=
\alpha_n(\bgamma),
\]
and its restriction to the grazing manifold is identically constant,
\[
\Phi|_{G_\n}\equiv0.
\]
Hence the grazing manifold contains no non-degenerate stationary point, and the Morse assumption underlying the generic grazing theory fails identically. The stationary phase mechanism responsible for the generic grazing asymptotics is therefore replaced by an exact cancellation between the real and complex grazing branches. Consequently, the absence of a grazing contribution in the Helmholtz asymptotics is best understood as a consequence of this exceptional symmetry rather than as the generic behavior of grazing geometries.

\begin{remark}[Other degenerate grazing geometries]

The Helmholtz equation provides one example of a degenerate grazing
geometry, where the restriction of the phase to the grazing manifold is
identically constant. More generally, the generic assumptions of
Section~\ref{sec:decay_gra} may fail in several different ways,
leading to qualitatively different asymptotic behaviors.

Typical examples include degenerate critical points of
\(\Phi|_{G_\n}\), degeneracies of the second fundamental form of
\(G_\n\), and singular grazing manifolds. These cases generally require techniques beyond the stationary phase
analysis developed in the present work and appear to be a natural next
step toward a complete asymptotic theory of Green's functions in
periodic media.

\end{remark}

The explicit cancellation obtained above also clarifies the relationship between the SSA decomposition and the classical propagating--regular decomposition. We now return to the original representation and show that the grazing contribution is not absent from the classical decomposition; rather, it is distributed between the propagating and regular parts and disappears only through an order-by-order cancellation.

\subsubsection{Comparison with the classical decomposition}

This cancellation also explains the structure of the classical propagating--regular decomposition.  The Helmholtz equation provides a particularly transparent example, since both decompositions can be computed explicitly. We recover the classical
propagating--regular decomposition
\[
u=u^{\rm prop}+u^{\rm reg},
\]
where
\[
u^{\rm prop}
=
\i C_d
\int_0^k
\frac{\rho^{d-2}}
{\sqrt{k^2-\rho^2}}
e^{\i R\sqrt{k^2-\rho^2}}
\,\d\rho=u^{\rm ng}+C_d I^r:=C_d I^+,
\]
and
\[
u^{\rm reg}
=
C_d
\int_k^\infty
\frac{\rho^{d-2}}
{\sqrt{\rho^2-k^2}}
e^{-R\sqrt{\rho^2-k^2}}
\,\d\rho=u^{\rm evan}+C_d I^c:=C_d I^-.
\]
Thus the classical propagating part consists of the non-grazing contribution together with the real grazing contribution, whereas the regular part consists of the evanescent contribution together with the complex grazing contribution.

Applying the same change of variables as above gives
\[
I^{+}
=
\i
\int_0^k
(k^2-s^2)^{\frac{d-3}{2}}
e^{\i Rs}\,\d s,
\]
and
\[
I^{-}
=
\int_0^\infty
(k^2+s^2)^{\frac{d-3}{2}}
e^{-Rs}\,\d s.
\]

The following examples illustrate this mechanism in dimensions three through five.

\bigskip
\noindent
\textbf{Example 1 (three dimensions). } For \(d=3\), 
\[
I^+=\i\int_0^k e^{\i R s}\d s=\frac{e^{\i k R}}{R}-\frac{1}{R}
\]
and
\[
I^-=\int_0^\infty e^{-R s}\d s=\frac{1}{R}.
\]
Therefore,
\[
I^++I^-=\frac{e^{\i k R}}{R}.
\]
The canceled endpoint term is of the same order $R^{-1}$
 as the outgoing radiation itself. Thus the leading behavior of the propagating part is not radiative; it is determined only after cancellation with the regular part.

\bigskip
\noindent
\textbf{Example 2 (four dimensions). } For \(d=4\),  the oscillatory integrals admit explicit
representations in terms of Bessel and Struve functions:
\begin{align*}
I^{+}&=\i\int_0^k\sqrt{k^2-s^2}e^{\i R s}\d s\\
&=\frac{\i k\pi}{2R}
\bigl(
i\mathbf{H}_1(kR)
+
J_1(kR)
\bigr)\\
&=\frac{\i k \pi}{2R}H_1^{(1)}(kR)-\frac{k\pi}{2R}\bigl(
\mathbf{H}_1(kR)
-
Y_1(kR)
\bigr)
\end{align*}
and
\[
I^{-}
=\int_0^k\sqrt{k^2+s^2}e^{-R s}\d s
=
\frac{k\pi}{2R}
\bigl(
\mathbf{H}_1(kR)
-
Y_1(kR)
\bigr),
\]
where \(J_1\), \(Y_1\), $H_1^{(1)}$ and \(\mathbf H_1\) denote the Bessel, Neumann, Hankel and Struve functions of order one, respectively.
Therefore,
\[
I^{+}+I^{-}
=
\frac{\i k\pi}{2R}
H_1^{(1)}(kR).
\]
Using the standard asymptotic expansions
\[
\mathbf H_1(z)-Y_1(z)
=
\frac{2}{\pi}
+O(z^{-2})
,
\]
and
\[
H_1^{(1)}(z)=\sqrt{\frac{2}{\pi z}}e^{\i(z-3\pi/4)}\left(1+O(z^{-1}\right)
\]
Therefore, 
\begin{align*}
I^+=\frac{\i k\pi}{2R}\sqrt{\frac{2}{k\pi R}}e^{\i (kR-3\pi/4)}\left(1+O(R^{-1})\right)-
\frac{k\pi}{2R}\left(\frac{2}{\pi}+O(R^{-2})\right)=-\frac{k}{R}+O(R^{-3/2})
\end{align*}
and
\[
I^{-}
=
\frac{k\pi}{2R}\left(\frac{2}{\pi}+O(R^{-2}\right)=\frac{k}{R}+O(R^{-2}).
\]
Therefore, the leading terms for both $I^+$ and $I^-$ are $O(R^{-1})$. However,\[
I^{+}+I^{-}
=\i R^{-3/2}\sqrt{\frac{k\pi}{2}}e^{\i (kR-3\pi/4)}+O(R^{-5/2})
\]
which agrees with the classical far-field asymptotics.
The canceled endpoint contribution is of order $R^{-1}$, which dominates the true radiation term of order $R^{-3/2}$.

\bigskip

\noindent
\textbf{Example 3 (five dimensions). } For \(d=5\), 
\[
(k^2-s^2)^{\frac{5-3}{2}}
=
k^2-s^2,
\]
direct computation gives
\[
I^+=\i\int_0^k (k^2-s^2) e^{\i R s}\d s
=
e^{\i kR}
\left(
-\frac{2\i k}{R^2}
+\frac{2}{R^3}
\right)
-
\left(
\frac{k^2}{R}
+\frac{2}{R^3}
\right),
\]
while
\[
I^-
=\int_0^k (k^2+s^2)e^{- R s}\d s=
\frac{k^2}{R}
+\frac{2}{R^3}.
\]
Consequently,
\[
I^++I^-
=
e^{\i kR}
\left(
-\frac{2\i k}{R^2}
+\frac{2}{R^3}
\right).
\]
Again, the leading $R^{-1}$  behavior of the propagating part is canceled completely, while the genuine radiation starts only at order $R^{-2}$.

These examples show that the cancellation observed in the classical propagating--regular decomposition is not accidental. It is a direct consequence of the degenerate grazing geometry of the Helmholtz equation, where the real and complex grazing branches contribute equally with opposite signs. In generic periodic media this symmetry is lost. The grazing manifold is non-degenerate, the cancellation disappears, and the grazing contribution survives as an independent asymptotic component.

From this perspective, the SSA decomposition separates propagating, grazing, and evanescent waves according to their geometric origin on the pole surface, whereas the classical propagating--regular decomposition combines contributions generated by different geometric mechanisms. The SSA decomposition therefore provides a more faithful geometric description of the far-field asymptotics. In this sense, the Helmholtz equation serves not merely as an example but as an explicit explanation of why the geometric decomposition introduced in this work is necessary.

\section*{Statements and Declarations}

\subsection*{Competing Interests}

The author has no competing interests to declare.

%

\appendix

\section*{Appendix}
\addcontentsline{toc}{section}{Appendix}
\numberwithin{equation}{section}

\section{Background on Floquet–Bloch Transform}
\label{sec:fbt}

For the reader's convenience, we briefly recall the definition and basic mapping properties of the Floquet--Bloch transform. Standard references include \cite{Reed1978,Kuchm1993}. We also state the extension to arbitrary dimensions following \cite{Lechl2016}.

Suppose $\phi\in C_0^\infty(\R^d)$ be a function, which is compactly supported and smooth in $\R^d$. Then the Floquet-Bloch transform of $\phi$ with respect to the periodicity cell $\Omega$ is defined as follows:
\begin{equation}
\label{eq:FB}
(\J\phi)(\balpha,\bx)=\sum_{\bj\in\Z^d}\phi(\bx+2\pi\bj)e^{-\i 2\pi\balpha\cdot\bj},\quad \balpha,\,\bx\in\R^d.
\end{equation}
Since $\phi$ is compactly supported, we notice that the transform is well-defined for all $\balpha$ and $\bx$ in $\R^d$. It is also easily checked that $\J\phi$ depends smoothly on $\bx$ and analytically on $\balpha$. Moreover, we also have the following further properties:
\begin{itemize}
\item For fixed $\bx\in\R^d$, the function is periodic with respect to $\balpha$ where $\Bz$ is its periodicity cell:
\[
(\J\phi)(\balpha+\bj,\bx)=(\J\phi)(\balpha,\bx);
\]
\item for fixed $\balpha\in\R^d$, the function is $\balpha$-quasi-periodic with respect to $\balpha$:
\[
(\J\phi)(\balpha,\bx+2\pi\bj)=e^{\i 2\pi\balpha\cdot\bj}(\J\phi)(\balpha,\bx).
\]
\end{itemize}
From these two results, it is sufficient for us to restrict $(\balpha,\bx)\in \Bz\times\Omega$.

\begin{remark}
We can also treat the Floquet-Bloch transform as a generalized Fourier series, where the Fourier coefficients are functions instead of scalars. With this in mind, one also gets the inverse Floquet-Bloch transform easily:
\begin{equation}
\label{eq:iFB}
(\J^{-1}\psi)(\bx+2\pi\bj)=\int_\Bz \psi(\balpha,\bx)e^{\i 2\pi\balpha\cdot\bj}\d\balpha,\quad \bx\in\Omega.
\end{equation}
\end{remark}

Now we will extend the Floquet-Bloch transform defined for compactly supported smooth functions to more general function spaces. To this end, we will introduce some Sobolev spaces. First define Bessel potential spaces by
\[
H^s(\R^d):=\left\{\phi\in\mathcal{D}'(\R^d):\,\int_{\R^d}(1+|\bz|^2)^s\left|\hat{\phi}(\bz)\right|\d \bz<\infty\right\},\quad s\in\R.
\]
In particular, when $s=0$, $H^s(\R^d)$ is the classic space $L^2(\R^d)$.
Then the weighted spaces are defined by
\[
H^s_r(\R^d):=\left\{\phi\in\mathcal{D}'(\R^d):\,(1+|\cdot|^2)^{r/2}\phi(\cdot)\in H^s(\R^d)\right\},\quad s,r\in\R.
\]
When we study the Floquet-Bloch transform for any function $\phi\in H^s_r(\R^d)$, we also need the following spaces which are defined in the domain $\Bz\times\Omega$. First define the subspace $H^s_\balpha(\Omega)$ of $\mathcal{D}'_\balpha(\R^d)$ which contains all $\balpha$-quasi-periodic distributions with the finite norm
\[
\|\phi\|_{H^s_\balpha(\Omega)}:=\left(\sum_{\bj\in\Z^d}(1+|\bj|^2)^s\left|\hat{\phi}(\bj)\right|^2\right)^{1/2},
\]
where $\hat{\phi}(\bj)$ is the $\bj$-th Fourier coefficient. When $s=0$, $H^s_\balpha(\Omega)$ is also denoted by $L^s(\Omega)$.

Then we define the Sobolev spaces $L^2(\Bz;H^s_\balpha(\Omega))$ of distributions $\mathcal{D}'(\R^d\times\R^d)$ with the norm:
\[
\|\psi\|_{L^2(\Bz;H^s_\balpha(\Omega))}:=\left(\sum_{\bj\in\Z^d}(1+|\bj|^2)^s\int_\Bz\left|\hat{\psi}(\balpha,\bj)\right|^2\d\balpha\right)^{1/2}<\infty,
\]
where $\hat{\psi}(\balpha,\bj)$ is the $\bj$-th Fourier coefficient of $\psi(\balpha,\cdot)$. When $s=0$,
\[
\|\psi\|_{L^2(\Bz;L^2_\balpha(\Omega))}=\left(\int_\Bz \|\psi(\balpha,\cdot)\|^2_{L^2(\Omega)}\d\alpha\right)^{1/2}=\left(\int_\Bz \int_\Omega |\psi(\balpha,\bx)|^2\,\d\bx\,\d\balpha\right)^{1/2}.
\]
We can also extend the definition of the norm of the space $H_0^\ell(\Bz;H^s_\balpha(\Omega))$ for any $\ell\in\N$:
\[
\|\psi\|_{H_0^\ell(\Bz;H^s_\balpha(\Omega))}:=\left[\sum_{\bgamma\in \N^d,\,|\bgamma|\leq\ell}\int_\Bz \left\|\partial^\bgamma_\balpha\psi(\balpha,\cdot)\right\|^2_{H^s_\balpha(\Omega)}\d\balpha\right]^{1/2},
\]
where $\bgamma=(\gamma_1,\dots,\gamma_d)\in\N^d$ and
\[
\partial^\bgamma_\balpha=\frac{\partial^{\gamma_1}}{\partial\alpha_1^{\gamma_1}}\cdots\frac{\partial^{\gamma_n}}{\partial\alpha_1^{\gamma_d}},\quad |\bgamma|=\gamma_1+\cdots+\gamma_d.
\]
Here the subscript $0$ indicates that the function is periodic with respect to $\balpha$ with the periodicity cell $\Bz$.

Then we can use the interpolation between spaces to extend the index $\ell$ to any positive number $r=\ell+\theta$  where $0<\theta<1$ (see \cite{Bergh1976}):
\[
H_0^{r}(\Bz;H^s_\balpha(\Omega)):=\left[H_0^{\ell}(\Bz;H^s_\balpha(\Omega)),H_0^{\ell+1}(\Bz;H^s_\balpha(\Omega))\right]_\theta;
\]
With a duality argument we can also define the space with negative index $H_0^{-r}(\Bz;H^s_\balpha(\Omega))$. Now we are prepared to introduce the properties of the Floquet-Bloch transform.

\begin{theorem}[Theorem 4, \cite{Lechl2016}]
(a) The Floquet-Bloch  transform $\J$ extends  to an isometric isomorphism between $L^2(\R^d)$ and $L^2(\Bz;L^2(\Omega))$, and the inverse is given by \eqref{eq:iFB}.\\
(b) For $s,r\in\R$, the  Floquet-Bloch transform $\J$ extends to an isomorphism between $H^s_r(\R^d)$ to $H_0^r(\Bz;H^s_\balpha(\Omega))$. The inverse transform is still given by \eqref{eq:iFB} with equality in the sense of the norm of $H_r^s(\R^d)$.
\end{theorem}

The following identities will be used repeatedly in the sequel:
\begin{itemize}
\item The Floquet-Bloch transform commutes with periodic functions with the same periodicity:
\[
[\J (p u)](\balpha,\bx)=p(\bx)(\J u)(\balpha,\bx);
\]
\item the Floquet-Bloch transform commutes with partial differential operators with respect to the space variable $\bx$:
\[
\left[\J\left(\frac{\partial u}{\partial x_j}\right)\right](\balpha,\bx)=\frac{\partial}{\partial x_j}(\J u)(\balpha,\bx),\quad j=1,2,\dots,d.
\]
\end{itemize}

\section{Technical lemmas}

\subsection{Local geometry near the Fermi surface}
\label{sec:local_expand}
This appendix derives the local quadratic normal form of the Fermi surface near a regular point. The expansion is the basic local model underlying the directional complex deformation and the grazing asymptotic analysis in the main text.

We fix $j\in J(\lambda)$ and omit the index $j$ for simplicity. 
Let $\balpha_0\in\bF_\lambda$, such that $\mu(\balpha_0)=\lambda$. Since $\nabla\mu(\balpha_0)\neq0$, the level set
$\bF_\lambda=\{ \balpha\in\Bz:
\mu(\balpha)=\lambda
\}$
is a smooth hypersurface near $
\balpha_0$.

Let $\e_1=\nu(\balpha_0)$ be the unit normal vector at $\balpha_0$. Choose an orthonormal basis
$(\e_2,\dots,\e_d)$ of the tangent space $T_{\balpha_0}\bF_\lambda=\e_1^\perp$
 such that the quadratic form $
D^2\mu(\balpha_0)
\big|_{T_{\balpha_0}\bF_\lambda}$ 
is diagonalized. Namely,
\[
D^2\mu(\balpha_0)[\e_k,\e_\ell]
=
\kappa_k(\balpha_0)\delta_{k\ell},
\qquad
2\le k,\ell\le d,
\]
where $
\kappa_2(\balpha_0),\dots,
\kappa_d(\balpha_0)$
are the principal curvatures of the Fermi surface at
$\balpha_0$ up to the normalization factor $v(\balpha_0)$.

We introduce local orthogonal coordinates $(x_1,{\bx}')
=(x_1,x_2,\dots,x_d)$ through
\[
\balpha
=
\balpha_0
+
x_1\e_1
+
\sum_{k=2}^d x_k\e_k.
\]
In these coordinates, the defining function of the Fermi surface admits the following local expansion.

\begin{proposition}[Local quadratic expansion]
\label{prop:local_taylor}

In the above coordinates,
$
\mu_j(\balpha)-\lambda$ 
admits the expansion
\[
\mu(\balpha)-\lambda
=
v(\balpha_0)x_1
+
\frac12
\sum_{k=2}^d
\kappa_k(\balpha_0)x_k^2
+\frac{a_{11}}{2}x_1^2+x_1 L(\bx')+o(|\bx|^2).
\]
\end{proposition}

\begin{proof}

Since $
\mu(\balpha_0)=\lambda$, 
Taylor expansion gives
\[
\mu(\balpha)-\lambda
=
\nabla\mu(\balpha_0)\cdot
(\balpha-\balpha_0)
+
\frac12
D^2\mu(\balpha_0)
[\balpha-\balpha_0,\balpha-\balpha_0]
+
O(|\bx|^3).
\]
Since $\nabla\mu(\balpha_0)=
v(\balpha_0)\e_1$, the linear term  
\[
\nabla\mu(\balpha_0)\cdot(\balpha-\balpha_0)=v(\balpha_0)x_1
\]

Next decompose the quadratic term into tangential,
mixed, and normal components. Recall that the Hessian matrix is diagonalized in the tangential hyperplane expanded by ${\e_2,\dots,\e_{d}}$, it has the form of
\[
D^2\mu(\balpha_0)=\left(
\begin{matrix}
a_{11} & a_{12} & \cdots & a_{1d}\\
a_{12} & \kappa_2 & & 0 \\
\vdots & &\ddots & \\
a_{1d} & 0 & & \kappa_d
\end{matrix}
\right)
\]
and all the components are real. Therefore,
\[
\frac12
D^2\mu(\balpha_0)
[\balpha-\balpha_0,\balpha-\balpha_0]
=
Q(\bx')
+\frac{1}{2}a_{11} x_1^2+x_1 L(\bx'),
\]
where
\[
L(\bx')=\sum_{k=2}^d a_{1k} x_k,\quad
Q(\bx')
=
\frac12
\sum_{k=2}^d
\kappa_k(\balpha_0)x_k^2.
\]
Combining the above expansions yields the result.
\end{proof}

\begin{corollary}[Directional complex expansion]
\label{cor:directional_expansion}

Let $\bomega=(\bomega_1,\bomega_2,\dots,\bomega_d)\in\mathbb S^{d-1}$.
Then, in the above coordinates,
\[
\mu(\balpha+s\e_1+\i t\bomega)-\lambda
=
A(\balpha,\bomega)t^2+B(\balpha)(s+\i \bomega_1 t)+C(\balpha) (s+\i \bomega_1 t)^2+(s+\i \bomega_1 t)t D(\balpha,\bomega)+o(s^2+t^2),
\]
where
\[
A(\balpha,\bomega)
=
-\frac12
\sum_{k=2}^d
\kappa_k(\balpha)\bomega_k^2,
\qquad
B(\balpha)
=v(\balpha)\]\[ C(\balpha)=\frac{1}{2}\frac{\partial^2\mu(\balpha)}{\partial x_1^2},\qquad D(\balpha,\bomega)=\sum_{k=2}^d \frac{\partial^2 \mu(\balpha)}{\partial x_1 x_k}\bomega_k.
\]
Moreover, the term $o(t^3)$ holds  uniform with respect to $\bomega\in\mathbb S^{d-1}$.
\end{corollary}

\begin{proof}
Substitute
\[
\balpha_0+s\e_1+\i t\v
\]
into the expansion of
Proposition~\ref{prop:local_taylor}. The uniformity of
the remainder follows from Taylor's theorem and the
compactness of $\mathbb S^{d-1}$.
\end{proof}

\begin{lemma}[Uniform nondegeneracy of the crossing polynomial]
\label{lem:crossing-non-degenerate}
Let the coefficients be defined as above and Assumption \ref{asp1}-\ref{asp3} hold. There exists
\(
\delta_0>0
\)
such that
\[
|A(\balpha,\v)-C(\balpha)\bomega_1^2|\ge c_0>0,\quad\forall\,\balpha\in\overline{U_{\balpha_0}},\,|\bomega_1|\le\delta_1,
\,\bomega\in\mathbb S^{d-1}.
\]
Moreover, there exist
\(
\epsilon_1,\delta_1>0
\)
such that
\[
|B(\balpha)+2C(\balpha)s+D(\balpha,\bomega)t|
\ge
\frac{v_{\min}}2,\qquad\forall\,
|s|\le\epsilon_1,
\,
0\le t\le\delta_2.
\]
\end{lemma}

\begin{proof}

By Lemma~\ref{lem:uniform-geometry} and Assumption \ref{asp2}
\[
|A(\balpha,\bomega)|
\ge
\kappa_{\min}|\bomega'|^2=\kappa_{\min}(1-\bomega_1^2),
\qquad
|C(\balpha)|
\le M.
\]
Consequently, choosing
\(
\delta_1>0
\)
sufficiently small yields 
\[
|A(\balpha,\bomega)-C(\balpha)\bomega_1^2|
\ge
\kappa_{\min}(1-\bomega_1^2)-M\bomega_1^2
\ge
\frac{\kappa_{\min}}2.
\]
This proves the first claim.

For the second claim, Lemma~\ref{lem:uniform-geometry} gives
\[
B(\balpha)=|\nabla\mu_j(\balpha)|\ge v_{\min}>0.
\]
Since
\[
|C(\balpha)|+|D(\balpha,\v)|
\le M,
\]
for sufficiently small $\epsilon_1,\delta_2>0$,
\[
|2C(\balpha)s+D(\balpha,\v)t|
\le
M(2|s|+t)
\le
\frac{v_{\min}}2,
\]
and therefore
\[
|B(\balpha)+2C(\balpha)s+D(\balpha,\v)t|
\ge
\frac{v_{\min}}2.
\]

\end{proof}
This lemma guarantees the uniform solvability of the local crossing equation used throughout the directional complex deformation arguments.

\subsection{Periodic boundary cancellation}
\label{sec:bdr}

This appendix establishes the periodic boundary identification underlying the slice-wise contour deformation. The main consequence is the cancellation of the boundary contributions arising from opposite faces of the Brillouin zone.

At the beginning of this section, we introduce the face identification map
\begin{equation}
\label{eq:trans}
T\balpha=(\alpha_1,\dots,-\alpha_j,\dots,\alpha_d),
\end{equation}
defined for $\balpha\in\Cz_j^\pm$. Then
\[
T(\Cz_j^\pm)=\Cz_j^\mp,
\qquad
T^2=\mathrm{Id}.
\]

For a fixed slice
\[
\ell_\bgamma=\{\bgamma\t+s\n:\,s\in\R\},
\]
let $\bgamma\t+\ell_l(\bgamma)\n$ and $\bgamma\t+\ell_r(\bgamma)\n$
denote the left and right endpoints of
$\ell_\bgamma\cap\Bz$, respectively.

The following lemma describes the periodic identification of these endpoints.

\begin{lemma}[Periodic identification of opposite faces]
\label{lm:sym}

Fix $\n$ and $\bgamma$, and let
\[
\balpha_l
=
\bgamma\t+\ell_l(\bgamma)\n
\]
be the left endpoint of
$\ell_\bgamma\cap\Bz$.
Then there exists
$\widetilde{\bgamma}\in D$
such that
\[
T\balpha_l
=
\widetilde{\bgamma}\t
+
\ell_r(\widetilde{\bgamma})\n.
\]
In other words, the map $T$ transforms left endpoints into right endpoints, and vice versa.

\end{lemma}

\begin{proof}

Since $\balpha_l\in\partial\Bz$, it lies on a boundary face
$\Cz_j^\sigma$. By the periodicity property in Theorem~\ref{th:kuch2},
$T\balpha_l\in\Cz_j^{-\sigma}$. Hence there exist
$\widetilde{\bgamma}\in D$
and
$s_0\in\R$
such that
\[
T\balpha_l
=
\widetilde{\bgamma}\t+s_0\n.
\]

Let $\nu$ denote the exterior unit normal vector of $\partial\Bz$.
Since $\alpha_l$ is the left endpoint of
$\ell_\bgamma\cap\Bz$,
$\nu\cdot\n\le0$ on $\Cz_j^\sigma$. The exterior normal vector on the opposite face $\Cz_j^{-\sigma}$ has the opposite orientation. Therefore
$\nu\cdot\n\ge0$ on $\Cz_j^{-\sigma}$. It follows that
$T\alpha_l$ is the right endpoint of
$\ell_{\widetilde{\bgamma}}\cap\Bz$.
Consequently, $s_0
=
\ell_r(\widetilde{\bgamma}),
$
which proves the stated implication. The converse follows immediately
from the identity $T^2=\mathrm{Id}$.

\end{proof}

Recall that the complex domain with fixed $\bgamma$: 
\[
\bR(\bgamma):=\big\{s+\i\delta:\,s\in(\ell_l(\bgamma),\ell_r(\bgamma)),\,\delta\in(0,\sigma(\bgamma\t+s\n))\big\}.
\]
Then we define the left and right boundaries as
\begin{eqnarray}
\label{eq:left_C}
C_\bgamma^l&:=\{\ell_l(\bgamma)+\i t:\,t\in[0,\sigma(\bgamma\t+\ell_l(\bgamma)\n)]\},\\\label{eq:right_C}
C_\bgamma^r&:=\{\ell_r(\bgamma)+\i t:\,t\in[0,\sigma(\bgamma\t+\ell_r(\bgamma)\n)]\}.
\end{eqnarray}
The following result expresses the cancellation of the boundary
contributions produced by the slice-wise contour deformation.

\begin{lemma}[Cancellation of periodic boundary integrals]
\label{lm:boundary_cancel}
Assume that $C_\bgamma^l$ and $C_\bgamma^r$ are defined by
\eqref{eq:left_C} and \eqref{eq:right_C}, respectively.
Let $F$ be a continuous $\Bz$-periodic function. Then
\begin{equation}
\label{eq:boundary_cancel}
\int_D
\int_{C_\bgamma^l}
F(\bgamma\t+s\n)\,\d s\,\d\bgamma
=
\int_D
\int_{C_\bgamma^r}
F(\bgamma\t+s\n)\,\d s\,\d\bgamma.
\end{equation}
\end{lemma}

\begin{proof}

By Lemma~\ref{lm:sym}, for every
$\bgamma\in D$
there exists a unique
$\widetilde{\bgamma}\in D$
such that
\[
T(C_\bgamma^l)
=
C_{\widetilde{\bgamma}}^r.
\]
Since both curves are parameterized by the same vertical variable
$t$, the periodic identification preserves the orientation of the
boundary integrals.

Moreover, the periodicity of $F$ implies that
\[
F(\bgamma\t+z\n)
=
F(\widetilde{\bgamma}\t+T(z)\n)
\]
for corresponding points
$z\in C_\bgamma^l$
and
$T(z)\in C_{\widetilde{\bgamma}}^r$.
Therefore
\[
\int_{C_\bgamma^l}
F(\bgamma\t+z\n)\,\d z
=
\int_{C_{\widetilde{\bgamma}}^r}
F(\widetilde{\bgamma}\t+z\n)\,\d z.
\]

Finally, the correspondence $\bgamma\longmapsto\widetilde{\bgamma}$
is a bijection of $D$. Integrating over $D$ and performing the change
of variables
$\widetilde{\bgamma}=\widetilde{\bgamma}(\bgamma)$
yields 
\[
\int_D
\int_{C_\bgamma^l}
F(\bgamma\t+z\n)\,\d z\,\d\bgamma
=
\int_D
\int_{C_\bgamma^r}
F(\bgamma\t+z\n)\,\d z\,\d\bgamma.
\]

The proof is complete.

\end{proof}

\subsection{Stationary phase on manifolds}

For completeness, we record the stationary phase theorem used in the proof of Section \ref{sec:asym_non_gra}. The result is standard and is included only for reference. We refer to Hörmander~\cite[Chapter~VII]{HormanderI} for the classical stationary phase theorem.

\begin{theorem}[Stationary phase]
Let M be a compact m-dimensional manifold and
\[
I(R)
=
\int_M
a(x)e^{iR\phi(x)}
\,\d\sigma(x).
\]
Assume that all critical points of $\phi$ are non-degenerate. Then
\[
I(R)
=
(2\pi R)^{-m/2}
\sum_j
\frac{
a(x_j)
}
{
\sqrt{
|\det D^2\phi(x_j)|
}
}
e^{iR\phi(x_j)}
e^{i\pi\sigma_j/4}
+
O(R^{-m/2-1}).
\]

\end{theorem}

\bibliographystyle{plain}
\bibliography{aa-biblio}

\end{document}